\documentclass[12pt]{article}

\usepackage{amsfonts,amssymb,amsmath,amsthm}
\usepackage{fullpage}
\usepackage{bm}
\usepackage{bbm}
\usepackage{dsfont}
\usepackage[pdftex]{graphicx}
\usepackage{xcolor}
\usepackage{epstopdf}
\usepackage[colorlinks,
			pdffitwindow=false,
      plainpages=false,
      pdfpagelabels=true,
     	pdfpagemode=UseOutlines,
      pdfpagelayout=SinglePage,
			bookmarks=false,
			colorlinks=true,
      hyperfootnotes=false,
			linkcolor=blue,
			citecolor=green!50!black]{hyperref}
\usepackage{tikz}
\usetikzlibrary{arrows,decorations.pathmorphing,backgrounds,fit,positioning,shapes.symbols,chains}

\usepackage{enumerate}
\usepackage{titlesec}
\usepackage{floatrow}
\usepackage{sidecap}
\usepackage{calc}
\usepackage{ifthen}
\usepackage{fancyvrb}
\usepackage{tcolorbox}

\graphicspath{{./graphics/}}

\newtheoremstyle{custom}{3pt}{3pt}{}{}{\bfseries}{:}{.5em}{}
\theoremstyle{custom}
\newtheorem{example}    {Example}
\newtheorem{definition} [example]{Definition}
\newtheorem{theorem}    [example]{Theorem}
\newtheorem{lemma}      [example]{Lemma}
\newtheorem{corollary}  [example]{Corollary}

\newtheorem{remark} 		[example]{Remark}
\newtheorem{proposition}[example]{Proposition}
\newtheorem{assumption}[example]{Assumption}

\newcommand{\R}{\mathbb{R}}
\newcommand{\C}{\mathbb{C}}
\newcommand{\N}{\mathbb{N}}
\newcommand{\Z}{\mathbb{Z}}

\newcommand{\IG}{\mathcal{G}}
\newcommand{\augIG}{\bm{\mathcal{G}}}
\renewcommand{\P}{\mathcal{P}}

\newcommand{\aug}[1]{\bm{#1}}
\newcommand{\augX}{\bm{X}}

\def \prob {\mathbb{P}}

\newcommand{\ep}{\varepsilon}

\title{Estimating long-term behavior of periodically driven flows without trajectory integration}
\author{Gary~Froyland\thanks{School of Mathematics and Statistics, University of New South Wales, Sydney NSW 2052, Australia. E-mail: g.froyland{@}unsw.edu.au}
\and P\'eter~Koltai\thanks{Institute of Mathematics, Freie Universit\"at Berlin, 14195 Berlin, Germany. E-mail: peter.koltai{@}fu-berlin.de} }
\date{}

\begin{document}

\maketitle

\begin{abstract}
Periodically driven flows are fundamental models of chaotic behavior and the study of their transport properties is an active area of research.
A well-known analytic construction is the augmentation of phase space with an additional time dimension;  in this augmented space, the flow becomes autonomous or time-independent.
We prove several results concerning the connections between the original time-periodic representation and the time-extended representation, focusing on transport properties.
In the deterministic setting, these include single-period outflows and time-asymptotic escape rates from time-parameterized families of sets.
We also consider  stochastic differential equations with time-periodic advection term.
In this stochastic setting one has a time-periodic \emph{generator} (the differential operator given by the right-hand-side of the corresponding time-periodic Fokker-Planck equation).
We define in a natural way an autonomous generator corresponding to the flow on time-extended phase space.
We prove relationships between these two generator representations and use these to quantify decay rates of observables and to determine time-periodic families of sets with slow escape rate.
Finally, we use the generator on the time-extended phase space to create efficient numerical schemes to implement the various theoretical constructions.
These ideas build on the work of~\cite{FrJuKo13}, and no expensive time integration is required.
We introduce an efficient new hybrid approach, which treats the space and time dimensions separately.
\end{abstract}


\section{Introduction}

Periodically driven flows are fundamental testing grounds for many questions in nonlinear dynamics.
Part of their attraction is their ability to create complicated dynamics even in two-dimensional phase space, and their suitability as models of many phenomena with periodic driving, due to biological, environmental, geophysical, or engineered cycles.
The quantification of transport in dynamical systems has been intensely studied over recent decades \cite{mackay1984transport,romkedar_etal_1990,wiggins_92,haller1998finite,aref_02,jones2002invariant,wiggins2005dynamical,shadden2005definition,froyland_padberg_09,FrPa14}, from both geometric and probabilistic points of views.
For certain periodically driven area-preserving flows \cite{romkedar_etal_1990} and two-dimensional maps \cite{romkedar_wiggins_90}, the theory of lobe dynamics has helped explain how unions of pieces of stable and unstable manifolds of low period hyperbolic periodic points delineate ``lobes'' of fluid which are transported in a way organised by the manifolds.
Our contributions in the present paper also concern the quantification of transport and span both geometric and probabilistic ideas.

We consider time-periodic flows on a compact state space $X\subset\mathbb{R}^d$ of the form
\begin{equation}
\label{basicode}
\dot{x}(t)=v(t,x(t)),
\end{equation}
where $v:\tau S^1\times X\to \mathbb{R}^d$ is a continuous vector field, and $S^1$ denotes the circle of unit circumference.
One can rewrite (\ref{basicode}) as an autonomous flow at the expense of adding an extra (time) dimension:
\begin{eqnarray}
\label{extendode1}
\dot{\theta}(t)&=&1,\\
\label{extendode2}
\dot{x}(t)&=&v(\theta(t),x(t))\,.
\end{eqnarray}

\medskip

\emph{The central thread running through this paper is formalizing the connections between these two representations for various transport-related questions and exploiting these connections to compute several transport-related quantities, both theoretically and numerically.}

\medskip

First, we consider a fundamental transport question, namely, the computation of time-integrated flux through the boundaries of a time-periodic family of sets $A_t\subset X$, $t\in\tau S^1$.
We show that this problem has a natural description in the time-extended domain and can be computed as an \emph{instantaneous} flux through the time-extended set $\cup_{t\in\tau S^1}\{t\}\times A_t$ (Theorem~\ref{thm:1=2}). Our construction has a natural extension to the case of aperiodic velocity fields and aperiodic families of sets~$A_t$.

Second, we turn to the time-asymptotic transport question of defining the rate of escape from a periodic collection of sets $\{A_t\}_{t\in \tau S^1}$.
One can view this as considering an open dynamical system in a time-periodic flow, with an open domain~$A_t\subset X$ that can itself be time-periodic.
Trajectories vanish once they leave a set~$A_t$ at a time~$t+k\tau, k\in \mathbb{Z}^+$.
In discrete-time autonomous dynamical systems, $\Phi:X\to X$, one studies \emph{escape rates} from a fixed set $A\subset X$:  $E(A)=\limsup_{n\to\infty} (1/n)\log m\left(\bigcap_{i=0}^{n-1}\Phi^{-i}A\right)$, where~$m$ is normalised Lebesgue measure on~$X$.
The escape rate~$E(A)$ measures the exponential rate of decay of the measure of the set of points~$x$ which ``survive'' for~$n$ steps ($x\in A, \Phi x\in A,\ldots, \Phi^{n-1}x\in A$).
We refer the reader to the excellent survey~\cite{DeYo06} for further background on escape rates in a discrete-time autonomous setting.
Here, we generalise these ideas in two ways:  firstly to continuous time, and secondly to periodic time-dependence.
Under mild assumptions on the measurable and topological properties of this collection, we show that (i) the escape rate is not a time-dependent quantity (Theorem~\ref{thm:escrate_indep})  and (ii) the escape rate in the original phase space~$X$ is identical to the escape rate of the autonomous flow in time-extended space (Theorem~\ref{thm:escrate_nonauto_aug}).

The notion of escape rates can be generalized to stochastic differential equations (SDEs) of the form
\begin{equation}
\label{eq:sdenonauto}
dx_t=v(t,x_t)dt+\ep dw_t,
\end{equation}
where~$v:[0,\infty)\times X\to \mathbb{R}^d$ is a continuously differentiable, time-dependent vector field,~$\{w_t\}_{t\ge 0}$ is a standard Wiener process, and~$\ep \ge 0$ is a parameter describing the standard deviation of the Wiener process. We define a probabilistic description of escape rate and in Proposition~\ref{prop:escapeduality} we show that, for periodically driven SDEs, the escape rate computed in the original phase space~$X$ has the same value as the escape rate of the associated autonomous SDE in time-extended space.

Connections between the Perron--Frobenius operator and escape rates for time-homo\-geneous diffusion processes have been considered in~\cite{SchHM03,WilhelmHuisingaSeanMeyn2004}.
In the discrete-time deterministic setting, it was established in~\cite{FrSt10} that eigenfunctions of the Perron--Frobenius operator~$\mathcal{P}:L^1(X,m)\circlearrowleft$ corresponding to eigenvalues close to 1, so-called \emph{strange eigenmodes}, can be used to create subsets~$A\subset X$ with slow escape rate.
More precisely, if~$\mathcal{P}f=\lambda f$, $0<\lambda<1$, then the sets~$A^+=\{f\ge 0\}, A^-=\{f<0\}$ both have escape rates slower than~$\log \lambda$. These statements are generalized in~\cite{FrSt13} to random discrete-time systems and in~\cite{FrJuKo13} to discrete-time systems with i.i.d.\ noise.
In Theorem~\ref{thm:cont_time_escrate} we first extend this result to the continuous time setting for SDEs of the form:
\begin{equation}
\label{sde}
dx_t=v(x_t)dt+\ep dw_t\,.
\end{equation}


In Theorem~\ref{thm:nonauto_escrate} we then further generalize the result to general non-homogeneous SDEs of the form~\eqref{eq:sdenonauto}, where~$v:[0,\infty)\times X\to \mathbb{R}^d$ need not be periodic in $t$.
In this latter generalization, one chooses an essentially bounded function $f:X\to \mathbb{R}$ so that
\[
\limsup_{t\to\infty}\frac{1}{t}\log\|\P_{s,t}f\|_{L^1}=\lambda<0\,.
\]
Here,~$\P_{s,t}$ denotes the Perron--Frobenius operator (also called the transfer operator) that pushes forward densities from time~$s$ to time~$t\ge s$.
Under mild topological assumptions on the time-dependent family of sets~$A_t:=\{\P_{s,t}f\ge 0\}\subset X$, we show that the escape from this family is slower than the Lyapunov exponent (or decay rate)~$\lambda$.
Returning to the periodic driving setting, where~$v(t,\cdot)=v(t+k\tau,\cdot)$,~$k\in\mathbb{Z}^+$, we demonstrate that the family of functions that produce the slowest decay rate are the periodic family~$\{f_s\}_{s\in\tau S^1}$ corresponding to the second largest eigenvalue of~$\P_{s,s+\tau}$ (Proposition~\ref{prop:CohPoiB}).

In the final theoretical section, we consider time-extended SDEs of the form:
\begin{eqnarray}
\label{sdetext}
d\theta_t&=&dt,\\
dx_t&=&v(\theta_t,x_t)dt+\ep dw_t,
\end{eqnarray}
where~$v:\tau S^1\times X\to \mathbb{R}^d$.
We construct a time-extended generator~$\augIG$ for this time-extended SDE and identify the time fibres of eigenfunctions of~$\augIG$ as equivariant functions under~$\mathcal{P}_{s,s+t}$ (Lemma~\ref{lem:spectral_con}).
Theorem~\ref{thm:IGcohPeriodic} then states that the slowest decaying family of functions~$\{f_t\}_{t\in\tau S^1}$ can be obtained as time fibres of the eigenfunction~$\aug{f}$ corresponding to the eigenvalue of~$\augIG$ with second largest real part.
We discuss how to interpret complex eigenvalues and eigenfunctions of~$\augIG$ (they describe periodic motions different to the driving period), and other properties of the point spectrum of~$\augIG$, including how numerical methods distort the spectrum.
To give the reader an overview, we partially summarize our findings graphically in Figure~\ref{fig:connections}.

In the numerical section we describe in detail two numerical methods for approximating~$\augIG$.
First, an extension of the Ulam Galerkin approach from \cite{FrJuKo13}, and second, a new, efficient hybrid approach, which uses Fourier collocation in the time direction to take advantage of the special dynamics in that coordinate.
We apply the full Ulam numerics to analyse an SDE version of the periodically driven double gyre flow~\cite{FrPa14}, finding that the largest transport barrier is a smoothed analogue of the unstable manifold of a hyperbolic periodic point on the boundary of the domain.
Boundaries of well-known regions of elliptic dynamics are determined as the sets with the next smallest escape rate.
We numerically check Theorem \ref{thm:nonauto_escrate} by explicitly computing escape rates using simulated trajectories and verify that the escape rates are indeed slower than indicated by the corresponding eigenvalue of~$\augIG$.
Similar computations are performed for the Bickley jet \cite{RypEtAl07} using the hybrid approach to approximating~$\augIG$.
We determine transport barriers and vortices in regions similar to previous work~\cite{RypEtAl07,FrSaMo10,BVetAl10,HaBV12}, but we can do this \emph{without any expensive trajectory integration} and additionally obtain numerical estimates of transport rates such as escape rates and decay rates.
Furthermore, by utilizing information carried in complex eigenvalues, we are able to detect previously unknown slowly decaying structures with \emph{periods different to the driving period}; building upon similar observations~\cite{SinEtAl2009,FrGTQu14} for discrete-time autonomous dynamics.

\section{Accumulated outflow from a time-periodic family of sets} \label{sec:accumoutflow}

Suppose we have a compact state space~$X\subset\R^d$ and a time-dependent, continuous vector field~$v:\R\times X\to\R^d$.
We assume that the time-dependence is periodic with period~$\tau$, i.e.~$v:\tau S^1\times X\to\R^d$.
Let a family $\{A_t\}_{t\in \tau S^1}$ of subsets of~$X$ be given, such that
\begin{assumption} \label{ass:boundary}
\quad
\begin{enumerate}[(a)]
\item for every $t\in\tau S^1$, the boundary~$\partial A_t$ of~$A_t$ is parametrized by the piecewise smooth function~$a(t,\cdot): R\to\R^d$ over some parameter domain~$R\subset\R^{d-1}$; and
\item for every $t\in\tau S^1$ and $x\in\partial A_t$, there is a unique~$r\in R$ with $a(t,r)=x$, and the function $w(t,x) = \tfrac{\partial}{\partial t}a(t,r)$ is well-defined. We  think of~$w(t,x)$ as the instantaneous velocity of the boundary of~$A_t$ at~$x\in A_t$.
\end{enumerate}
\end{assumption}


The \textit{cumulative outflow flux} from the family of sets over a period of the vector field is then given by
\begin{equation}	\label{eq:cumoutflow}
\int_{\tau S^1}\int_{\partial A_t}\left\langle v(t,x)-w(t,x),n_t(x)\right\rangle^+ d\sigma(x)dt,
\end{equation}
where $\langle \cdot,\cdot\rangle$ denotes the Euclidean inner product, $n_t(x)$ denotes the unit outer normal on~$\partial A_t$ at~$x$, the superscript~$+$ denotes the positive part, i.e.~$f^+:=\max\{0,f\}$, and~$d\sigma(\cdot)$ denotes integration with respect to co-dimension 1 volume on~$X$.
It can clearly be seen that the cumulative outflow is zero if the velocities of the boundaries of the family of sets~$A_t$ match the vector field~$v$, i.e.\ if~$v\equiv w$.

Let us now consider the \textit{augmented} state space $\augX = \tau S^1\times X$ with elements~$(\theta,x) = \aug{x}\in\aug{X}$, and the augmented ODE
\begin{eqnarray}
\label{eq:augode}
\dot \theta(t) & = & 1\\
\dot x(t) & = & v(\theta(t), x(t))
\end{eqnarray}
We define the augmented vector field $\aug{v}:\augX\to\R^{d+1}$ by $\aug{v}(\theta,x) = (1,v(\theta,x)^{\rm T})^{\rm T}$, and denote its flow map by $\aug{\phi}^t:\tau S^1\times X\to \tau S^1\times X$.
Note that this dynamical system is autonomous in the extended phase space.
Using the periodic family $\{A_{\theta}\}_{\theta\in\tau S^1}$ we define the \emph{augmented set}
\begin{equation}
\aug{A} = \bigcup_{\theta=0}^\tau \{\theta\}\times A_\theta,
\label{eq:augset}
\end{equation}
Considering the autonomous system defined by $\dot{\aug{x}} = \aug{v}(\aug{x})$ acting on the augmented state space, we can define the \textit{instantaneous outflow flux} from the set~$\aug{A}$ by
\begin{equation}	\label{eq:instoutflow}
\int_{\partial \aug{A}}\langle \aug{v}(\aug{x}),\aug{n}(\aug{x})\rangle^+ d\aug{\sigma}(\aug{x}),
\end{equation}
where~$\aug{n}(\aug{x})$ is the unit outer normal on~$\partial\aug{A}$ at~$\aug{x}$ and~$d\aug{\sigma}(\cdot)$ denotes integration with respect to co-dimension 1 volume on~$\aug{X}$. Note that conditions (a) and (b) from above imply that~$\aug{n}$ exists and is well-defined almost everywhere on~$\partial\aug{A}$.

\begin{theorem}
\label{thm:1=2}
For a continuous time-periodic vector field and a time-periodic family $\{A_{\theta}\}_{\theta\in\tau S^1}$ of sets satisfying the conditions of Assumption~\ref{ass:boundary}, the cumulative outflow flux and the instantaneous outflow flux in augmented space are equal:
\begin{equation}
\int_{\tau S^1}\int_{\partial A_{\theta}}\left\langle v(\theta,x)-w(\theta,x),n_{\theta}(x)\right\rangle^+ d\sigma(x)d\theta = \int_{\partial \aug{A}}\langle \aug{v}(\aug{x}),\aug{n}(\aug{x})\rangle^+ d\aug{\sigma}(\aug{x})\,.
\label{eq:1=2}
\end{equation}
Both sides of equation~\eqref{eq:1=2} are independent of the chosen parametrization~$a$ in Assumption~\ref{ass:boundary}.
\end{theorem}
\begin{proof}
See Appendix~\ref{app:proof:thm:1=2}.
\end{proof}
\begin{figure}[h]

\centering
\includegraphics[scale=1]{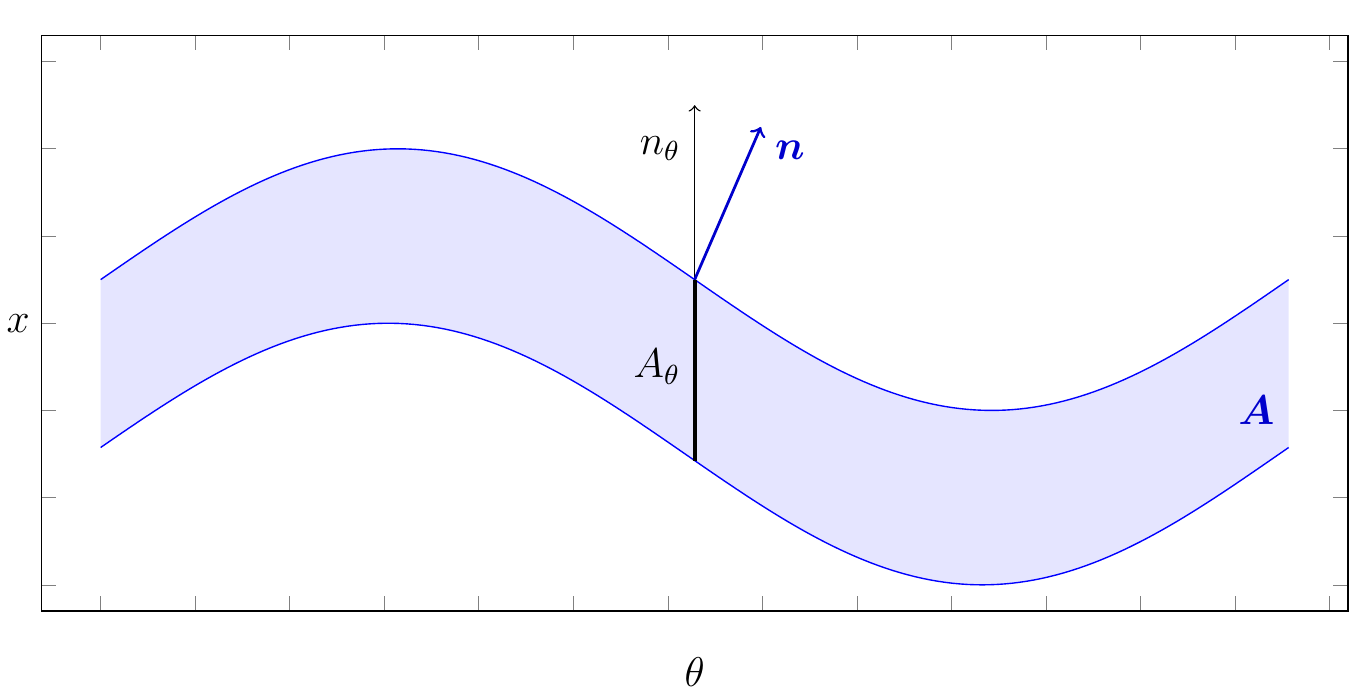}

\caption{Graphical representation of some objects from Theorem~\ref{thm:1=2}. The thick black segment is one specific~$A_{\theta}$ with a unit normal vector~$n_{\theta}$. The ``trace'' of all sets~$A_{\theta}$ in~$\aug{X}$ is the augmented set~$\aug{A}$, shown by the blue shaded region, and having unit outer normal~$\aug{n}$.}
\label{fig:augoutflow}
\end{figure}

\begin{remark} \label{rem:nonper_outflow}
An analogous statement to that of Theorem~\ref{thm:1=2} can be shown to also hold in the case when time and the family of sets is not periodic, but evolves in a finite interval, $t\in[t_0,t_1]$. Then, however, one has to integrate over the ``spatial boundary'' of the augmented set~$\aug{A}$, i.e.~$\bigcup_{t\in[t_0,t_1]}\{t\}\times \partial A_t$, the boundary in the~$x$ coordinate direction.
\end{remark}
This computation is related to a recent calculation~\cite{Kar16} of Lagrangian flux through a surface, as opposed to Eulerian flux considered here.

\section{Time-asymptotic escape rates from a time-periodic family of sets}


The purpose of this section is to introduce a description of coherence for deterministic dynamics for long flow times that admits an immediate generalization to stochastically perturbed systems.
We then state a result analogous to Theorem~\ref{thm:1=2} for this alternative coherence descriptor.

\subsection{The survivor set}	\label{ssec:survival}

Define~$\phi_{t_1,t_2}$ to be the flow map of the ODE $\dot x(t) = v(t,x(t))$ from time~$t_1$ to~$t_2\ge t_1$, with the convention $\phi_{t_1,t_2} = \left(\phi_{t_2,t_1}\right)^{-1}$ if~$t_2<t_1$. This latter case refers to integration backward in time.

Let some family of Borel measurable sets $\{A_{\theta}\}_{\theta\in\tau S^1}$ be given. For practical reasons, we will assume that the family of sets~$\{A_{\theta}\}$ is ``nice''.

\begin{definition} \label{def:niceness}
A family of measurable sets $\{A_{\theta}\}_{\theta\in\tau S^1}$ is called \emph{nice}, if
\begin{enumerate}[(a)]
\item $A_{\theta}$ is closed and has nonempty interior for every~$\theta\in \tau S^1$.
\item The temporal evolution of the $A_{\theta}$ is continuous, i.e.\ $h(A_{\theta},A_{\theta+\sigma}) \to 0$ as $\sigma\to 0$ for all~$\theta\in\tau S^1$. Here $h(\cdot,\cdot)$ refers to the Hausdorff-distance of two sets.
\end{enumerate}
\end{definition}

While in the previous section we have been looking at what \emph{leaves} a family of sets, here we will focus on what \emph{stays} inside, as measured by Lebesgue measure $m$. For this, we will allow arbitrary large times (thinking of the system running over many periods), hence we identify~$\tau S^1$ with~$[0,\tau)$ and for any~$t\in\R$ we identify~$A_t$ with~$A_{t \mod \tau}$. Let $A_{s,t}:=\bigcap_{r=s}^t \phi_{r,s}A_r=\left\{x\in X\,\vert\, \phi_{s,r}x\in A_r\,\forall r\in[s,t]\right\}$ denote the \emph{$t$-survivor set at anchor time~$s$}, i.e.\ the set of points at time~$s$ that stay in the family~$\{A_r\}$ until time~$t$. We will refer to it simply as ``survivor set'' if anchor time and duration are clear. Note that the intersection is taken over pre-images of sets under the flow.

\begin{proposition}	\label{prop:nice=measurable}
Let $\{A_t\}$ be a nice family of sets. Then
\begin{enumerate}
\item $A_{s,t}$ is measurable; and
\item for any sequence~$(r_i)_{i\in\N}\subset\R$ which is dense\footnote{A sequence $(r_i)\subset S$ is called dense in a set $S$, if for any $s\in S$ there is a subsequence~$r_{i_j}\to s$ as~$i_j\to \infty$.} in the interval~$(s,t)$ one has \linebreak[4]
$\lim_{n\to\infty} m\left(\bigcap_{i=1}^n\phi_{r_i,s}A_{r_i}\right) = m(A_{s,t})$.
\end{enumerate}
\end{proposition}

\begin{proof}
See Appendix~\ref{app:nice=measurable}.
\end{proof}

For \emph{volume-preserving} flows there is a simple relation between the measure of the survivor set and the outflow flux. Since the proof of this is cumbersome and does not further contribute to the exposition, we only give the main idea, and motivate this idea by an example. Note that a family of closed sets with non-empty interior satisfying Assumption~\ref{ass:boundary} of the previous section is also nice.

\begin{proposition} \label{prop:survivalVSoutflow}
Let $\{A_{\theta}\}$ be a family of closed sets with non-empty interior satisfying Assumption~\ref{ass:boundary}. Further, let~$v(t,\cdot)$ be divergence-free for all times~$t$. Then, the volume of the surviving set is monotonically decreasing in~$t$ for fixed~$s$ and satisfies at its points of differentiability
\begin{equation}
\label{eq:survivalVSoutflow}
0 \ge \frac{d}{dt}m(A_{s,t}) \ge -\int_{\partial A_t}\left\langle v(t,x)-w(t,x),n_t(x)\right\rangle^+ d\sigma(x)\,.
\end{equation}
\end{proposition}
\begin{proof}[Idea of the proof.]
Proposition~\ref{prop:survivalVSoutflow} essentially states that the rate at which the survivor set loses measure is at most the outflow flux. To see this, note that~$A_{s,t}$ will decrease in measure if $\phi_{s,r}A_{s,t} \nsubseteq A_r$ for some~$r>t>s$. Considering times only infinitesimally greater than~$t$, this means that~$\phi_{s,t}A_{s,t}$ and~$A_t$ have to share a part of their boundary where the integrand in~\eqref{eq:survivalVSoutflow} is positive, for~$m(A_{s,t})$ to decrease. Since the right-hand side of~\eqref{eq:survivalVSoutflow} does not take this into account (it integrates over the whole boundary of~$A_t$ not just over the common part with~$\phi_{s,t}A_{s,t}$), it overestimates the decay of the volume of the survivor set. The second inequality in~\eqref{eq:survivalVSoutflow} becomes an equality if every~$x\in\partial A_t$ where~$\left\langle v(t,x)-w(t,x),n_t(x)\right\rangle \ge 0$ is also a boundary point of~$\phi_{s,t}A_{s,t}$.
\end{proof}

\begin{example} \label{ex:outflow}
Let $X = S^1$, $v(t,x) = 0.3\cos(t)$, and $A_t = [0.3+0.2\sin(t),0.7+0.2\sin(t)]\subset X$. Then, for $t\in[0,\pi/2]$ trajectories leave~$A_t$ at the upper boundary, and $A_{0,t} = [0.3,0.7-0.1\sin(t)]$. For~$t\in[\pi/2,\pi]$, there is no loss in the measure of the survivor set~$A_{0,t}$, because~$\phi_{0,t}A_{0,\pi/2} \subset A_t$. For~$t>\pi$, trajectories which initially started in~$A_0$ may leave~$A_t$ on its lower boundary. The survivor sets and their images are shown in Figure~\ref{fig:example1}. Note that for~$\pi/2<t<\pi$ no trajectory that initially started in~$A_0$ leaves~$\{A_r\}$, because for these times the blue and grey regions do not share a common boundary. Nevertheless, the integral in~\eqref{eq:survivalVSoutflow} is positive, hence overestimates the decay of the survivor set, showing a strict inequality.

\begin{figure}[!htb]
\centering
\includegraphics[scale=1]{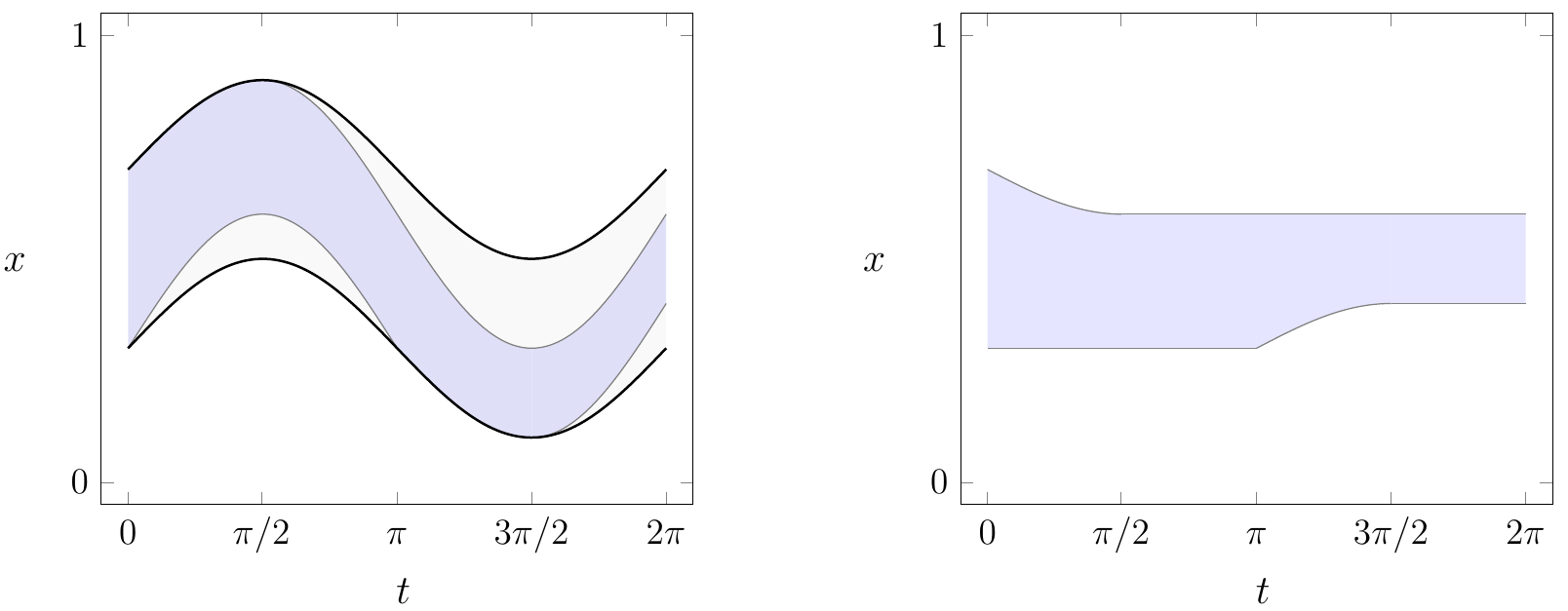}

\caption{Graphical representation of Example~\ref{ex:outflow}. Left: the thick black lines enclose the sets~$A_t$ (light gray shaded region), and the blue shaded region in it is the image of the survivor set~$A_{0,t}$, i.e.~$\phi_{0,t}A_{0,t}$, for times~$0$ to~$2\pi$. Right: the shaded region shows the survivor set~$A_{0,t}$ versus~$t$.}
\label{fig:example1}
\end{figure}

\end{example}

\subsection{Escape rates} \label{ssec:escape}

Since we think of the system running for large times (many periods), we are going to look at time-asymptotic quantities, as~$t\to\infty$.

Recall Proposition~\ref{prop:nice=measurable}, stating that the survivor set is measurable. Thus, we can make the following definition for escape rate in continuous time.
\begin{definition}
\label{def:lowerescaperate}
The \emph{(lower) escape rate} for a nice family of sets $\{A_r\}_{r\ge s}$, starting at some time~$s$, is given by
\begin{equation}
E(\{A_r\}_{r\ge s}) = -\limsup_{t\to\infty} \frac{1}{t-s} \log m\left(\bigcap_{r=s}^t\phi_{r,s}A_r\right)= -\limsup_{t\to\infty} \frac{1}{t-s} \log m\left(A_{s,t}\right).
\label{eq:escrate}
\end{equation}
\end{definition}

Using the augmented flow~$\aug{\phi}$ and augmented set~$\aug{A}$ defined by~\eqref{eq:augset}, we may also consider the \emph{escape rate} from this set, given by
\begin{equation}
\aug{E}(\aug{A}) = -\limsup_{t\to\infty}\frac{1}{t} \log \aug{m} \left(\bigcap_{r=0}^t\aug{\phi}_{-r}\aug{A}\right)\,,
\label{eq:augescrate}
\end{equation}
where~$\aug{m}$ denotes the Lebesgue measure on~$\aug{X}$.
Our goal is to show that \eqref{eq:escrate}=\eqref{eq:augescrate}.
Before stating a theorem in this direction, we interpret~$\bigcap_{r=0}^t\aug{\phi}_{-r}\aug{A}$ in terms of the time-fibred representation:
\begin{eqnarray}
\nonumber\bigcap_{r=0}^t \aug{\phi}_{-r}\aug{A} = \bigcap_{r=0}^t \aug{\phi}_{-r} \left(\bigcup_{\theta=0}^\tau\{\theta\}\times A_{\theta}\right)&=&
\bigcap_{r=0}^t  \bigcup_{\theta=0}^\tau\left(\{\theta-r\}\times{\phi}_{\theta,\theta-r} A_{\theta}\right)\\
&=&\nonumber\bigcup_{\theta'=0}^\tau \{\theta'\}\times \bigcap_{\substack{r=0\\ \theta=\theta'+r}}^t\phi_{\theta,\theta-r}A_{\theta}\\
&\stackrel{r'=\theta'+r}{=}&\nonumber\bigcup_{\theta'=0}^\tau \{\theta'\}\times \bigcap_{r'=\theta'}^{t+\theta'}\phi_{r',\theta'}A_{r'}\\
&=& \bigcup_{\theta'=0}^\tau \{\theta'\}\times A_{\theta',\theta'+t}\,,
\label{eq:remainingset}
\end{eqnarray}
where the second line is obtained by noting that intersections of the augmented sets are non-empty only if $\theta-r = \theta' = \text{const}$. Thus, the set $\bigcap_{r=0}^t\aug{\phi}_{-r}\aug{A}$ in augmented space is the natural augmentation (via a time-union) of the fibre-wise survivor sets~$A_{\theta',\theta'+t}$.


\begin{remark} \label{rem:Poincare}
It is natural to ask whether the continuous time escape rates we consider here are also connected to escape rates for the Poincar\'e return map of the dynamics.
The latter would be defined by a discrete-time version of~\eqref{eq:escrate}, namely:
 \begin{equation}
\label{eq:discrete_escrate}
-\limsup_{n\to\infty} \frac{1}{n\tau}\log m\left(\bigcap_{k=0}^{n-1} \phi_{s+k\tau,s}A_{s+k\tau}\right).
\end{equation}
Clearly, because the set formed by the discrete intersection in~\eqref{eq:discrete_escrate} contains the set formed by the continuous intersection in~\eqref{eq:escrate}, the escape rate for the Poincar\'e return map will be slower than for the continuous flow.
See also~\cite[section 3.3]{FrJuKo13} for a discussion of escape rates of flow maps in both the deterministic and stochastic cases.
\end{remark}

Our first result regarding escape rates is that $E(\{A_r\}_{r\ge s})$ is independent of~$s$.

\begin{theorem}\label{thm:escrate_indep}
For every $s_1,s_2\in \tau S^1$ one has
\[
E(\{A_r\}_{r\ge s_1}) = E(\{A_r\}_{r\ge s_2})\,.
\]
\end{theorem}
\begin{proof}
See Appendix~\ref{app:indeptime}.
\end{proof}
To express the independence of the starting time we write from now on $E(\{A_r\})$ for the escape rate from the family $\{A_r\}$ of sets.
Our second result on escape rates is that the two definitions~\eqref{eq:escrate} and~\eqref{eq:augescrate} are equivalent.
\begin{theorem}\label{thm:escrate_nonauto_aug}
If~$\{A_r\}\subset X$ is a nice family of sets and~$v:\tau S^1\times X\to \mathbb{R}^d$ is a continuous, time-periodic vector field, then
\[
E\left(\{A_r\}\right) = \aug{E}\left(\aug{A}\right)\,,
\]
\end{theorem}
\begin{proof}
See Appendix~\ref{app:augrate}.
\end{proof}

\section{Time-asymptotic escape rates for SDEs from a time-periodic family of sets }
\label{sec:cohfam}

Many processes in nature are modeled through dynamics perturbed by random noise and
we wish to extend our results to this case as well.
Unfortunately, the mathematical framework presented in Theorem~\ref{thm:1=2} is not suitable for handling this situation, because the outflow flux induced by the stochastic forcing is infinite. To see this, note that the mean displacement of a standard Brownian motion in a time interval of length~$dt$ is~$\sqrt{dt}$. The volume of the portion of state space that can leave~$A_t$ in this time interval is thus potentially of order~$\sqrt{dt}$, giving an instantaneous flux of order~$(dt)^{-1/2}$, which diverges as~$dt\to 0$. 
On the other hand, escape rates can easily be generalized to cope with the stochastic situation.


To facilitate presentation, we assume from now on that the vector field at hand is divergence-free, thus the (deterministic) flow~$\phi_{s,t}$ preserves the Lebesgue measure, meaning that $m\circ\phi_{t,s} = m$ for every $t\ge s$.
We will need the following regularity assumptions:
\begin{assumption} \label{ass:regularity}
\quad
\begin{enumerate}[(a)]
\item The state space~$X$ is a bounded domain in~$\R^d$, uniformly regular of class~$C^4$. In particular, its boundary is a~$C^4$-manifold.
\item The vector field is continuously differentiable on the closure of~$X$ for all times, i.e.~$v(\theta,\cdot)\in C^1(\overline{X},\R^d)$, and it is continuously differentiable in time, i.e.\ the mapping~$\theta\mapsto v(\theta,x)$ is in~$C^1(\tau S^1,\R^d)$ for every~$x\in \overline{X}$. Further,~$v$ is tangential to~$\partial X$ on the boundary~$\partial X$ for all times.
\end{enumerate}
\end{assumption}

From now on we will consider small random perturbations of the deterministic dynamics. These are stochastic processes $\{x_t\}_{t\ge s}$ governed by the It\^o differential equation
\begin{equation}
dx_t = v(t,x_t)dt + \ep dw_t,
\label{eq:SDE}
\end{equation}
where $\{w_t\}_{t\ge s}$ is a $d$-dimansional standard Wiener process~\cite{Oks98}, and~$\ep>0$ is a parameter which is assumed to be small compared with the magnitude of the vector field.\footnote{For simplicity, we are only considering noise with constant coefficients. Nevertheless, all the following statements hold true if~$\ep$ is replaced by~$\sigma(t,x)$, such that~$\sigma\in C^1(\tau S^1\times\overline{X},\R^{d\times d})$, and~$\sigma(\theta,x)$ is uniformly positive definite in~$\theta$ and~$x$.} Also, we take reflecting boundary conditions, i.e.\ no realizations of the process~$\{x_t\}$ are allowed to leave~$X$ at any times.

Escape rates for SDEs~\cite{WilhelmHuisingaSeanMeyn2004} are defined through the probability distribution of the process:
\begin{definition} \label{def:escrate}
The \emph{(lower) escape rate} for a nice family of sets $\{A_r\}_{r\ge s}$, starting at some time~$s$, is given by
\begin{eqnarray*}
E(\{A_r\}_{r\ge s}) & = & -\limsup_{t\to\infty} \frac{1}{t-s} \log \prob_{x_s\sim m}\left(x_{r}\in A_r\ \forall r\in[s,t]\right)\\
&=&-\limsup_{t\to\infty} \frac{1}{t-s} \log \prob_{x_s\sim m}\left(\bigcap_{r\in[s,t]} \{\omega\,\vert\, x_r(\omega)\in A_r\}\right).
\end{eqnarray*}
where $m$ denotes the Lebesgue measure,~$x_s\sim m$ means that the random variable~$x_s$ has distribution~$m$, and~$\omega$ represents a realization of the noise process from a probability space which we do not specify here further.
\end{definition}
The existence of $E(\{A_r\})$ requires the $\prob$-measurability of the event
\[
\bigcap_{r\in[s,t]} \{\omega\,\vert\, x_r(\omega)\in A_r\}\,.
\]
One can follow the proof of Proposition~\ref{prop:nice=measurable}, using the continuity of sample paths\footnote{i.e.\ that $t\mapsto x_t$ is a continuous function for $\prob$-a.e.\ $\omega$; cf.~\cite{Oks98}.} to obtain the measurability of $\mathcal{E}:=\bigcap_{r\in[s,t]} \{\omega\,\vert\, x_r(\omega)\in A_r\}$ for a family $\{A_r\}$ of nice sets.

If the vector field is time-independent, the process $\{x_r\}$ is \emph{(time-)homogeneous} (i.e.\ $\prob_{x_s=y}(x_t\in A) = \prob_{x_{s+r}=y}(x_{t+r}\in A)$ for every $r\ge 0$, $s\le t$, and $A\in\mathfrak{B}$). Here, and in the following, $\mathfrak{B}$ denotes the Borel sigma algebra on~$X$. Then, the escape rate from a single closed set $A\in\mathfrak{B}$ is well defined,
\begin{eqnarray*}
E(A)&=& -\limsup_{s\to\infty} \frac{1}{t-s} \log \prob_{x_s\sim m}\left(x_{r}\in A\ \forall r\in[s,t]\right)\\
&=&-\limsup_{s\to\infty} \frac{1}{t-s} \log \prob_{x_s\sim m}\left(\bigcap_{r\in[s,t]} \{\omega\,\vert\, x_r(\omega)\in A\}\right).
\end{eqnarray*}

Before we go on to find families of sets with low escape rates, we generalize the result of Theorem~\ref{thm:escrate_nonauto_aug} to the noisy dynamical setting.
To this end, note that the non-homogeneous SDE~$dx_t = v(t,x_t)dt + \ep dw_t$ can be made homogeneous by augmenting the state space, exactly as before:
\begin{equation}
\begin{aligned}
d\theta_t &= dt,\\
dx_t &= v(\theta_t,x_t)dt + \ep dw_t,
\end{aligned}
\label{eq:fullSDE}
\end{equation}
or, alternatively,
\begin{equation}
d\aug{x}_t = \aug{v}(\aug{x}_t)dt + \aug{\ep}d\aug{w}_t
\label{eq:augSDE}
\end{equation}
where the augmented state~$\aug{x}$ and vector field~$\aug{v}$ are as before,~$\{\aug{w}_t\}$ is a~$(d+1)$-dimensional standard Wiener process, and
\[
\aug{\ep} = \begin{pmatrix}
0 & 0\\
0 & \ep I_{d\times d}
\end{pmatrix}\in\R^{(d+1)\times(d+1)},
\]
with~$I_{d\times d}\in\R^{d\times d}$ being the identity matrix.

\begin{proposition} \label{prop:escapeduality}
Let~$\{A_r\}_{r\in\tau S^1}$ be a family of sets and~$\aug{A}\subset\tau S^1\times X$ the associated augmented set, cf.~\eqref{eq:augset}. Let~$E(\{A_r\})$ denote the escape rate of the process~$\{x_r\}_{r\ge s}$ governed by~\eqref{eq:SDE} from the family~$\{A_r\}$ of sets, as defined in Definition~\ref{def:escrate}. Let~$\aug{E}(\aug{A})$ denote the escape rate of the associated augmented process~$\{\aug{x}_r\}$ governed by~\eqref{eq:augSDE} from the augmented set~$\aug{A}$. Then, we have
\begin{equation*}
E(\{A_r\}) = \aug{E}(\aug{A}).
\end{equation*}
\end{proposition}
\begin{proof}
By~\eqref{eq:fullSDE} and~\eqref{eq:augSDE} we can identify~$x_r$ with the last~$d$ coordinates of~$\aug{x}_r$, and the time~$r$ with the first coordinate of~$\aug{x}_r$. Thus, we can also identify the probability laws of these two processes (given~$x_s\sim m$ and~$\aug{x}_s\sim \delta_s\otimes m$). With
\[
x_r\in A_r\quad \Longleftrightarrow \quad \aug{x}_r \in \{r\}\times A_r \subset \aug{A}
\]
the claim follows from the fact that for time-periodic forcing the escape rates are independent on starting time, and that the time-fibers of the Lebesgue measure~$\aug{m}$ on~$\tau S^1\times X$ is the Lebesgue measure~$m$, independently of the fiber chosen (recall that we chose the initial distribution for escape rates to be Lebesgue).
\end{proof}

Now we turn to the questions of (i) whether there are coherent families of sets with small escape rates, and (ii) how to find them.

\section{Finding families with low escape rates}

\subsection{The transfer operator and its infinitesimal generator} \label{ssec:transferop}

The evolution of the system~\eqref{eq:SDE} is characterized by the time-parametrized family of \emph{stochastic transition function} $p_{s,t}:X\times\mathfrak{B}\to[0,1]$, $t\ge s$,
\[
p_{s,t}(y,A) = \prob_y(x_t\in A),
\]
describing the probability distribution of the process, given it started in $x_s=y\in X$ a.s.\ (almost surely). The stochastic transition function~$p_{s,t}(y,\cdot)$ can be shown to be absolutely continuous with respect to the Lebesgue measure for a.e.\ (almost every)~$y\in X$ (cf.~\cite[Section~11.6]{LaMa94}), and thus there is a family of \emph{transition density functions}~$\{q_{s,t}\}_{t\ge s}$, with
\[
p_{s,t}(y,A) = \int_A q_{s,t}(y,z)\,dm(z)\quad\text{for a.e.~}y\in X.
\]

If the initial condition~$x_s$ is only given by its density function~$f$, the probability distribution of~$x_t$ is
\[
\prob_{x_s\sim f}(x_t\in A) = \int_A\int_X q_{s,t}(y,z)f(y)\,dm(y)dm(z)\,.
\]
Note that the probability measure $\prob_{x_s\sim f}(x_t\in \cdot)$ is absolutely continuous with respect to~$m$: we denote its Radon--Nikod\'ym derivative by~$f_t$. The (linear) operator mapping~$f$ to~$f_t$ is called the \emph{transfer operator}~$\P_{s,t}$,\footnote{In the field of operator semigroups (see below, or, for instance~\cite{EnNa00}) the family of transfer operators,~$\{\P_{s,t}\}_{t\ge s}$, is also called \emph{evolution family}.} and is given by
\[
\P_{s,t}f := \int_X q_{s,t}(y,\cdot)f(y)\,dm(y)\,.
\]
We will see below that, due to $v$ being divergence-free for all times, $\P_{s,t}\mathbbm{1} = \mathbbm{1}$ for every~$s\le t$, and thus $\P_{s,t}:L^r(X)\to L^r(X)$ is a well-defined contraction for every $1\le r\le \infty$; cf.~\cite[Lemma 1]{BaRo95}. The function~$\mathbbm{1}$ is the a.e.\ constant function on~$X$ with value~$1$.

Suppose that~$f_t(x)$ and~$v(t,x)$ are~$C^2$ functions with respect to both~$t$ and~$x$. Then, $g(t,\cdot)=f_t$ solves the \emph{Fokker--Planck} equation~\cite[Section~11.6]{LaMa94}
\begin{equation}
\partial_t g(t,x) = \frac{\ep^2}{2}\Delta g(t,x) - div\left(g(t,\cdot) v(t,\cdot)\right)(x),\quad g(0,x) = f(x),\quad \frac{\partial g(t,\cdot)}{\partial n}\!\bigg\vert_{\partial X}=0\,,
\label{eq:FPeq}
\end{equation}
where $\Delta = \sum_{i=1}^d\partial_{x_i}^2$ denotes the Laplace operator, and $div = \sum_{i=1}^d\partial_{x_i}$ denotes the divergence operator on~$X$, and~$\tfrac{\partial}{\partial n}$ is the normal derivative on the boundary.

Assume for a moment that $v(t,x) = v(x)$, i.e.\ the vector field doesn't depend on time. Then,~$\P_{s,t}$ only depends on~$t-s$, and for simplicity we write~$\P_t$. The \emph{semigroup property} holds: $\P_{t+s} = \P_t\P_s$ for every $s,t\ge 0$. We define the operator
\begin{equation}
\IG f = \lim_{t\to 0}\frac{\P_tf-f}{t}\qquad (\text{convergence in }L^r).
\label{eq:defIG}
\end{equation}
$\IG:D(\IG)\to L^r(X)$ is called the \emph{(infinitesimal) generator} of the semigroup of operators $\{\P_t\}_{t\ge 0}$, and $D(\IG)$ is its domain, i.e.\ the set containing all $f\in L^r(X)$ where the limit~\eqref{eq:defIG} exists. If $f\in C^2(X,\R)$, then $f\in D(\IG)$, and $\IG f = \tfrac12\ep^2\Delta f - div(fv)$, hence~$\IG$ is the operator building the right hand side of the Fokker--Planck equation~\eqref{eq:FPeq}.
That $\IG$ indeed ``generates'' $\{\P_t\}$ is underlined by the following result.

\begin{proposition}[Spectral Mapping Theorem {\cite[Theorem~2.2.4]{Paz83}}] \label{prop:SMT}
For~$f\in L^r(X)$ and~$\lambda\in\C$ one has~$\IG f = \lambda f$ iff~$\P_t f = \exp(\lambda t)f$ for every~$t\ge 0$.
\end{proposition}

Essentially, all the information contained in the family $\{\P_t\}$ is also contained in a single operator,~$\IG$. To work directly with the generator has some advantages, if one is to (numerically) compute quantities of interest related to dynamical systems. This has been exploited e.g.\ in~\cite{FrJuKo13}. In the non-autonomous setting a general result such as Proposition~\ref{prop:SMT} does not exist, however if~$f_t$ is sufficiently smooth, it satisfies $\partial_t f_t = \IG_t f_t$, where~$\IG_t$ is the infinitesimal generator associated with the vector field~$v_t:=v(t,\cdot)$, i.e.~$\IG_t f = \tfrac12\ep^2\Delta f - div(fv_t)$. This also shows that if $v(t,\cdot)$ is divergence-free for all~$t$, then $\IG_t\mathbbm{1}=0$ for all~$t$, and hence $\P_{s,t}\mathbbm{1} = \mathbbm{1}$ for all~$s\le t$.

\subsection{Lyapunov exponents and escape rates}

\paragraph{Escape rates for homogeneous processes.}

The following theorem extends results obtained in~\cite{FrSt10,FrSt13,FrJuKo13}, to continuous time stochastic processes.
It states that for eigenfunctions of~$\IG$ with decay rate $\lambda$, the positive and negative supports of the function determine sets from which the escape rate is slower than the decay rate~$\lambda$.
\begin{theorem}	\label{thm:cont_time_escrate}
For a homogeneous stochastic process $\{x_t\}_{t\ge 0}$ with almost surely continuous sample paths (cf.~\cite{Oks98}), let $\{\P_t\}_{t\ge 0}$ denote the associated transfer operator semigroup on $L^1 = L^1_m$. Let $f\in L^1\cap L^{\infty}$ be such that $\P_t f = e^{\lambda t}f$ for a $\lambda < 0$. Let us further assume that $A^{\pm} := \{\pm f\ge 0\}$ is closed. Then $E(A^{\pm})\le -\lambda$.
\end{theorem}
\begin{proof}
See Appendix~\ref{app:auto_escrate}.
\end{proof}
The case~$\lambda\notin\R$ has been considered in~\cite{FrJuKo13} and in Remark~\ref{rem:complex} below we extend these results as well to the non-autonomous (non-homogeneous) setting.

\paragraph{Escape rates for non-homogeneous processes.}

If we are looking for a non-homogeneous version of Theorem~\ref{thm:cont_time_escrate}, it is to be expected that the role of \emph{one} function~$f$ is going to be taken by a whole \emph{family} of functions $\{f_r\}$. To assume that the temporal change in sets like $\{f_t\ge 0\}$ is everywhere continuous might be too optimistic. We thus weaken the notion of niceness introduced in Definition~\ref{def:niceness}:
\begin{definition}[Sufficiently nice family of sets] \label{def:suff_niceness}
A \emph{sufficiently nice} family of sets $\{A_r\}_{r\in \mathcal{I}}$, for an interval~$\mathcal{I}\subset\R$, is such that
\begin{enumerate}[(i)]
\item for every~$r\in\mathcal{I}$ the set $A_r$ is closed; and
\item the function $\sigma\mapsto h(A_r,A_{r+\sigma})$ is either left- or right-continuous at $\sigma=0$ for every~$r\in\mathcal{I}$.
\end{enumerate}
\end{definition}
We also introduce \emph{decay rates} as a function-based version of the set-based escape rates.
\begin{definition}[Lyapunov exponents]
Given a $f\in L^1_m$, the \emph{Lyapunov exponent} of~$f$ with respect to initial time~$s$ is defined as the decay rate of the norm of $f$ while transported by the transfer operator:
\begin{equation}
\Lambda_s(f) := \limsup_{t\to\infty} \frac{1}{t-s}\log \|\P_{s,t}f\|_1\,.
\label{eq:LyapExp}
\end{equation}
\end{definition}
We have the following:
\begin{theorem} \label{thm:nonauto_escrate}
Given~$f\in L^1\cap L^{\infty}$, $\int f=0$, define the family of sets~$\{A_r^{\pm}\}_{r\ge s}$ by~$A_r^{\pm} = \{\pm\P_{s,r}f \ge 0\}$. Suppose that~$\{A_r^{\pm}\}$ is a sufficiently nice family. Then~$E(\{A_r^{\pm}\}) \le -\Lambda_s(f)$.
\end{theorem}
\begin{proof}
See Appendix~\ref{app:nonauto_escrate}.
\end{proof}

Theorem~\ref{thm:nonauto_escrate} states that if we find a function with a low decay rate~$\Lambda$ (a ``low'' decay means $\Lambda$ is small negative, i.e.\ close to zero), than we readily have a pair of sets with coherence (escape) rate at least as low as the decay rate~$\Lambda$. Thus, we aim for finding functions with low decay rates. Note that since $\P_{s,t}\mathbbm{1} = \mathbbm{1}$ for all~$s\le t$, any function not perpendicular to~$\mathbbm{1}$, i.e.\ satisfying $\int f \neq 0$, will have a decay rate~$\Lambda=0$, because $\|\P_{s,t}f\|_1 \not\to 0$ as $t\to\infty$.


In general, properties of interest in periodically forced nonlinear dynamical systems can be characterized by those of the associated \emph{Poincar\'e return map}, the map describing the time-$\tau$-evolution of the system with forcing period~$\tau$. By considering the transfer operator for a single period~$\tau$ we have the following simple characterization of the slowest decaying family of functions.
\begin{proposition}	\label{prop:CohPoiB}
For every $s,\tilde{s}\in \tau S^1$ we have
\begin{equation}
\max_{\int\!\! f=0} \Lambda_s(f) = \frac{1}{\tau}\log\left|\lambda_2\big(\P_{s,s+\tau}\big)\right| = \frac{1}{\tau}\log\left|\lambda_2\big(\P_{\tilde{s},\tilde{s}+\tau}\big)\right|,
\label{eq:CohPoiB}
\end{equation}
where~$\lambda_2(\cdot)$ denotes the second largest eigenvalue of the argument. The maximizer of the left hand side is the eigenfunction~$f_s$ of~$\P_{s,s+\tau}$ corresponding to this eigenvalue.
\end{proposition}
\begin{proof}
See Appendix~\ref{app:decaybound1}.
\end{proof}
The second equality in~\eqref{eq:CohPoiB} states that this eigenvalue does not depend on~$s$.


\section{The generator on augmented phase space}
\label{sec:genaug}

Let us first recapitulate the developments in the last three sections. In order to have a notion of coherence with a tractable stochastic generalization, we introduced escape rates based on survivor sets in sections~\ref{ssec:survival} and~\ref{ssec:escape}. There, the main result was Theorem~\ref{thm:escrate_nonauto_aug}, showing that the non-autonomous escape rate equals the autonomous escape rate of the associated augmented system. The generalization of escape rates to stochastically perturbed systems was introduced in section~\ref{sec:cohfam}. Finally, Theorem~\ref{thm:nonauto_escrate} and Proposition~\ref{prop:CohPoiB} showed that coherent families can be found by considering the Lyapunov spectrum of the transfer operator family associated with the system.

Next we would like to find coherent families directly in augmented space, in a way crossbreeding theorems~\ref{thm:escrate_nonauto_aug} and~\ref{thm:nonauto_escrate}. To this end, we start with an informal reasoning to get to our main object, introduced in equation~\eqref{eq:augIG} below.

Recall from equations~\eqref{eq:fullSDE} and~\eqref{eq:augSDE} that the non-homogeneous SDE $dx_t = v(t,x_t)dt + \ep dw_t$ can be made homogeneous by augmenting the state space,~$d\aug{x}_t = \aug{v}(\aug{x}_t)dt + \aug{\ep}d\aug{w}_t$.
For a homogeneous process, Theorem~\ref{thm:cont_time_escrate} connects the escape rate from a set and the spectrum of the corresponding transfer operator semigroup. Hence, if~$\aug{\P}_t$ denotes the transfer operator semigroup of the augmented SDE~\eqref{eq:augSDE}, then the eigenfunctions of~$\aug{\P}_t$ yield augmented sets with low escape rates\footnote{Thus, by Proposition~\ref{prop:escapeduality}, it also yields coherent families. Here, however, we will take a different path, to add further connections shown in Figure~\ref{fig:connections}, and to allow us to deal with slightly more general coherent families (recall that Theorem~\ref{thm:cont_time_escrate} requires closed sets). This construction will further reveal \emph{quasi-periodic} coherent families, cf.~Remark~\ref{rem:complex} below.}. Furthermore, by the spectral mapping theorem (Proposition~\ref{prop:SMT}), these eigenfunctions are the same as those of the associated infinitesimal generator~$\augIG$. The generator can be obtained from the Fokker--Planck equation associated with~\eqref{eq:augSDE}: for a scalar function~$\aug{f}$, with~$\aug{f}(\aug{x}) = \aug{f}(\theta,x)$ and \emph{``time-slices''}~$f_{\theta}:=\aug{f}(\theta,\cdot)$, we have
\begin{equation}
\augIG \aug{f}(\theta,x) = -\partial_{\theta}\aug{f}(\theta,x) + \IG_{\theta}f_{\theta}(x)\,,
\label{eq:augIG}
\end{equation}
with domain satisfying
\[
D(\augIG) \supset \left\{ \aug{f}\,\big\vert\, \theta\mapsto f_\theta\in C^1\!\left(\tau S^1,L^r\right),\ f_{\theta}\in D(\IG_{\theta})\ \forall \theta\in \tau S^1\right\}\,,
\]
for some~$1\le r<\infty$.
We call~$\augIG$ the \emph{augmented generator}.

\begin{remark}
The semigroup~$\aug{\P}_t$ above is equivalent to the so-called \emph{evolution semigroup}; cf.~\cite{ChiLa99,Sch99,RaeEtAl00,EnNa00} and references therein, while the idea seems to go back to~\cite{How74}.
\end{remark}

\paragraph{Lyapunov exponents.}

Our next results state that the augmented generator contains the information about the Lyapunov spectrum of the periodically driven transfer operator family; that is, it could be viewed as a spectral mapping theorem for such families of operators. In the context of evolution semigroups such results are again known by the name spectral mapping theorems (or the existence of exponential dichotomies), and can be found e.g.~in~\cite{ChiLa99} or~\cite[Section~VI.9]{EnNa00}.
\begin{lemma}	\label{lem:spectral_con}
If $\augIG \aug{f} = \mu \aug{f}$ for some~$\mu\in\C$, then $\aug{f}(s+t \mod \tau,\cdot) = e^{-\mu t}\P_{s,s+t} \aug{f}(s,\cdot)$, for $s\in \tau S^1$, $t\ge 0$. In particular, $\P_{s,s+\tau}f_s = e^{\mu\tau} f_s$, i.e.\ the time-slices of~$\aug{f}$ are eigenfunctions of~$\P_{s,s+\tau}$ for the corresponding~$s\in\tau S^1$.
\end{lemma}
\begin{proof}
See Appendix~\ref{app:auggenproofs}.
\end{proof}
Note that since eigenfunctions of $\P_{s,t}$ have zero mean, we have $\int f_{\theta}\,dm =0$ for all~$\theta\in\tau S^1$.
Thus, combining Lemma~\ref{lem:spectral_con} with Proposition~\ref{prop:CohPoiB} shows that eigenfunctions of~$\augIG$ at \emph{subdominant} eigenvalues, i.e.\ at eigenvalues~$\mu$ where $\mathrm{Re}(\mu)$, the real part of~$\mu$, is close to zero, yield coherent families. The following theorem shows that these eigenfunctions are exactly the same as those in the transfer operator based family arising implicitly in~\eqref{eq:CohPoiB}. Hence, it is an augmented time-infinitesimal version of Proposition~\ref{prop:CohPoiB}.


\begin{theorem}	\label{thm:IGcohPeriodic}
For every~$s\in\tau S^1$ we have
\begin{equation}
\max_{\int\!\! f=0} \Lambda_s(f) = \mathrm{Re}\left(\lambda_2\left(\augIG\right)\right)\,,
\label{eq:lyap auggen}
\end{equation}
with~$\lambda_2(\augIG)$ denoting the first subdominant eigenvalue of~$\augIG$. The $s$-dependent maximizer of the left-hand side is~$f = \aug{f}(s,\cdot)$, where~$\aug{f}$ is the corresponding eigenfunction of~$\augIG$.
\end{theorem}
\begin{proof}
See Appendix~\ref{app:auggenproofs}.
\end{proof}

\begin{remark}[The complex eigenvalue case] \label{rem:complex}
Suppose~$\lambda_2(\augIG) =: \mu\in\C\setminus\R$. Then there is a~$\delta>0$ such that $e^{\delta\mu}\in\R$. For~$t=s+k\delta$, $k\in\N$, we then have by recalling~$\P_{s,t}f_s = e^{\mu(t-s)}f_t$ from Lemma~\ref{lem:spectral_con}, that
\begin{equation}
\P_{s,t}f_s = e^{k\delta\mu}f_t\,.
\label{eq:complexEV}
\end{equation}
However,~$\P_{s,t}$ maps real-valued functions again to real-valued ones, hence~\eqref{eq:complexEV} holds also for the real parts~$f_s^{\rm R}, f_t^{\rm R}$ and imaginary parts~$f_s^{\rm I}, f_t^{\rm I}$ of~$f_s,f_t$, respectively. The operator~$\P_{s,t}$ is non-expansive for \emph{every}~$s\le t$, hence the function~$t\mapsto \|\P_{s,t}g\|_{L^1}$ is monotonically decreasing for any~$g$. Setting~$g = f_s^{\rm R}$ or~$g = f_s^{\rm I}$, this shows together with~\eqref{eq:complexEV} that
\begin{equation}
\label{eq:realimag}
\Lambda_s(f_s^{\rm R}) = \Lambda_s(f_s^{\rm I}) = \mu.
\end{equation}

In order to obtain a coherent family in the sense of Theorem~\ref{thm:nonauto_escrate}, we may consider the positive (or negative) support of~$\P_{s,t}f_s^{\rm R}$. Again from Lemma~\ref{lem:spectral_con}, we see that
\begin{equation}
\P_{s,t}f_s^{\rm R} = \mathrm{Re}\big(e^{\mu(t-s)}f_t\big) = e^{\alpha(t-s)}\left(\cos\big(\beta(t-s)\big)f_t^{\rm R} - \sin\big(\beta(t-s)\big)f_t^{\rm I}\right)\,,
\label{eq:complex family}
\end{equation}
where~$\mu = \alpha + \beta i$, $\alpha,\beta\in\R$ (we note that a similar construction has been employed in~\cite{FrGTQu14} for the discrete time case). 
In particular, the real part of $f_s$ returns to a real function $f_t$ at some later $t=2\pi/\beta$, where $\beta$ is completely independent of the driving period $\tau$.
The corresponding coherent family of sets will be periodic if and only if~$\tau$ and~$\tfrac{2\pi}{\mathrm{Im}(\mu)}$ are rationally dependent (i.e., they have a finite common multiple), otherwise the family will be quasi-periodic.
We now show that the decision to use $f^{\rm R}$ represents one ``phase'' from a cycle of period $2\pi/\beta$.

Since the decay rate of both~$f_s^{\rm R}$ and~$f_s^{\rm I}$ is~$\mu$, the decay rate of any (complex) linear combination of these functions is also $\mu$. Thus,
\[
f_s^{\vartheta}:= \mathrm{Re}(e^{i\vartheta}f_s)
\]
yields a coherent family for every~$\vartheta\in[0,2\pi)$ and~$\vartheta$ acts as a \emph{phase} for the family of coherent sets. One has the same evolution rule as in~\eqref{eq:complex family}:
\[
\P_{s,t}f_s^{\vartheta} = \mathrm{Re}(e^{i\vartheta+\mu (t-s)}f_t)\,.
\]

\end{remark}

\paragraph{Some notes on the spectrum of $\augIG$.}

The differentiation operator~$\partial_{\theta}$, acting on~$C^1(\tau S^1)$ has the eigenfunctions~$\psi_k(\theta) = \exp(2\pi i k \theta/\tau)$,~$k\in\Z$, with corresponding eigenvalues~$2\pi i k \tau^{-1}$. If~$\aug{f}$ is an eigenfunction of~$\augIG$ for some eigenvalue~$\lambda$, then~$\aug{g}$ defined by~$g_{\theta} = f_{\theta} \psi_k(\theta)$ is also an eigenfunction of~$\augIG$, since
\begin{equation}
\begin{aligned}
(\augIG \aug{g})(\theta,\cdot) &= \IG_{\theta} g_{\theta} - \partial_{\theta} g_{\theta}\\
&= (\underbrace{\IG_{\theta} f_{\theta} -\partial_{\theta} f_{\theta}}_{=(\augIG \aug{f})(\theta,\cdot) = \lambda f_{\theta}})\psi_k(\theta) - 2\pi ik\tau^{-1} f_{\theta}\psi_k(\theta) \\
&= (\lambda-2\pi ik\tau^{-1}) \aug{g}(\theta,\cdot)\,.
\end{aligned}
\label{eq:spectral_shift}
\end{equation}
Thus the point spectrum of~$\augIG$ is invariant under translation by the multiples of the scaled imaginary unit~$2\pi i\tau^{-1}$. We call such translates of an eigenvalue ``companion eigenvalues''. We shall also analyze in the following how this spectral property is distorted by the numerical discretization.

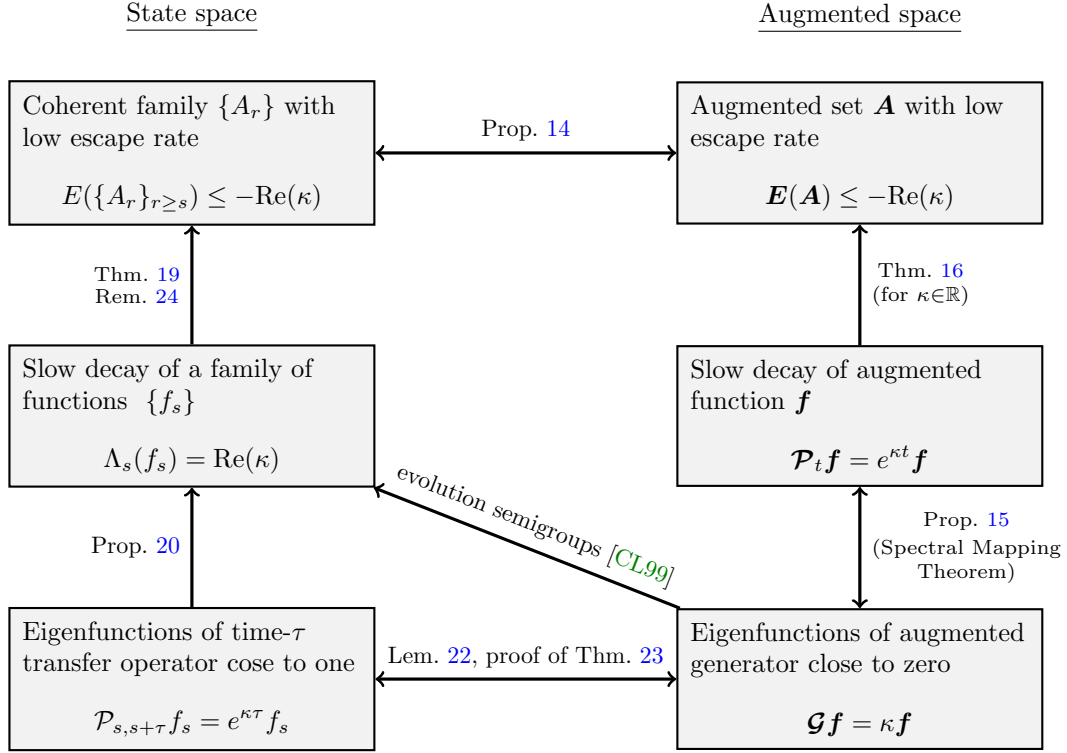
\begin{figure}[h]
\centering

\begin{tikzpicture}
[auto,font=\footnotesize,
every node/.style={node distance=3.5cm},
object/.style={
rectangle, thick, draw, fill=black!5, inner sep=5pt, text width=4.5cm, minimum height=1.5cm
}] 


\node [object] (cohfam) {
Coherent family~$\{A_r\}$ with low escape rate
\[
E(\{A_r\}_{r\ge s}) \le -\mathrm{Re}(\kappa)
\]
};

\node [object, right=4cm of cohfam] (augcohfam) {
Augmented set~$\aug{A}$ with low escape rate
\[
\aug{E}(\aug{A}) \le -\mathrm{Re}(\kappa)
\]
};

\node [object, below of=cohfam] (decfam) {
Slow decay of a family of functions ~$\{f_s\}$
\[
\Lambda_s(f_s) = \mathrm{Re}(\kappa)
\]
};

\node [object, below of=augcohfam] (decfun) {
Slow decay of augmented function~$\aug{f}$
\[
\aug{\P}_t\aug{f} = e^{\kappa t}\aug{f}
\]
};

\node [object, below of=decfam] (eigper) {
Eigenfunctions of time-$\tau$ transfer operator cose to one
\[
\P_{s,s+\tau} f_s = e^{\kappa\tau} f_s
\]
};

\node [object, below of=decfun] (augeig) {
Eigenfunctions of augmented generator close to zero
\[
\augIG \aug{f} = \kappa \aug{f}
\]
};

\node[above= 0.5cm of cohfam] (stsp) {\underline{State space}};
\node [above=0.5cm of augcohfam] (augsp) {\underline{Augmented space}};



\path[<->,very thick]
(cohfam) edge [above] node {\scalebox{0.9}{Prop.~\ref{prop:escapeduality}}} (augcohfam)
(augeig) edge [right] node {\scalebox{1.1}{$\substack{\text{Prop.~\ref{prop:SMT}}\\[2pt] \text{(Spectral Mapping} \\ \text{Theorem)}}$}} (decfun)
;

\path[->,very thick]
(decfam) edge [left] node { \scalebox{1.1}{$\substack{\text{Thm.~\ref{thm:nonauto_escrate}}\\[2pt] \text{Rem.~\ref{rem:complex}}}$}}
(cohfam)
(decfun) edge [right] node { \scalebox{1.1}{$\substack{\text{Thm.~\ref{thm:cont_time_escrate}}\\[2pt] (\text{for }\kappa\in\R)}$}} (augcohfam)
(augeig) edge [midway, above, sloped] node { \scalebox{0.9}{evolution semigroups~\cite{ChiLa99}}} (decfam)
(eigper) edge [left] node { \scalebox{0.9}{Prop.~\ref{prop:CohPoiB}}} (decfam)
;
\path[<->,very thick]
(augeig) edge [above] node {\scalebox{0.9}{Lem.~\ref{lem:spectral_con}, proof of Thm.~\ref{thm:IGcohPeriodic}}} (eigper)
;
\end{tikzpicture} 
\caption{Connections between the different objects and concepts introduced and derived in sections~\ref{sec:cohfam}--\ref{sec:genaug}. The objects on the left-hand side live in or on the state space~$X$, while those on the right-hand side live in or on the augmented space~$\aug{X}$. The arrows indicate that from the objects in the box at the tail of an arrow one can construct objects in the box at the head of that arrow. The precise statements, and the conditions under which they hold true, can be found in the main text.}
\label{fig:connections}
\end{figure}

\section{Numerical analysis}
\label{sec:numanal}

Our general strategy is going to be to solve the eigenproblem in the bottom right corner of Figure~\ref{fig:connections} numerically, and then use our results to move to the top left corner of this figure and infer quantitative information about coherent families in the system at hand.

\subsection{Galerkin and Petrov--Galerkin methods}
\label{ssec:Galerkin}

The Galerkin discretization of an operator $\mathcal A$ over some Hilbert space~$\mathcal H$ with scalar product $\langle\cdot,\cdot\rangle$ can be described as follows. Suppose we have a finite-dimensional subspace~$\mathcal V \subset \mathcal H$ with basis~$(\psi_1,\ldots,\psi_k)$ given. The \emph{Galerkin projection} of $\mathcal A$ to $\mathcal V$ is the unique linear operator $A:\mathcal V\to \mathcal V$ satisfying
\begin{equation}
\langle \psi_j, \mathcal A \psi_i\rangle = \langle \psi_j, A\psi_i\rangle,\quad\text{for all } i,j=1\ldots,k\,.
\label{eq:Galerkin}
\end{equation}
If the operator $\mathcal A$ is not given on a Hilbert space, just a Banach space, or for computational reasons, it can be advantageous to take \emph{basis functions} (with respect to which the projected operator is defined) and \emph{test functions} (which serve in \eqref{eq:Galerkin} to project objects not necessarily living in the same subspace) from different sets.

If $\mathcal A:\mathcal X\to \mathcal X$ is an operator on a Banach space $\mathcal X$, $\mathcal V\subset \mathcal X$ a subspace with basis $(\psi_1,\ldots,\psi_k)$, $\mathcal V^*\subset\mathcal X^*$ a subspace of the dual of $\mathcal X$ with basis $(\psi_1^*,\ldots,\psi_k^*)$, in particular~$\mathrm{dim}\mathcal V = \mathrm{dim}\mathcal V^*$, then the \emph{Petrov--Galerkin} projection of $\mathcal A$ is the unique linear operator $A:\mathcal V\to\mathcal V$ satisfying
\begin{equation}
\langle \psi_j^*, \mathcal A\psi_i\rangle = \langle \psi_j^*,A\psi_i\rangle,\quad\text{for all }i,j=1,\ldots,k\,,
\label{eq:PetrovGalerkin}
\end{equation}
where $\langle\cdot,\cdot\rangle$ denotes the duality bracket. There are two methods of special interest for us.

\paragraph{Ulam's method.}

Let the operator to project be the transfer operator $\P_{s,t}:L^1\to L^1$, and suppose a partition of~$X$ into measurable sets $B_1,\ldots,B_n$ is given. Define the basis functions as the normed characteristic functions over the~$B_i$, that is, $\psi_i = m(B_i)^{-1}\mathbbm{1}_{B_i}$. Since thus $\psi_i\in L^{\infty}$, we can take~$\psi^*_i(f) = \int f\psi_i$, and the corresponding Petrov--Galerkin projection of~$\P_{s,t}$, denoted by $P_n(s,t)\in\R^{n\times n}$, is called the \emph{Ulam discretization} of the transfer operator. Direct computation shows that
\[
P_{n,ij}(s,t) = \prob_{x_s\sim\mathrm{unif}(B_j)}(x_t\in B_i)\,,
\]
where $\mathrm{unif}(B)$ denotes the uniform distribution over the set~$B$.

\paragraph{Collocation.}

If~$\mathcal X$ is a function space satisfying~$\mathcal X\subset C^0(X)$, and~$x_1,\ldots,x_n\in X$, one can take~$\psi^* = \delta(\cdot - x_j) \in C^0(\mathcal X)^*$, the delta distributions centered at the~$x_j$, i.e.\ $\langle\delta(\cdot - x_j),\psi\rangle = \psi(x_j)$. The associated projection is called \emph{collocation}, because it satisfies
\[
(A\psi)(x_i) = (\mathcal{A}\psi)(x_i)\quad\text{ for all }\psi\in\mathcal V,\ i=1,\ldots,n\,.
\]

\subsection{Ulam's discretization for the generator} \label{ssec:GeneratorUlam}

Partition $X$ into $\{B_1,\ldots,B_n\}$, with each $B_k$, $k=1,\ldots,n$ closed, with piecewise smooth boundary, and nonempty interior. For simplicity, let us assume that the~$B_i$ are rectangular sets.

We have seen in section~\ref{ssec:transferop} that the transfer operator family associated with the homogeneous process which is governed by the SDE~$dx_t = v(x_t)dt + \ep dw_t$ is a semigroup, and hence has an infinitesimal generator. The Ulam discretization of this generator consists of two components, to be discussed in the following.

The first takes care of the deterministic part, which corresponds to the drift~$v$. It can be seen as a discretization of the divergence term on the right hand side of the Fokker--Planck equation~\eqref{eq:FPeq}. The Ulam matrix~$P_n$ corresponding to the dynamics governed by the ordinary differential equation~$\dot x(t) = v(x(t))$ satisfies that~$P_n(s,t)$ only depends on~$(t-s)$, since the system is autonomous. Let this Ulam discretization be denoted by~$P_n(t)$. Note that since time evolution and discretization do not commute, in general~$P_n(t+s)\neq P_n(t)P_n(s)$. Nevertheless, we define the first component of the generator matrix as
\begin{equation}
G_n^{\rm drift}:= \frac{d}{dt}P_n(t)\!\Big\vert_{t=0}\,.
\label{eq:gen_drift}
\end{equation}
This limit exists, and can be shown to yield\footnote{In~\cite{FrJuKo13}, the matrix representation of the discrete generator was given with respect to the basis functions~$\psi_i = \mathbbm{1}_{B_i}$, here we give it with respect to the basis functions~$\psi_i = \tfrac{1}{m(B_i)}\mathbbm{1}_{B_i}$. Note also, that our matrix representations act by multiplication from the left, i.e.\ $f\mapsto G_n f$ for~$f\in\R^n$.}~\cite{FrJuKo13}
\[
\renewcommand{\arraystretch}{1.55}
G_{n,ij}^{\rm drift} = \left\{\begin{array}{ll}
										\tfrac{1}{m(B_j)}\int_{\partial B_i\cap\partial B_j} \left\langle v(x), n_j(x)\right\rangle^+\, d\sigma(x), & i\neq j \\
										-\tfrac{1}{m(B_i)}\int_{\partial B_i} \left\langle v(x), n_i(x)\right\rangle^+\, d\sigma(x), & i=j,
									\end{array}\right.
\]
where~$n_j$ denotes the unit outer normal vector on~$\partial B_j$. $G_n^{\rm drift}$ in this latter form is also known as the (spatial part of the) upwind scheme~\cite{Lev02}.

The second component deals with the diffusion, hence corresponds to the term with the Laplace operator in~\eqref{eq:FPeq}. Note that if we would have involved the diffusion in the construction above, then the limit in~\eqref{eq:gen_drift} would have diverged. This is due to the non-Lipschitz nature of Brownian motion, as discussed earlier. The diffusion component is merely going to be the finite difference approximation~$\Delta_n$ of the Laplace operator~$\Delta$ on the grid defined by the centroids of the rectangles~$B_i$, given by its representation with respect to the basis functions~$\psi_i = \tfrac{1}{m(B_i)}\mathbbm{1}_{B_i}$. We define the diffusion component of the generator as
\begin{equation}
G^{\rm diff}_n := \frac{\ep^2}{2}\Delta_n\,.
\label{eq:gen_diff}
\end{equation}
The discrete Ulam generator is then~$G_n:=G^{\rm drift}_n + G^{\rm diff}_n$. It can be readily seen that the matrix~$G_n$ is a \emph{generator matrix}, meaning that~$\exp(tG_n)$ is similar to a column-stochastic matrix for every~$t\ge0$. Hence, it has an eigenvalue~$0$, and all eigenvalues are in the left complex half plane, mimicking this property of its continuous original, the infinitesimal generator~$\IG$. Further properties of~$G_n$, such as convergence to~$\IG$ in an appropriate way, are shown in~\cite[Section~5]{KPPHD}.

\subsection{The discretized augmented generator and its spectrum} \label{ssec:discrete spectrum Ulam}

How does the numerical discretization from Section~\ref{ssec:GeneratorUlam} reflect the theoretical findings~\eqref{eq:spectral_shift} about the augmented generator~$\augIG$?

For the Ulam discretization of the augmented generator the time derivative appears as simple backward finite difference, cf.~\cite[Corollary~4.5]{FrJuKo13}.
More precisely, if~$\aug{v}$ is a discretized function in augmented space,\footnote{In this section we are not going to make any reference to the vector field, previously denoted by~$v$. We use~$v$ for a different object here.} then let~$v_t$, $t=0,h,2h,\ldots,\tau-h$, denote its time slices, where~$h = \tau/n$ for some~$n\in\N$. The backward difference operator~$\delta_t$ is given by $(\delta_t \aug{v})_t = h^{-1}(v_t-v_{t-h})$. By denoting~$\omega := e^{2\pi i h/\tau}$, the eigenvectors of~$\delta_t$ are the vectors $\psi_k = (1,\omega^k,\omega^{2k},\ldots,\omega^{(n-1)k})^T$, $k=0,\ldots,n-1$, with corresponding eigenvalues $\lambda_k := \tfrac{1-\omega^{-k}}{h}$.

Let~$\aug{G}$ denote the (Ulam) discretized augmented generator, and~$\aug{v}$ its eigenvector, such that~$\aug{G}\aug{v} = \lambda \aug{v}$. Further, consider~$\aug{w}$ with~$w_t = v_t\psi_k(t)$. A short computation shows that
\begin{equation}
(\aug{G}\aug{w})_t = \lambda w_t + \frac{\omega^{-k}-1}{h}v_{t-h}\psi_k(t) = \lambda w_t - \lambda_k v_{t-h}\psi_k(t)\,,
\label{eq:tildeGwt}
\end{equation}
which can be seen as a discrete analogue to~\eqref{eq:spectral_shift}, as for~$h\to 0$ we have~$\tfrac{\omega^{-k}-1}{h}\to -2\pi ik/\tau$. Note the time shift in~$\aug{v}$ occurring in~\eqref{eq:tildeGwt}. If~$h$ is small and~$v_t$ sufficiently smooth (in~$t$), then \eqref{eq:tildeGwt} shows that~$\aug{w}$ is close to being an eigenvector at~$\lambda - \lambda_k$, since~$v_t\approx v_{t-h}$.

If we have a discrete equivariant measure not changing in time, i.e.~$G_t v_t = 0$ with~$v_t = v$ for some~$v$ for all~$t$, then the associated~$\aug{w}$ are indeed eigenfunctions of~$\aug{G}$ at~$\lambda - \lambda_k$. This is the case e.g.\ for divergence-free non-autonomous vector fields, where the constant density is stationary. This property is inherited by Ulam's discretization of the generator, since by Gau{\ss}' theorem the total inward and outward fluxes cancel out on the boundary of every set~$B_i$, and thus the constant vector is in the null space of the discrete generator.

\subsection{Hybrid generator} \label{ssec:hybrid}

Rather than approximating the space $\tau S^1\times X$ with a single scheme, we approximate time-dependence using Fourier modes\footnote{In order to obtain a pure real approximation of real objects, for every Fourier frequency~$m$ one has to include the mode with the frequency~$-m$ too. Thus, we use an odd number of modes, $M$.} $e^{2\pi i m t/\tau}, m=-(M-1)/2,\ldots,(M-1)/2$; hence we arrive at a hybrid collocation-Ulam scheme.
The advantages of the Fourier representation are that the time derivative is explicitly known and the space spanned by this finite set of modes is invariant under time differentiation. Moreover, if the temporal change of the vector field is sufficiently smooth, due to the spectral convergence properties of Fourier collocation we expect to obtain good results with only a few modes in time direction, i.e.\ very small~$M$.

In the spatial direction we use Ulam's method for the generator \cite{FrJuKo13}.
Partition~$X$ into $\{B_1,\ldots,B_N\}$, with each $B_n$, $n=1,\ldots,N$, closed with piecewise smooth boundary, and nonempty interior. We allow~$B_j\cap B_k$ to intersect at their common boundary. The finite-dimensional estimate of~$\IG_t$ is an~$N\times N$ matrix~$G(t)$.
We are interested in the projection of the augmented generator~$\aug{\IG}$ onto the space spanned by the (orthogonal) functions $\{\aug{\tilde\psi}_{n,m}\}_{n=1,\ldots,N;\ m=-(M-1)/2,\ldots,(M-1)/2}$, with $\aug{\tilde\psi}_{n,m}(t,x) = |B_n|^{-1}\,e^{2\pi i m t/\tau}\mathbbm{1}_{B_n}(x)$. Note, that to avoid a clash in notation, the Lebesgue measure of some set~$B$ in this section is denoted by~$|B|$.

Since we want to use collocation methods, the Fourier modes shall be represented by a different basis, which is more convenient both for the assembling of the discrete operator and for the evaluation of functions given in that basis. Let $\{\ell_m\}_{m= 1 ,\ldots,M}$ denote the \emph{Lagrangian basis} of the subspace spanned by the Fourier modes with respect to the equidistant collocation nodes $t_l = \tau (l-1)/M$, $l = 1,\ldots,M$. That is, $\ell_m(t_l) = \delta_{l,m}$, where we use the usual Kronecker notation. The basis functions with respect to which we set up our discretization is $\aug{\psi}_{n,m} = |B_n|^{-1}\mathbbm{1}_{B_n}\ell_m$, $n=1,\ldots,N$, $m=1,\ldots,M$.


Now, to discretize~$\augIG$ over the hybrid basis, we employ a Petrov--Galerkin method with $\aug{\psi}^*_{k,\ell}(t,x) =\mathbbm{1}_{B_k}(x)\delta(t-t_l)$, $k = 1,\ldots,N$, $l=1,\ldots,M$. Putting all this into \eqref{eq:PetrovGalerkin}, we obtain the $(n,m),(k,l)$-th entry of the discrete augmented generator~$\aug{G}$,
\begin{equation}
\aug{G}_{(n,m),(k,l)} = -\ell_m'(t_l)\delta_{k,n} + G(t)_{n,k}\delta_{m,l}\,.
\label{eq:hybrid generator}
\end{equation}
%
%
The expected benefit of the hybrid scheme is that we can compute the generator(s) on~$X$ at various times instead of computing on the higher-dimensional space~$\aug{X} = \tau S^1\times X$. If the variation of the vector field in time is sufficiently smooth, then we expect high accuracy of the discretization already with a small number of time collocation points.

Just as in section~\ref{ssec:discrete spectrum Ulam}, we can search for ``companion eigenmodes'' in case of our hybrid discretization~\eqref{eq:hybrid generator}. If~$\aug{v}$ is an eigenfunction of~$\aug{G}$ for the eigenvalue~$\lambda$, then consider~$\aug{w}$ with~$w_t = v_t\psi_k(t)$, where~$\psi_k$ was defined in section~\ref{ssec:discrete spectrum Ulam}. The fact that the derivative of the functions~$\psi_k$, $-\tfrac{M-1}{2}\le k\le \tfrac{M-1}{2}$, is reproduced exactly by spectral differentiation implies together with~\eqref{eq:hybrid generator} that~$\aug{w}$ is an eigenfunction for eigenvalue~$\lambda - \lambda_k$, with~$\lambda_k = 2\pi i k\tau^{-1}$. This matches the analytical shift of the augmented generator, cf.~\eqref{eq:spectral_shift}.

\subsection{Example 1: a periodically driven double gyre flow}
\label{ssec:DG}

We consider the non-autonomous system~\cite{FrPa14}
\begin{equation}
\begin{aligned}
\dot x &= -\pi A \sin\left(\pi f(t,x)\right) \cos(\pi y) \\
\dot y &= \pi A \cos\left(\pi f(t,x)\right) \sin(\pi y)\frac{df}{dx}(t,x),
\end{aligned}
\label{eq:DoubleGyre}
\end{equation}
where $f(t,x)=\alpha \sin(\omega t)x^2+(1-2\alpha\sin(\omega t))x$. We fix the parameter values $A=0.25$, $\alpha=0.25$ and~$\omega=2\pi$, hence the vector field has period~1. The system preserves the Lebesgue measure on~$X = [0,2]\times [0,1]$. Explicit diffusion is added to the system (cf.~\eqref{eq:SDE}) with~$\ep = 0.1$ and reflecting boundary conditions. Equation~\eqref{eq:DoubleGyre} describes two counter-rotating gyres next to each other (the left one rotates clockwise), with the vertical boundary between the gyres oscillating periodically.

We consider the Ulam discretization of the augmented generator, $\aug{G}$, on a uniform~$30\times 100\times 50$ partition of the augmented state space~$\aug{X} = S^1\times X$. The first couple of eigenvalues with the largest real part are (as computed in Matlab by the command \texttt{eigs(G,11,'LR')})
\begin{Verbatim}[commandchars=\\\{\}]
  -0.0000 + 0.0000i
  -0.0832 + 0.0000i
  -0.3160 + 1.1437i
  -0.3160 - 1.1437i
  -0.3663 + 0.0000i
  \underline{-0.6556 + 6.2374i}
  \underline{-0.6556 - 6.2374i}
  \underline{-0.7362 + 6.2443i}
  \underline{-0.7362 - 6.2443i}
  -0.9609 + 7.3794i
  -0.9609 - 7.3794i
\end{Verbatim}
\begin{figure}[h]
\centering
\includegraphics[width=0.48\textwidth]{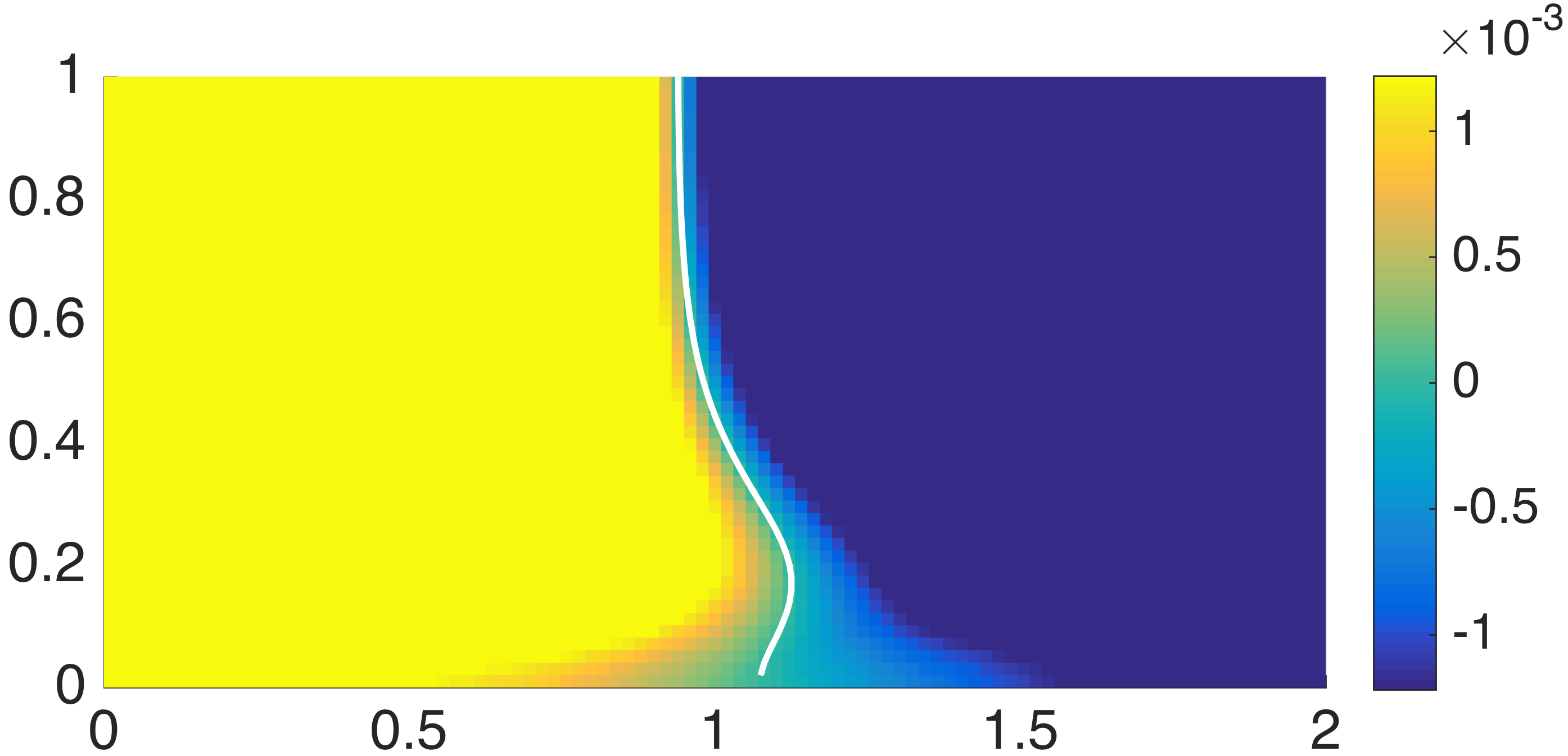}
\hfill
\includegraphics[width=0.48\textwidth]{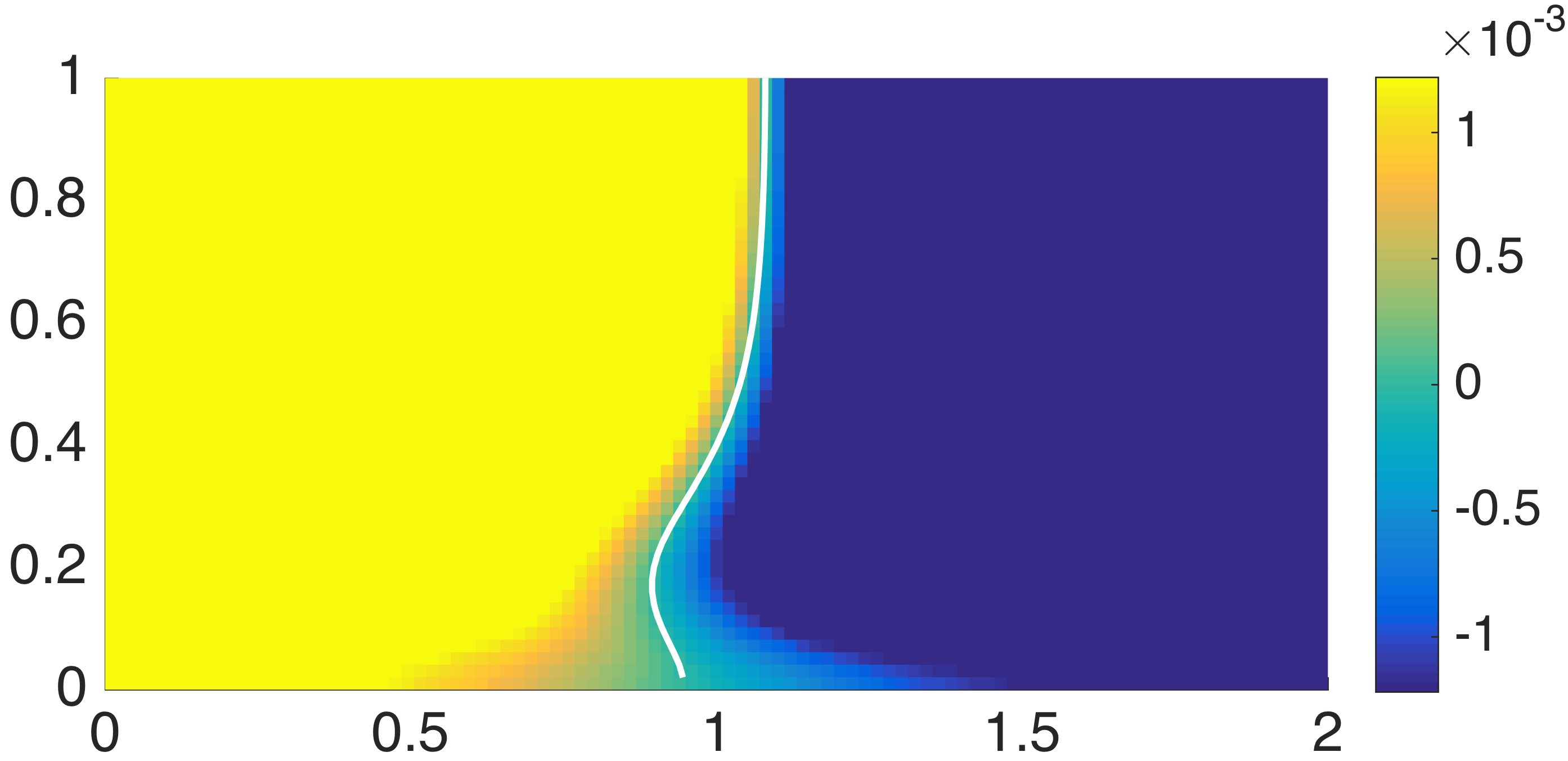}\\

\includegraphics[width=0.48\textwidth]{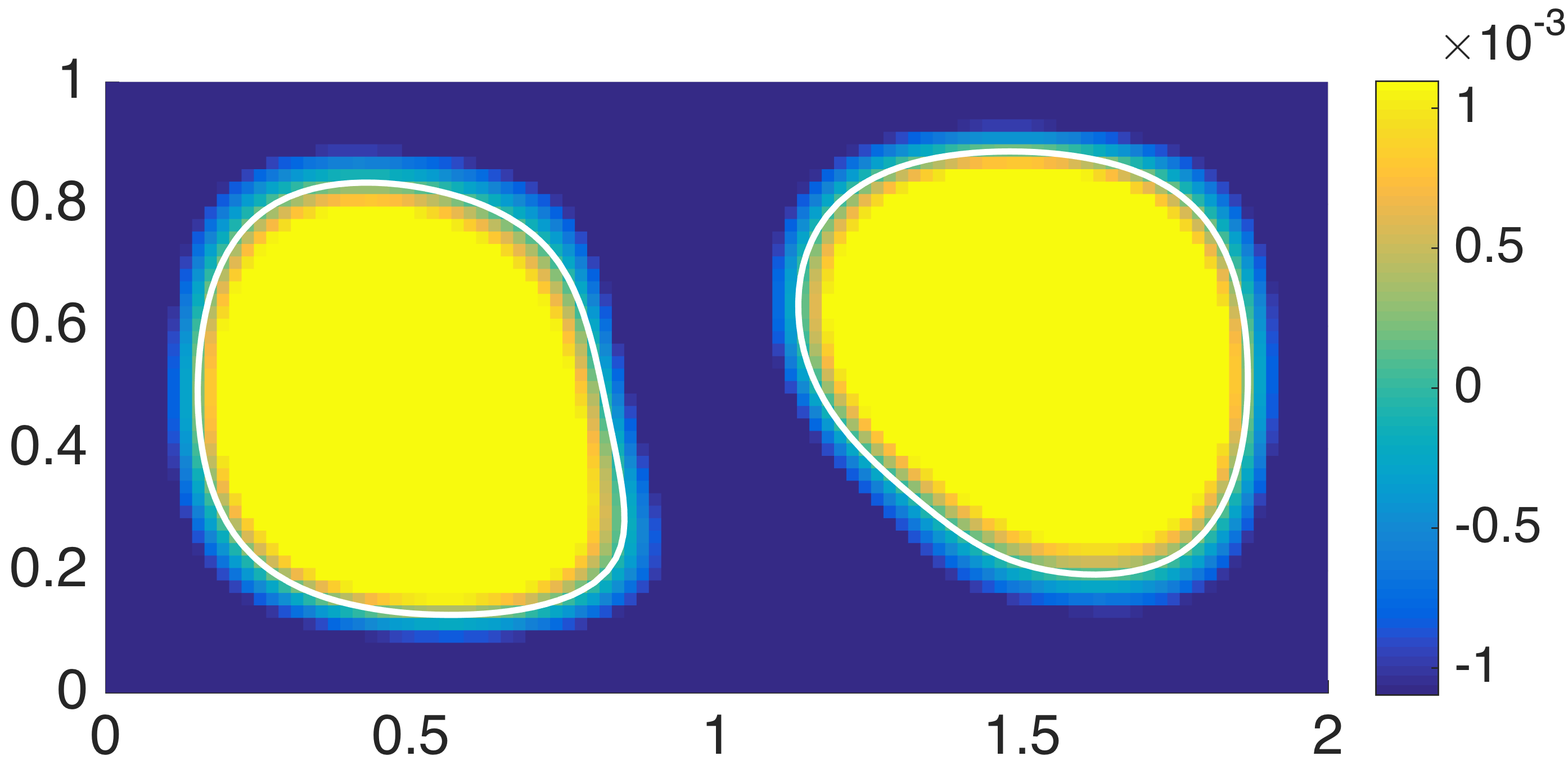}
\hfill
\includegraphics[width=0.48\textwidth]{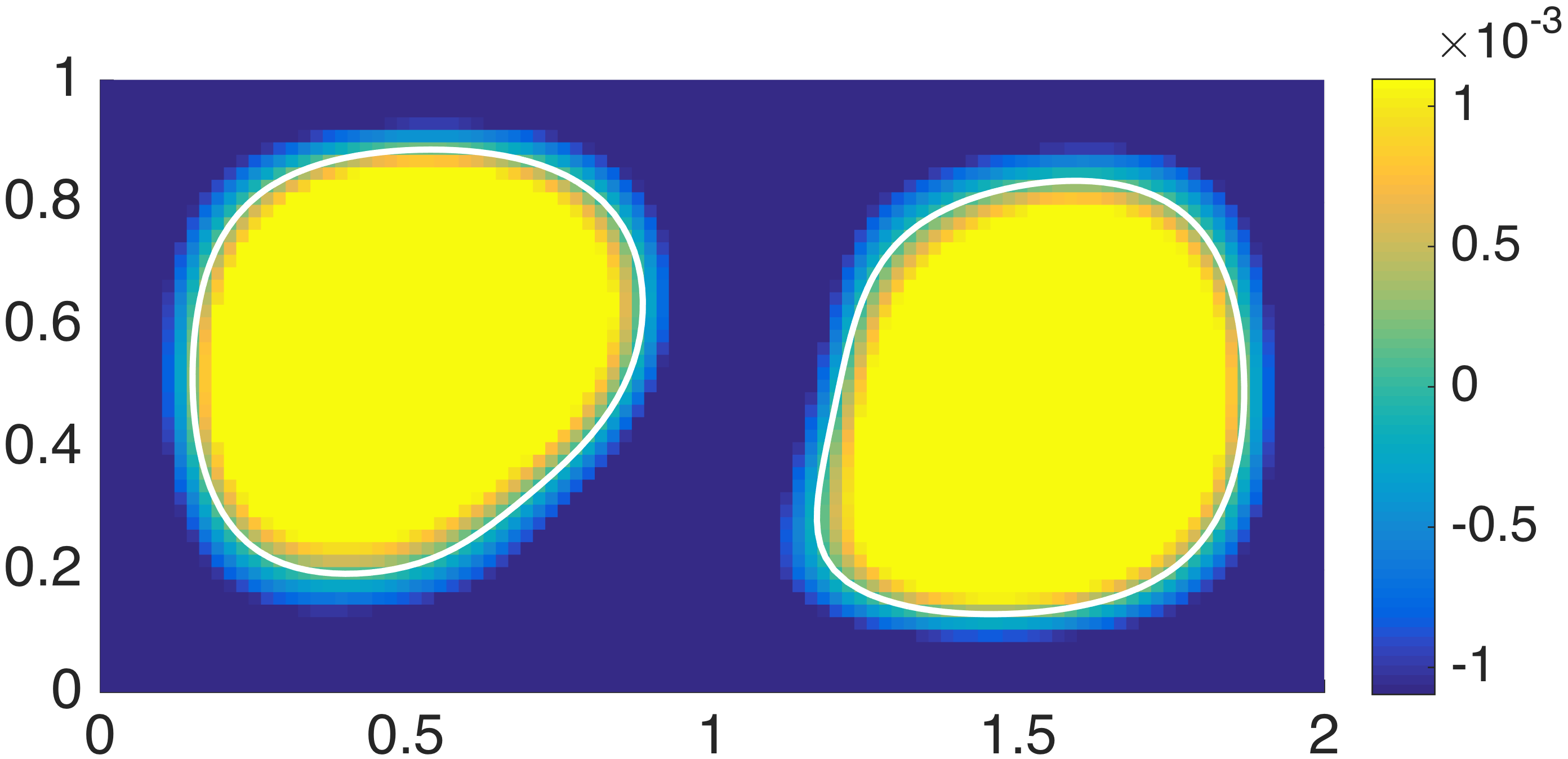}
\caption{Time slices of the eigenvectors for the largest nonzero real eigenvalues (top: largest real eigenvalue $-0.0832$, bottom: second largest real eigenvalue $-0.3663$) at time $t=0$ (left) and $t=0.5$ (right). The white contour indicates the zero-level curve.}
\label{fig:DoubleGyre}
\end{figure}
Since the constant density is invariant at every time instance, we expect~$-\lambda_k = -\tfrac{1-\exp(-2\pi i k/30)}{1/30}$ to be an eigenvalue for every~$k\in\Z$ with an eigenfunction which is constant in the spatial direction (cf.~section~\ref{ssec:discrete spectrum Ulam}). Indeed,~$-\lambda_{\pm 1}$ occurs among the computed eigenvalues, they are the top two underlined eigenvalues.

The second eigenvalue of~$\aug{G}$, about~$-0.0832$, indicates the most coherent family. The corresponding eigenvector~$\aug{v}$ is not constant in time, hence we cannot expect to find a companion eigenvalue of~$\aug{G}$ exactly at $-0.0832-\lambda_{\pm 1}= -0.7388 \pm 6.2374 i$, but we do anticipate an eigenvalue close by with an associated eigenvector close to~$\aug{v}\psi_{\pm 1}$ (the multiplication is meant pointwise for the temporal components). To check this, we compute correlations between the first 20 eigenvectors~$\aug{u}_n$, $n=1,\ldots,20$, and~$\aug{v}\psi_{\pm 1}$, i.e.
\[
c_n^{\pm} = \frac{\left\langle \aug{u}_n,\aug{v}\psi_{\pm 1}\right\rangle}{\|\aug{u}_n\|_2\,\|\aug{v}\psi_{\pm 1}\|_2},
\]
where~$\|\aug{u}\|_2$ is the standard Euclidean norm of the vector~$\aug{u}\in \R^{30\cdot 100\cdot 50}$.
The eigenvectors associated with the bottom two underlined eigenvalues, $-0.7362 \pm 6.2443 i$, yield correlations~$c_{5,6}^{\pm} \ge 0.99$, meanwhile the correlation with the other eigenfunctions does not exceed~$0.02$.

Figure~\ref{fig:DoubleGyre} shows two time slices for the first two subdominant eigenfunctions for real eigenvalues. The top row suggests that the most coherent splitting approximately separates the left and the right gyres from one another, while the bottom row illustrates the coherence of the cores of the gyres. Note, that the zero-level curve of the dominant eigenfunction is a smoothed version of the unstable manifold of the time-independent hyperbolic fixed point~$(1,1)$ of the unperturbed flow. This observation is consistent with, but not identical to, the results in~\cite{FrPa14}.
There, the computations were on a finite time interval, and the boundary between the dominant finite-time coherent sets was close to the stable (resp.\ unstable) manifold of the hyperbolic periodic point at the initial (resp.\ final) time.  Here, we consider coherence over an infinite time (as in~\cite{FrLlSa10}) and the boundaries are always approximately aligned with the \emph{unstable} manifold of the hyperbolic periodic point.

Between the first two real subdominant eigenvalues we find a pair of complex eigenvalues,~$-0.3160 \pm 1.1437i$, indicating a coherent family as well. Let~$\aug{u}$ be the eigenfunction for the third eigenvalue~$\mu = \alpha + \beta i = -0.3160 + 1.1437i$; we extract coherent families from it, as described in Remark~\ref{rem:complex}. To obtain the zero phase coherent families, we set
\[
A^{\pm}_t = \{\pm \P_{0,t}u_0^{\rm Re}\ge 0\} = \left\{\mathrm{Re}(\pm e^{i\beta t}u_t) \ge 0\right\}\,.
\]
Remark~\ref{rem:complex} and Theorem~\ref{thm:nonauto_escrate} show that the escape rate from this family is smaller than~$0.3160$. Figure~\ref{fig:DoubleGyreComplex} shows this coherent family at four different times.
\begin{figure}[htb]
\centering
\includegraphics[width=0.48\textwidth]{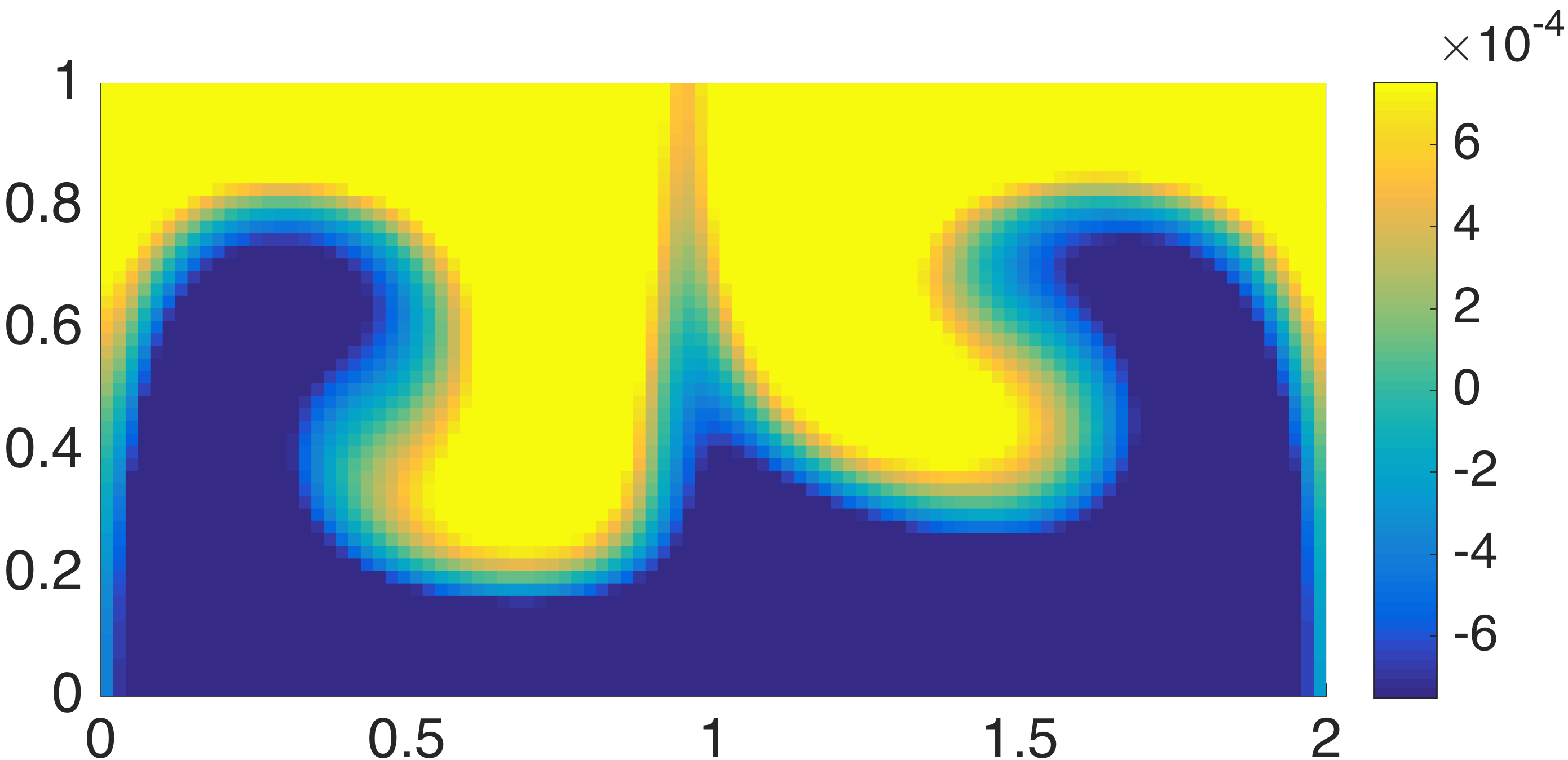}
\hfill
\includegraphics[width=0.48\textwidth]{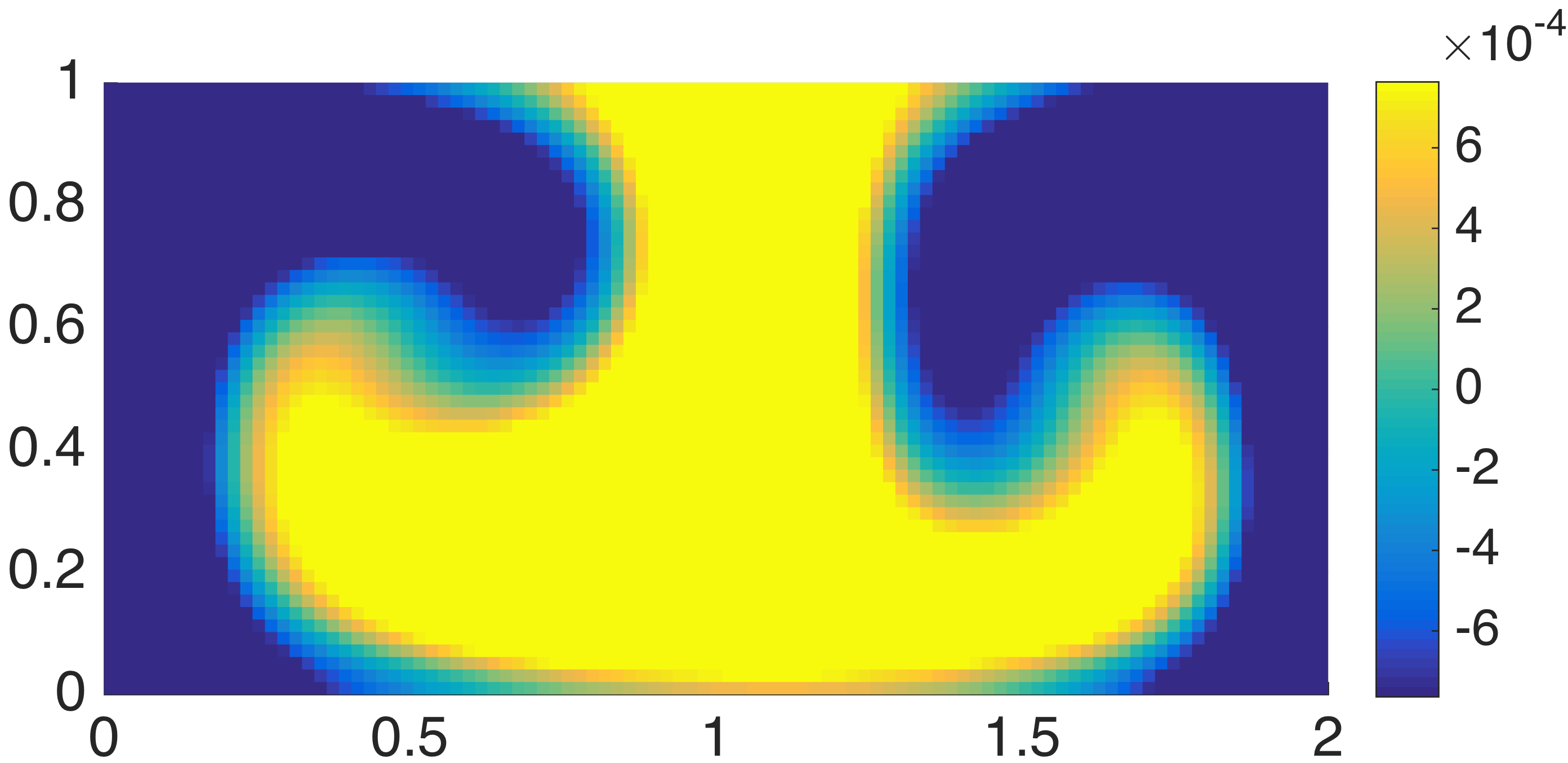}\\

\includegraphics[width=0.48\textwidth]{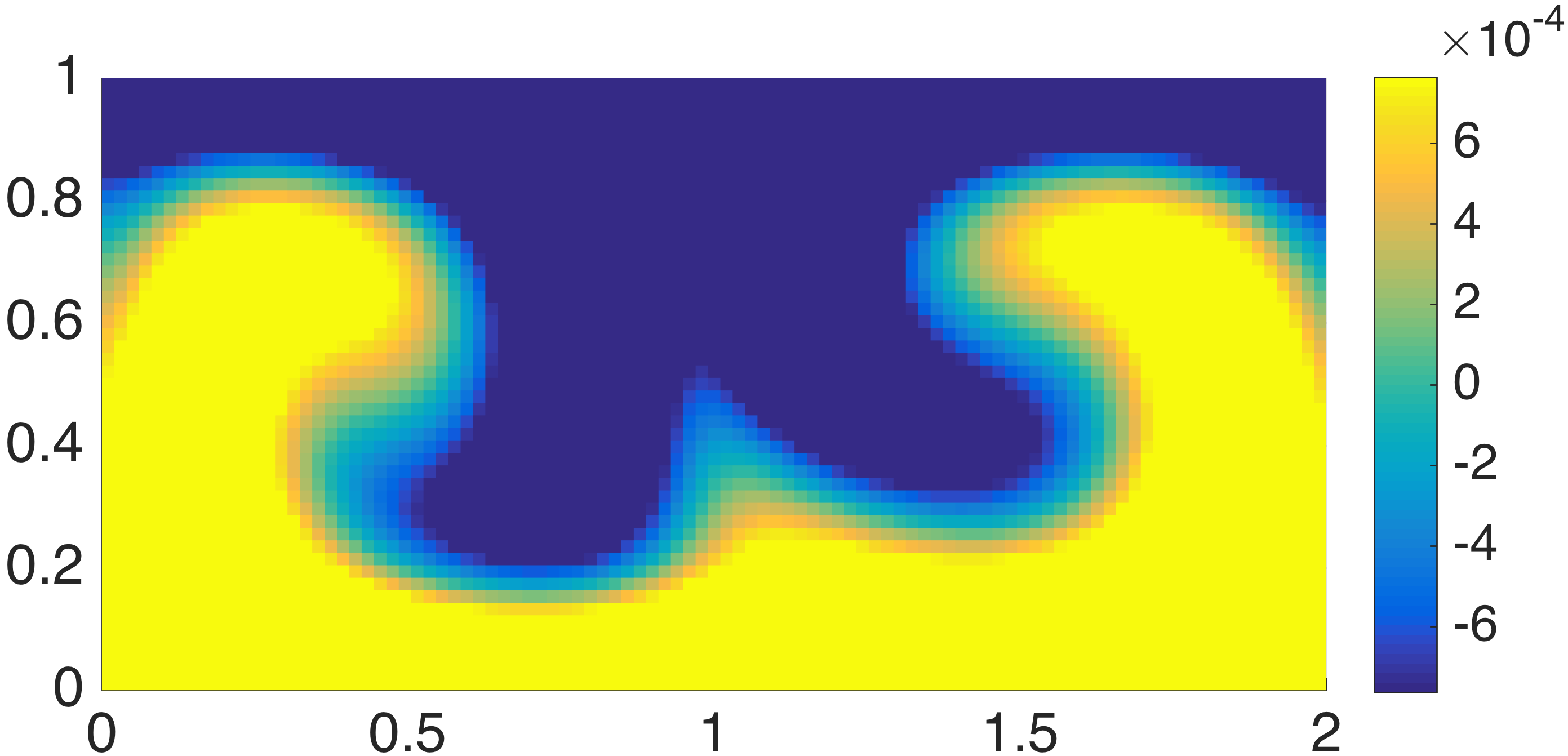}
\hfill
\includegraphics[width=0.48\textwidth]{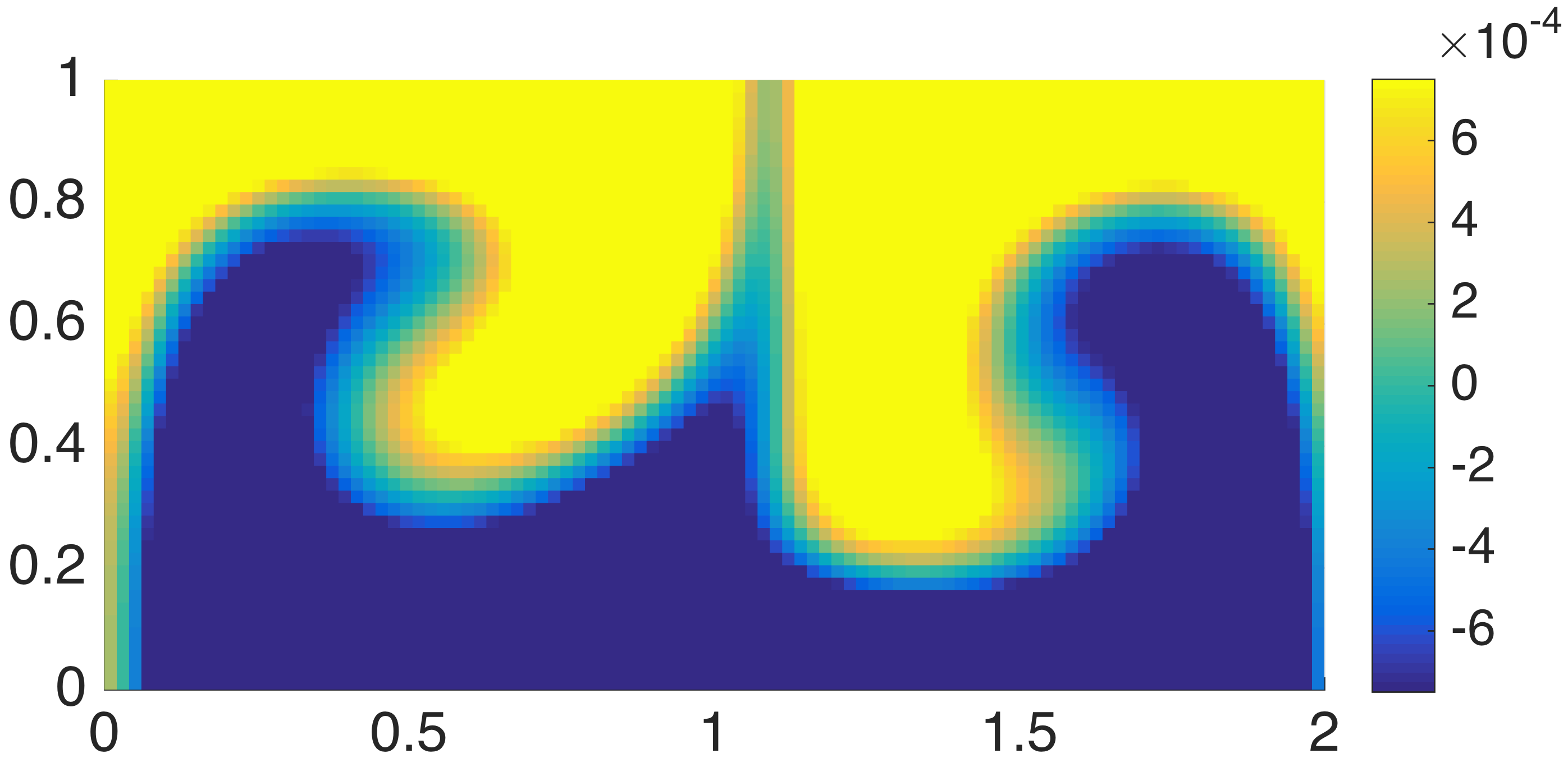}
\caption{The functions~$\mathrm{Re}(e^{i\beta t}u_t)$, indicating the coherent families, for four different times~$t=0,\,1.5,\,3,\,5.46$ (top left, top right, bottom left, and bottom right, respectively).}
\label{fig:DoubleGyreComplex}
\end{figure}
The period of the phase of the family is~$2\pi/\beta \approx 5.49$, and we see in Figure~\ref{fig:DoubleGyreComplex} that after this time the coherent sets seem to have approximately completed one revolution around the centers of the gyres. This time approximately corresponds to the period of trajectories starting on the zero-level curves at the bottom of Figure~\ref{fig:DoubleGyre} around the centers of the associated gyres (this is verified by numerical simulation; not shown here).

Next, we test Theorem~\ref{thm:nonauto_escrate} numerically. For this we choose~50000 uniformly distributed random points in~$X$, of which~25050 lie in the positive support of the~$t=0$ slice of the eigenfunction~$\aug{v}$ for eigenvalue~$-0.0832$; i.e.~in~$\{v_0\ge 0\}$. We integrate all points from~$s=0$ to~$t=10$ (that is, 10 periods) with the Euler--Maruyama scheme (stepsize $1/30$, reflecting boundary conditions). At the end we estimate the escape rate from the family~$\{A_r^{\pm}\}_{r\in S^1}$ by
\[
\hat{E}(\{A_r^{\pm}\}) = -\frac{1}{t-s}\log\left(\frac{\# \text{points that stayed in }A_{r\ \mathrm{mod}\ 1}^{\pm} \text{ for }r=0,\tfrac{1}{30},\tfrac{2}{30},\ldots,10}{\# \text{points that started in } A_0^{\pm}}\right)\,.
\]
After averaging over five runs, we obtain~$\hat{E}(\{A_r^+\}) = 0.0657$ and~$\hat{E}(\{A_r^-\}) = 0.0645$, both being smaller than the associated bound by the eigenvalue~$-\lambda = 0.0832$. Figure~\ref{fig:DoubleGyreEscape1} shows the fraction of ``surviving points'' in time for different realizations of the noise process, with the theoretical bound for the slope given by the corresponding eigenvalue.
Figure~\ref{fig:DoubleGyreEscape2} shows the dynamical evolution of points starting in the coherent families~$\{A_r^{\pm}\}$ for two different times.
\begin{figure}[htb]
\centering
\includegraphics[width=0.6\textwidth]{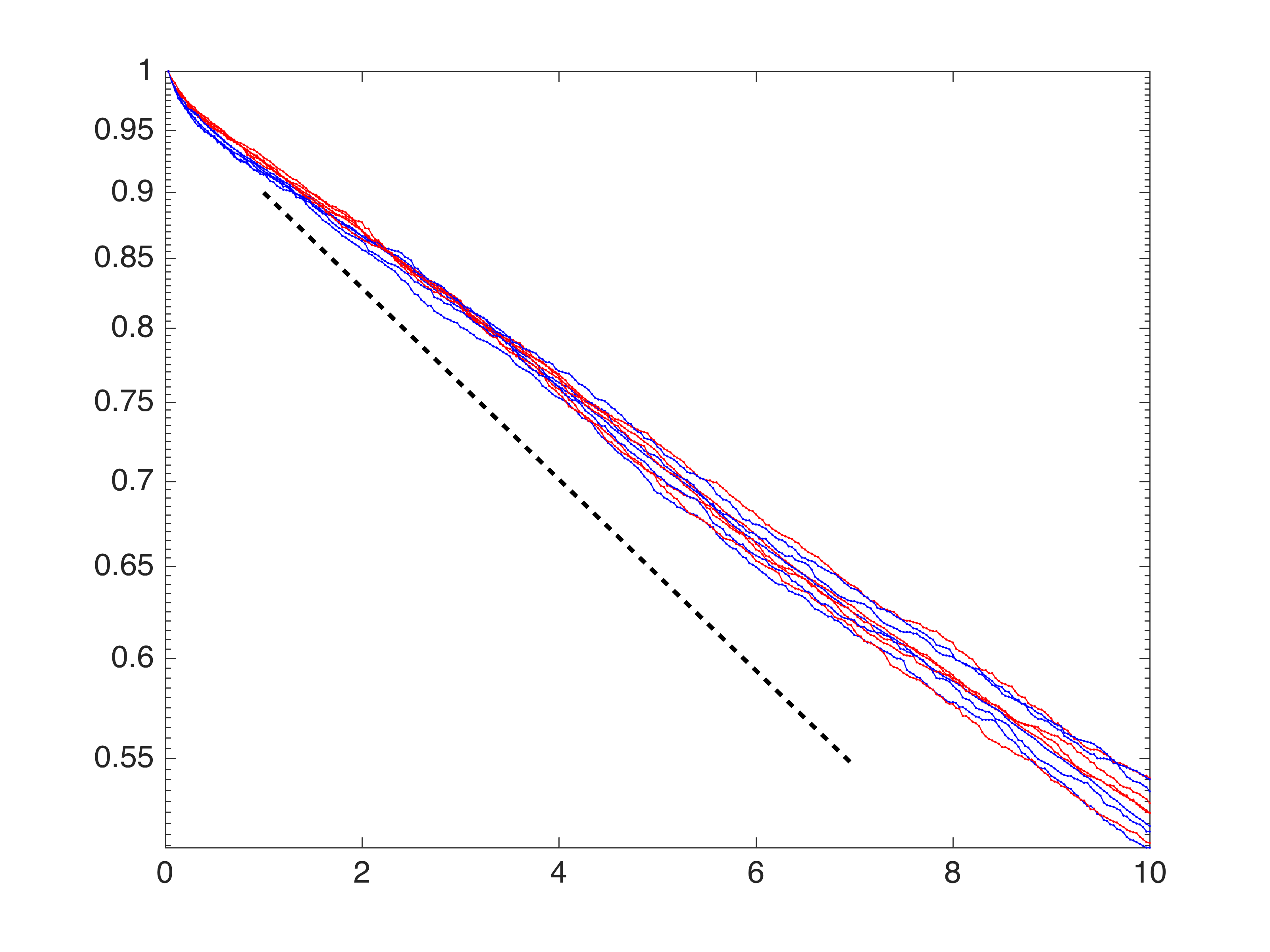}

\caption{Numerical simulation of escape rates at anchor time~$s=0$. We show the fraction of points that stay in the coherent family until time~$t$ versus~$t$ for five test runs. Blue curves correspond to escape rates from~$\{A_r^+\}$, red to those from~$\{A_r^-\}$. The dashed line has slope given by the second eigenvalue.}
\label{fig:DoubleGyreEscape1}
\end{figure}
\begin{figure}[htb]
\centering
\includegraphics[width=0.48\textwidth]{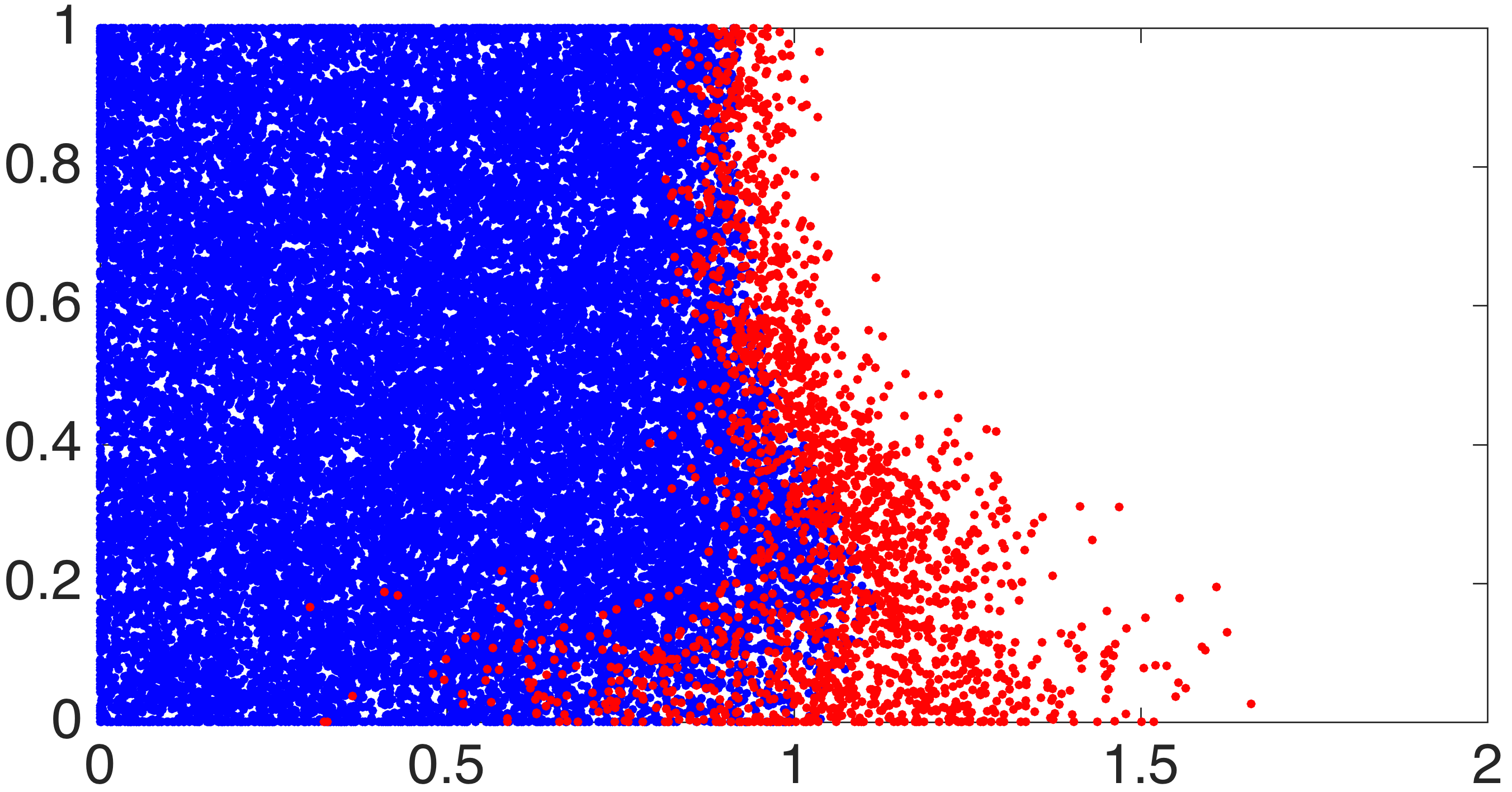}
\hfill
\includegraphics[width=0.48\textwidth]{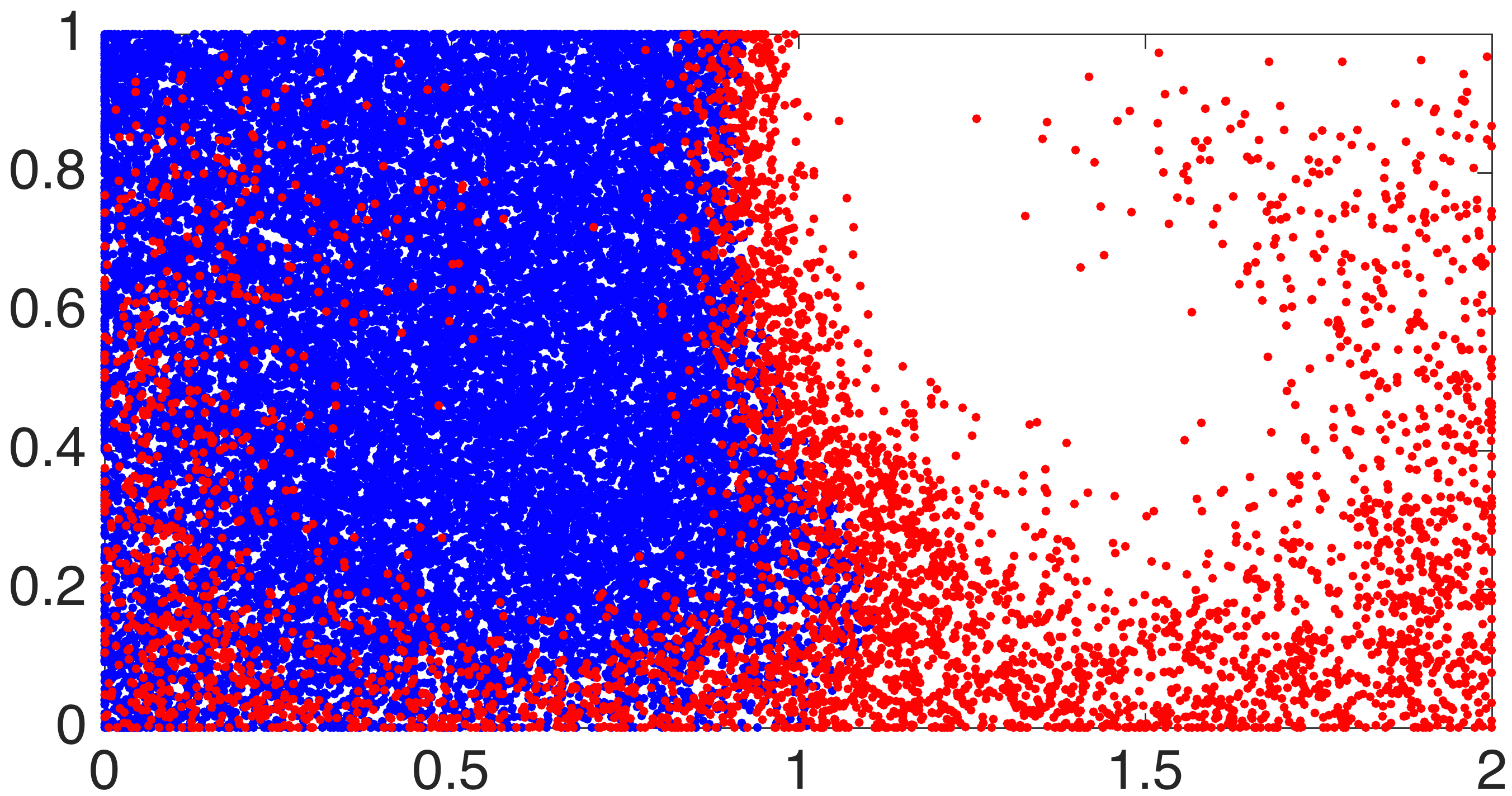}\\
\includegraphics[width=0.48\textwidth]{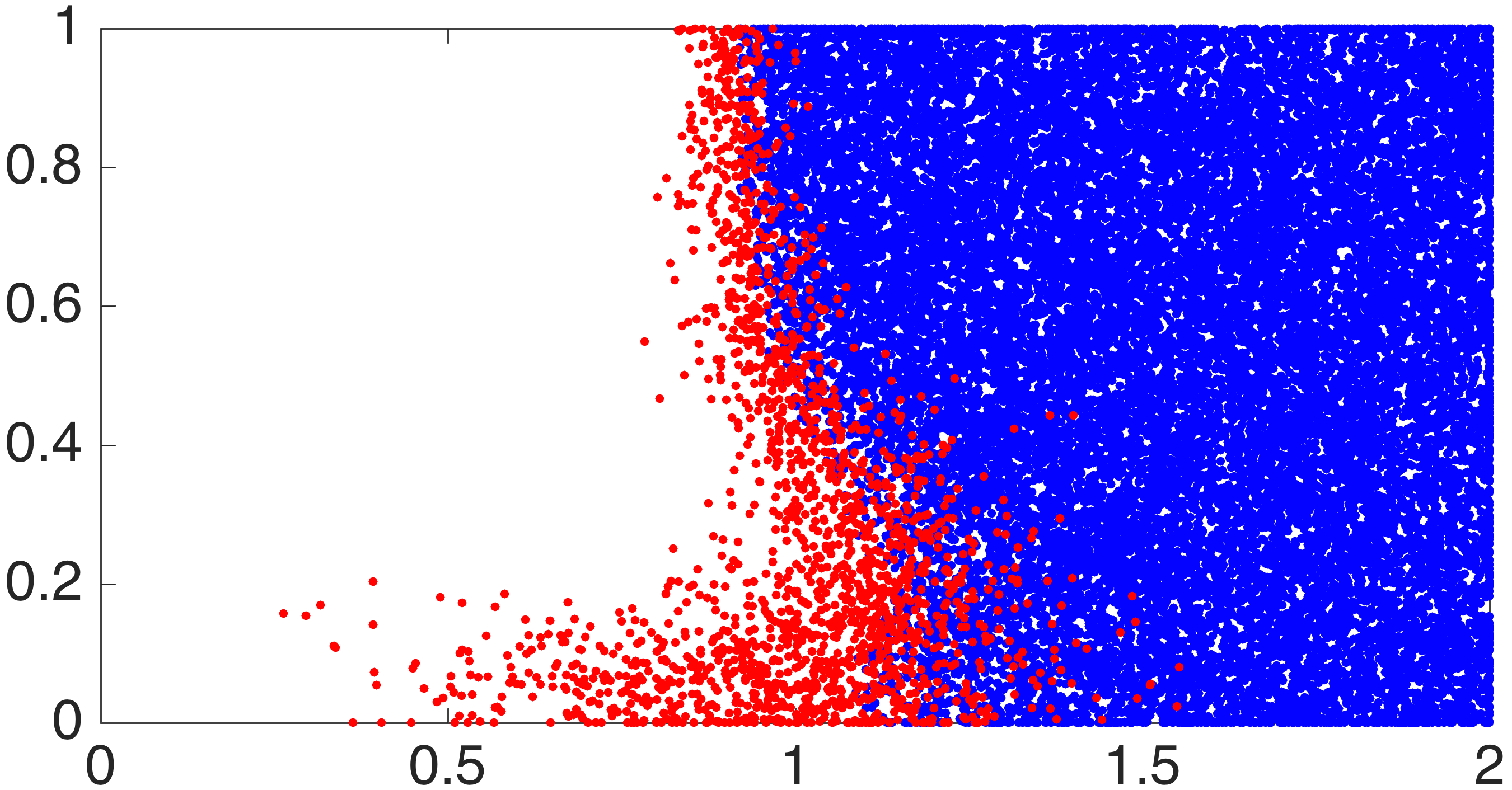}
\hfill
\includegraphics[width=0.48\textwidth]{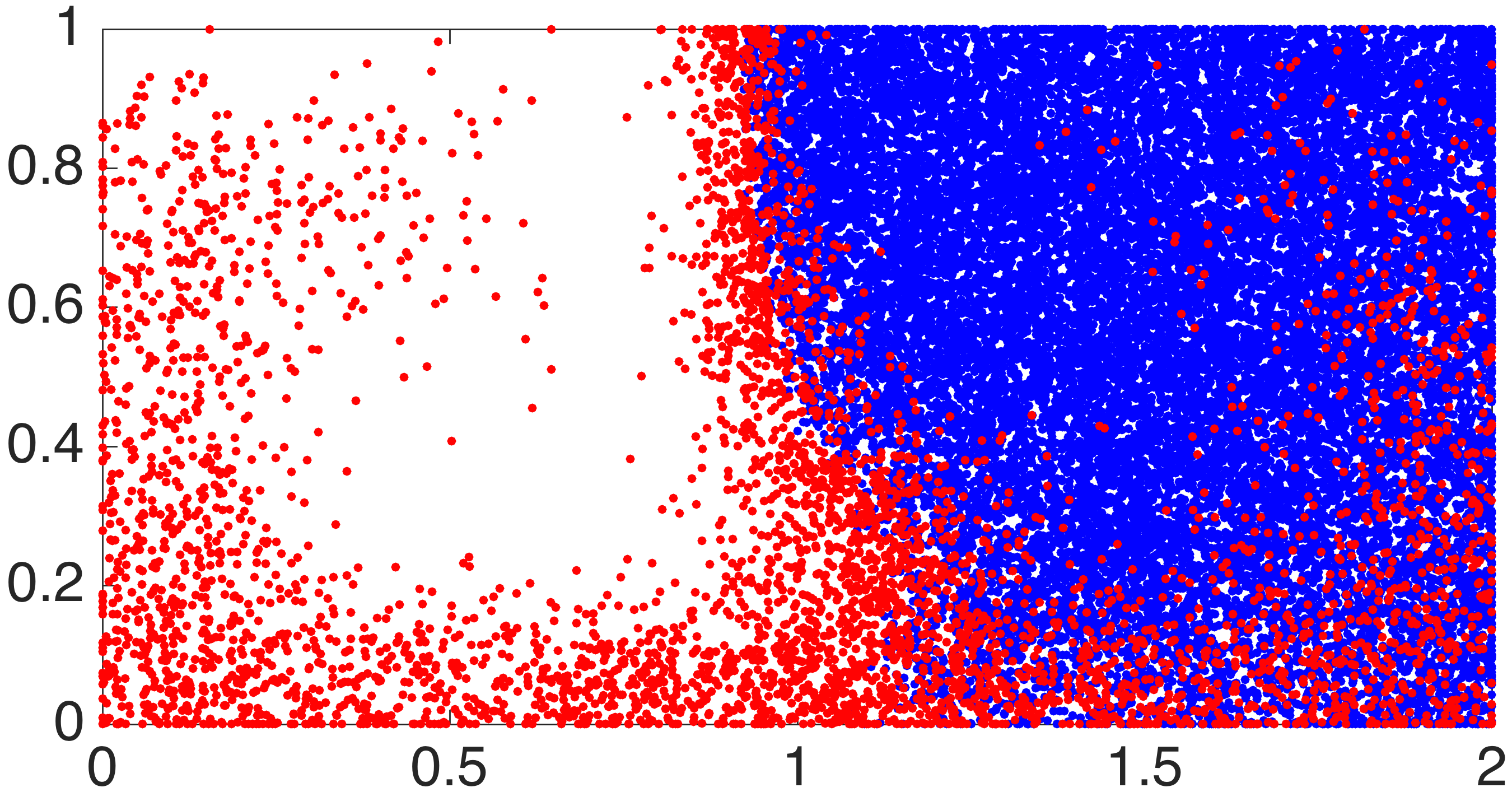}
\caption{Numerical simulation of escape rates at anchor time~$s=0$. The points colored blue stayed in the coherent family~$\{A_r^{\pm}\}$ for all times until time~$t=1$ (left) and~$t=3$ (right). The points that left the family at least once until time~$t$ are colored red. Top: coherent family~$\{A_r^+\}$, bottom: coherent family~$\{A_r^-\}$.}
\label{fig:DoubleGyreEscape2}
\end{figure}

\begin{remark}
In~\cite{froyland_padberg_09,FrPa14} a line-search along different level sets is used to obtain optimal finite-time coherent sets. In the current setup we can employ a similar procedure. Recall from Theorem~\ref{thm:1=2} and expression~\eqref{eq:instoutflow} that the cumulative outflow flux from the coherent family over one time period is given by the instantaneous outflow flux from the augmented set in augmented space. The entries of Ulam's discretization of the augmented generator are exactly box-to-box instantaneous flow rates, cf.~section~\ref{ssec:GeneratorUlam}. Hence, for any sub- or superlevel set of a given eigenfunction, the cumulative outflow flux can be computed from the discrete augmented generator~$\aug{G}$ in complexity that is linear in the number of boxes. Then, one would search for a level set optimizing the cumulative outflow flux to augmented volume ratio.
\end{remark}

\subsection{Example 2: periodically perturbed Bickley jet}

We consider a perturbed Bickley jet as described in~\cite{RypEtAl07}. This is an idealized zonal jet approximation in a band around a fixed latitude, assuming incompressibility, on which two traveling Rossby waves are superimposed.
The dynamics is given by
\[
\dot x = -\frac{\partial\Psi}{\partial y},\quad \dot y = \frac{\partial\Psi}{\partial x},
\]
with stream function
\[
\Psi(t,x,y) = -U_0 L \tanh\big(\frac{y}{L}\big)
 + U_0L\,\mathrm{sech}^2\big(\frac{y}{L}\big) \sum_{n=2}^3 A_n\cos\left(k_n\left(x- c_n t\right)\right)\,.
\]
The constants are chosen as in~\cite{RypEtAl07}, the length unit is Mm (1 Mm = $10^6$ m), the time unit is days. Then
\[
U_0 = 5.4138,\quad L = 1.77, \quad A_2 = 0.1, \quad A_3 = 0.3\,.
\]
We set~$k_n = 2n/r_e$ with~$r_e = 6.371$. The phase speeds~$c_n$ of the Rossby waves are modified slightly so that the forcing is periodic with the smallest common period~$\tau = 9\, \rm days$. We choose~$c_2 = 0.2054\, U_0$ and $c_3 = 0.4108\, U_0$. The state space is periodic in the~$x$ coordinate, and is given by~$X = \pi r_e S^1\times [-4,4]$.

Compared with our previous example, the spatial scale of dynamics is finer, and the temporal change in the vector field includes higher frequencies. Thus, we need higher resolutions for an increased accuracy, hence we employ the hybrid discretization of the augmented generator from section~\ref{ssec:hybrid} with~$\ep=0.1$. We resolve time with~$21$ Fourier modes (such that frequencies from~$-10$ to~$10$ are present), and space with a~$300\times 120$ uniform grid (resulting in almost square boxes).

Since the hybrid discretization allows for purely imaginary ``companion'' eigenvalues to occur by the shifts~$\lambda_k= \lambda_1 k = 0.698 k\,i$ (see the end of section~\ref{ssec:hybrid}), we do not compute the eigenvalues with the largest real part, but those with the smallest magnitude.\footnote{One could ask, what happens if the real eigenvalues we are interested in have larger modulus than the smallest companion eigenvalues? This would mean~$\lambda_2 < -2\pi/\tau$, and the corresponding coherence estimate would be quite miserable by yielding a survivor fraction of~$e^{-2\pi}\approx 0.002$ after one time period. In both of our examples this fraction is around~$0.9$. Note that rescaling time does not help, because this merely multiplies both the eigenvalues of the augmented generator and the companion shifts by the same scalar with which we sped up time. We leave the computational details to the reader.}
Calling \texttt{eigs(G,20,'SM')} the following eigenvalues, which we order according to their real parts.
\begin{Verbatim}[commandchars=\\\{\}]
   0.0000 + 0.0000i
  \fbox{-0.0138 + 0.0000i}
  \fbox{-0.0303 + 0.0000i}
  -0.0332 - 0.1131i
  -0.0332 + 0.1131i
  -0.0333 - 0.1131i
  -0.0333 + 0.1131i
  -0.0443 + 0.0000i
  -0.0671 - 0.0021i
  -0.0671 + 0.0021i
  -0.0684 - 0.0030i
  -0.0684 + 0.0030i
  \fbox{-0.1179 + 0.0000i}
  -0.1322 + 0.0000i
  -0.1431 - 0.1310i
  -0.1703 - 0.0461i
  -0.1703 + 0.0461i
  -0.1703 - 0.0461i
  -0.1703 + 0.0461i
  -0.1864 + 0.0000i
\end{Verbatim}
By comparing the shifts with the imaginary parts of the eigenvalues, we do not expect to have any companion eigenvalues in this list. Figure~\ref{fig:Bickley1} shows the~$t=0,1,2$ (top to bottom) time slices of the eigenfunctions for the framed eigenvalues. As time increases, the general pattern shown by the eigenfunctions shift from left to right.
\begin{figure}[htb]
\centering
\includegraphics[width = 0.32\textwidth]{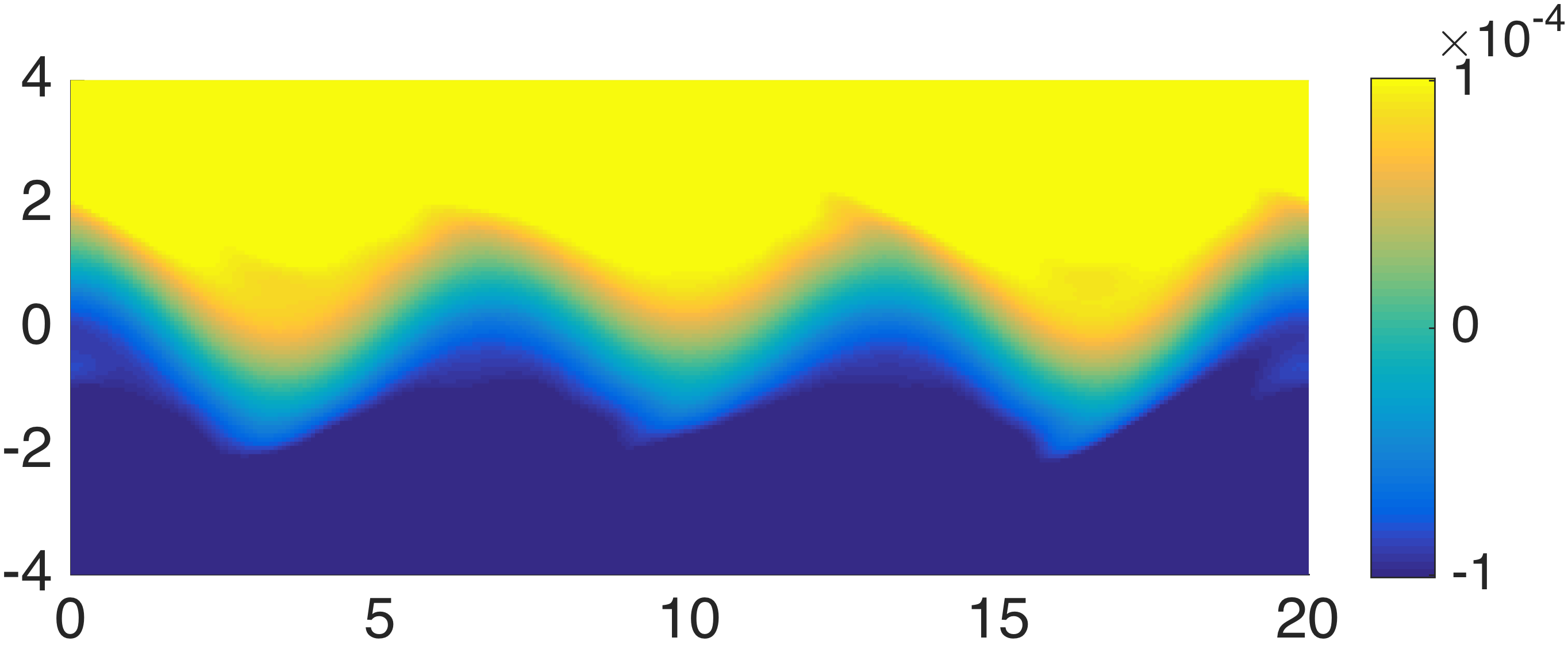} \hfill \includegraphics[width = 0.32\textwidth]{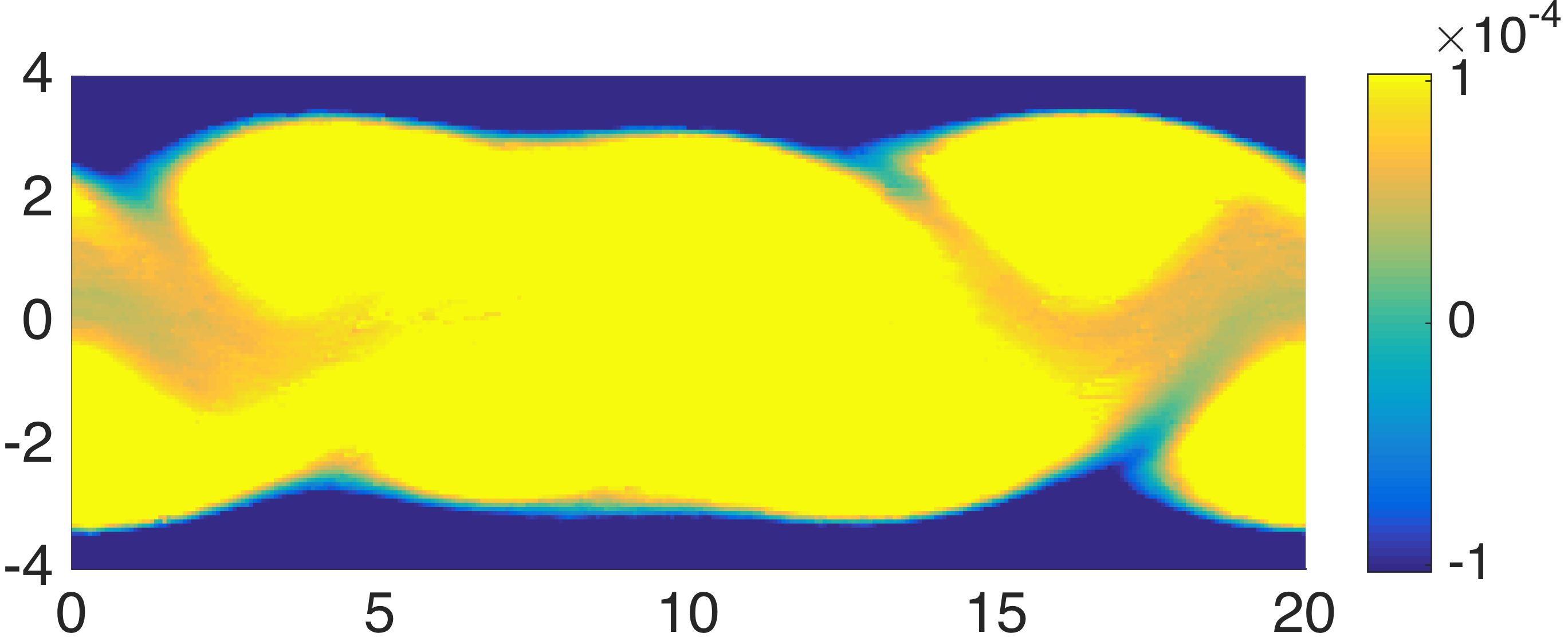} \hfill \includegraphics[width = 0.32\textwidth]{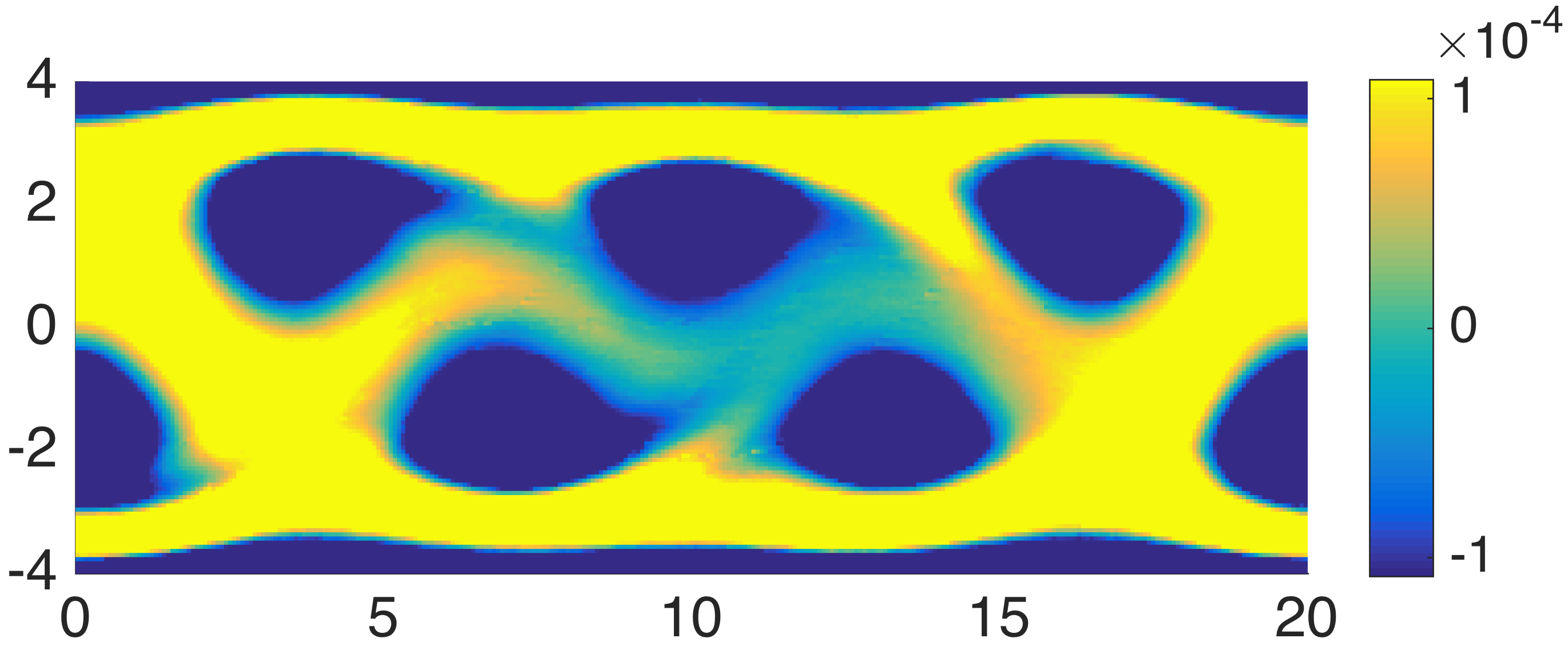}\\

\includegraphics[width = 0.32\textwidth]{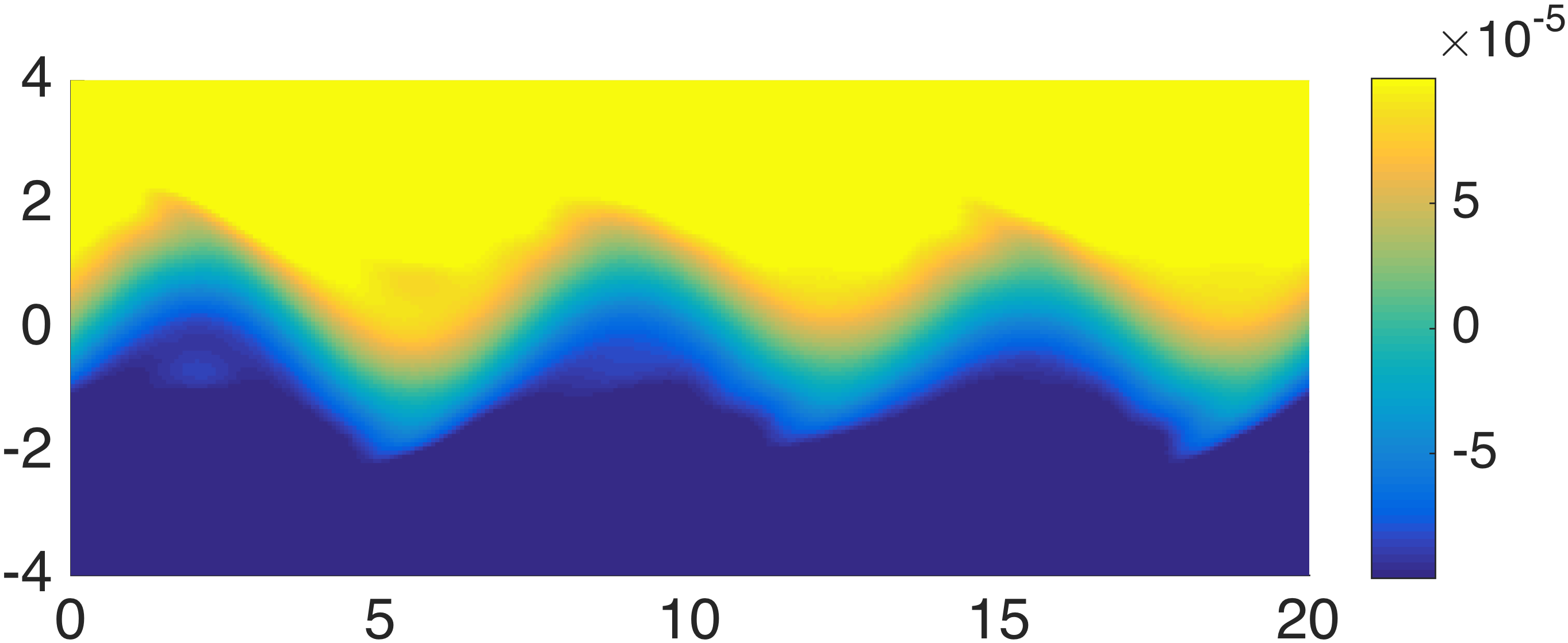} \hfill \includegraphics[width = 0.32\textwidth]{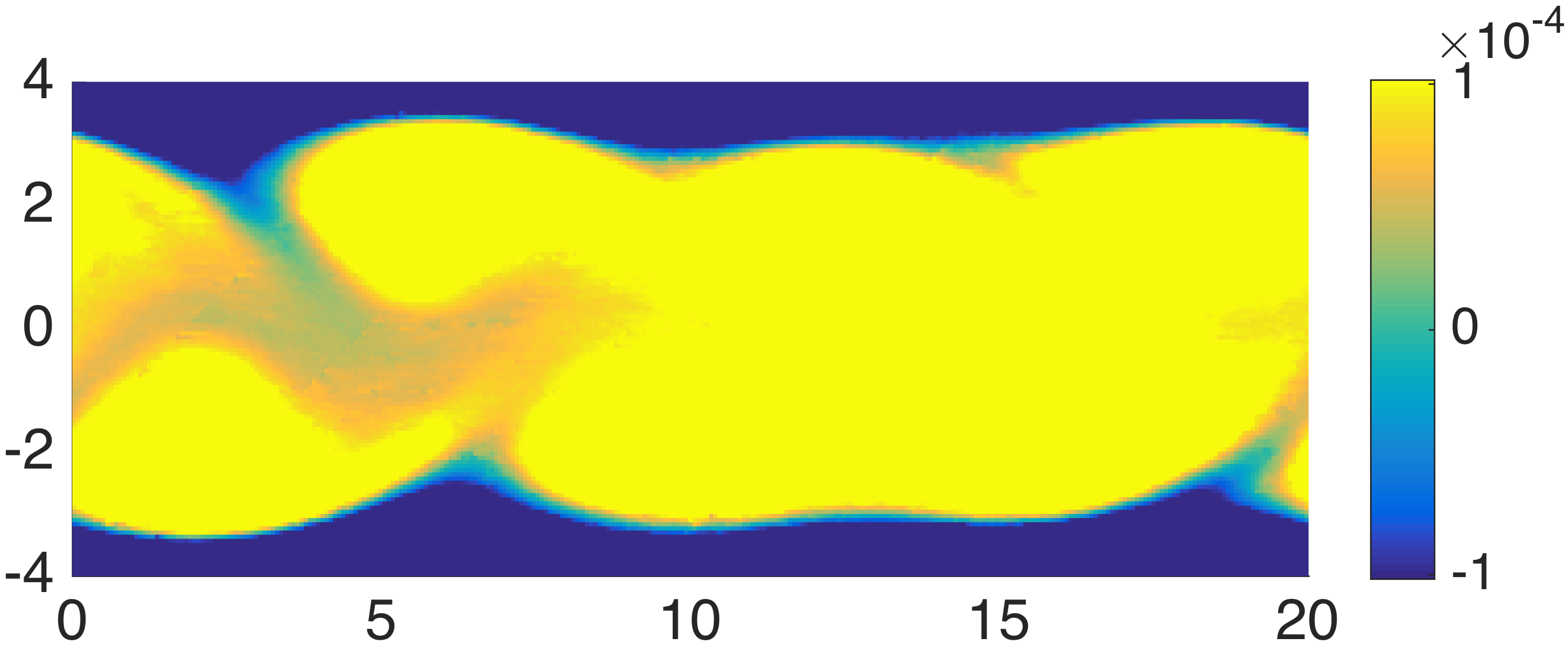} \hfill \includegraphics[width = 0.32\textwidth]{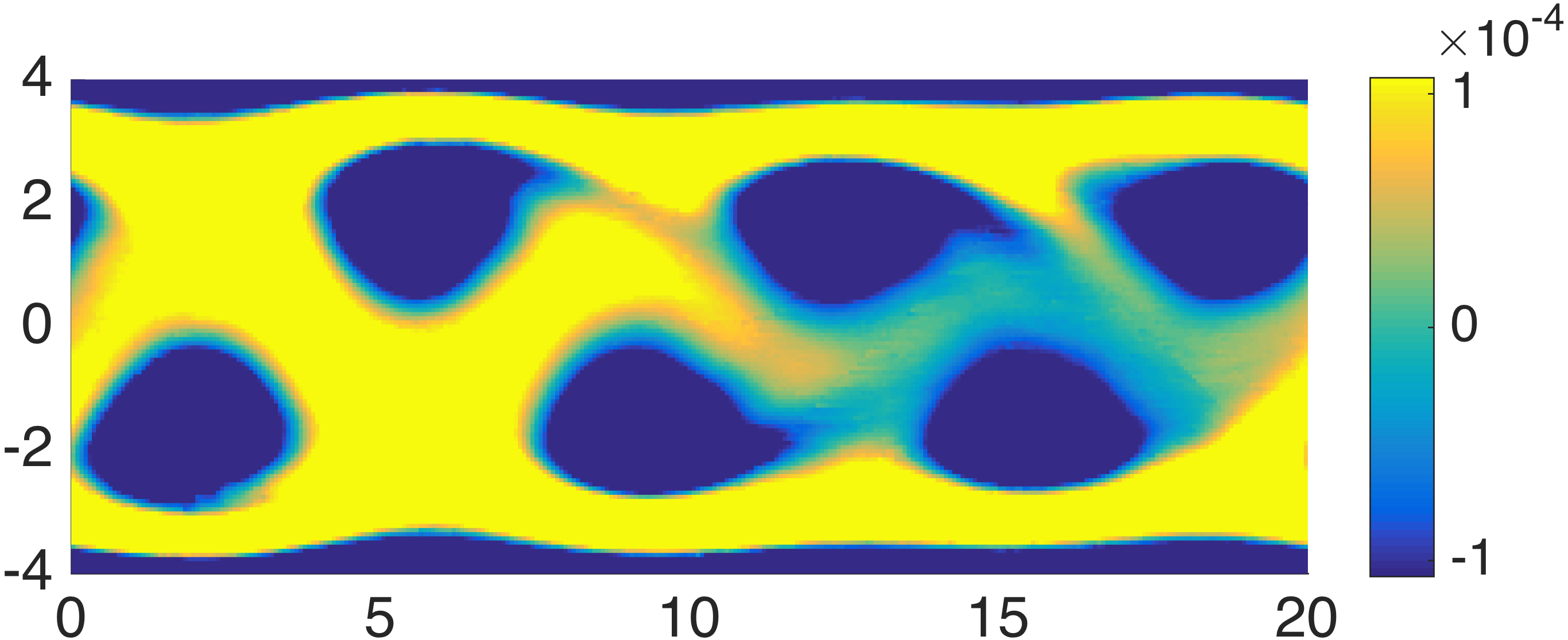}\\

\includegraphics[width = 0.32\textwidth]{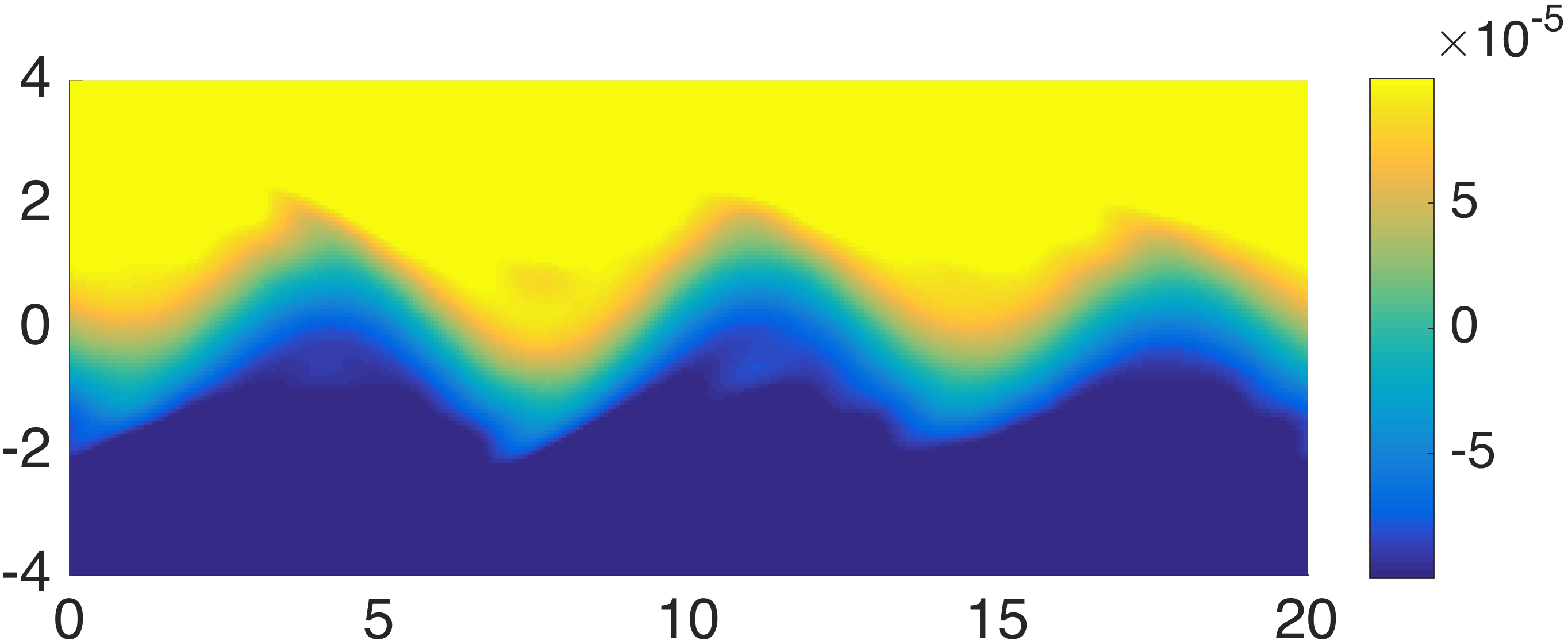} \hfill \includegraphics[width = 0.32\textwidth]{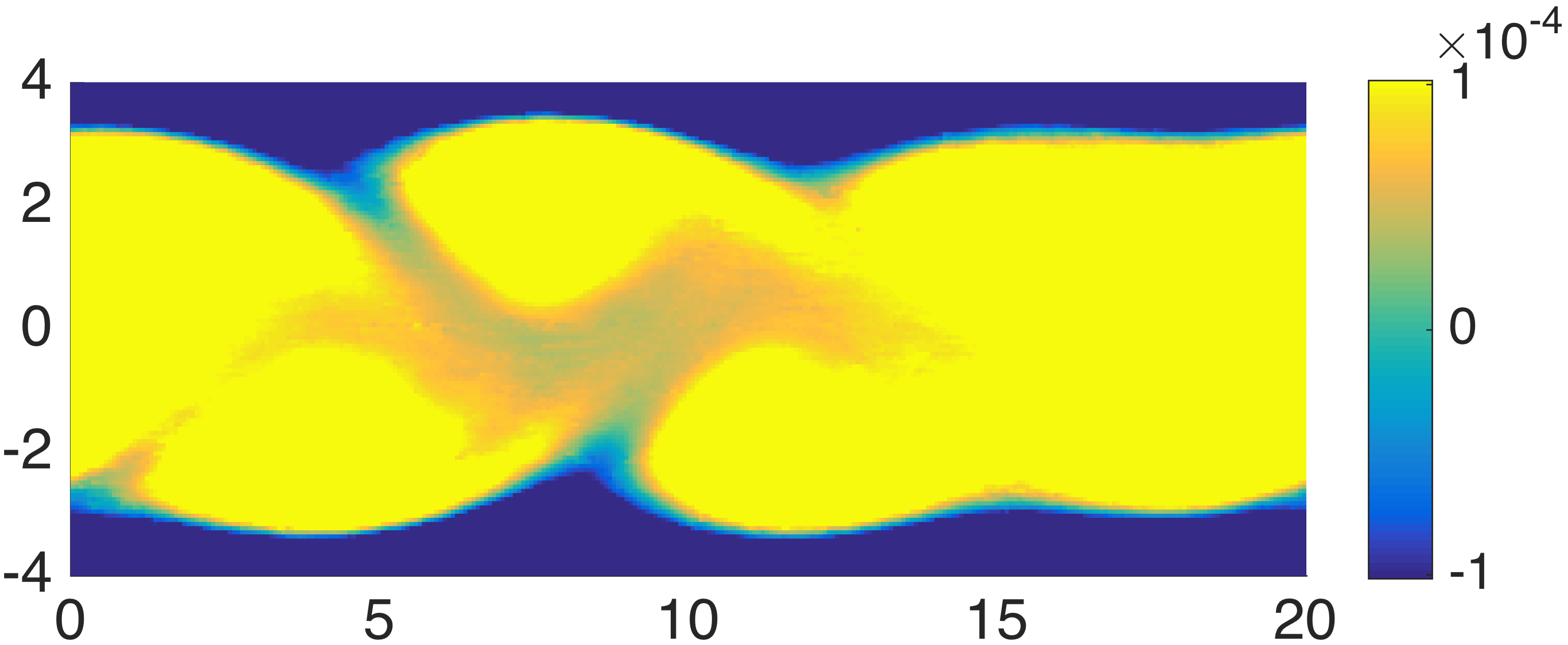} \hfill \includegraphics[width = 0.32\textwidth]{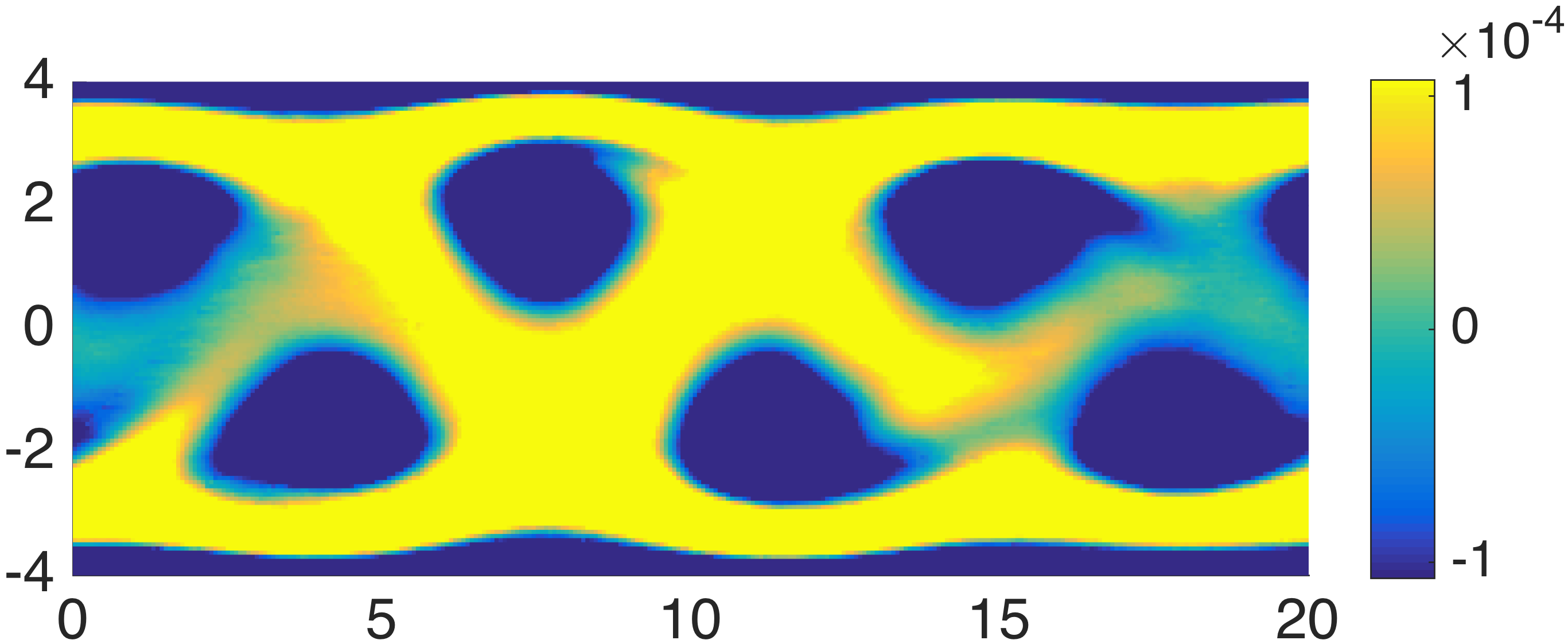}
\caption{Eigenfunctions of the hybrid discretized augmented generator at the second, third, and thirteenth eigenvalues (left to right), shown at time slices~$t=0, 1, 2$ (top to bottom).}
\label{fig:Bickley1}
\end{figure}
The eigenvalue bounds on the escape rates together with the corresponding eigenfunctions suggest that the least transport occurs across the meandering horizontal jet region around~$y=0$ ($2$nd eigenfunction, left column, Figure~\ref{fig:Bickley1}), while there is also low transport
between the extreme northern/southern boundaries, coloured dark blue, and the central horizontal region, coloured yellow ($3$rd eigenfunction, middle column, Figure~\ref{fig:Bickley1}). We also find six coherent vortices, coloured dark blue ($13$th eigenfunction, right column, Figure~\ref{fig:Bickley1}); these have a higher escape rate than the regions shown in the left and central columns of Figure~\ref{fig:Bickley1}.
Sampling-based numerical simulation of the escape rate from the most coherent family (i.e.\ that given by the eigenfunction in the left column of Figure~\ref{fig:Bickley1}; not shown) reveals escape rates of below~$0.0082$, which is again overestimated by the eigenvalue bound~$0.0138$, consistent with the theory. Figure~\ref{fig:BickleyVortex} shows the evolution of sample points starting from the top left vortex. 
\begin{figure}[htb]
\centering
\includegraphics[width = 0.45\textwidth]{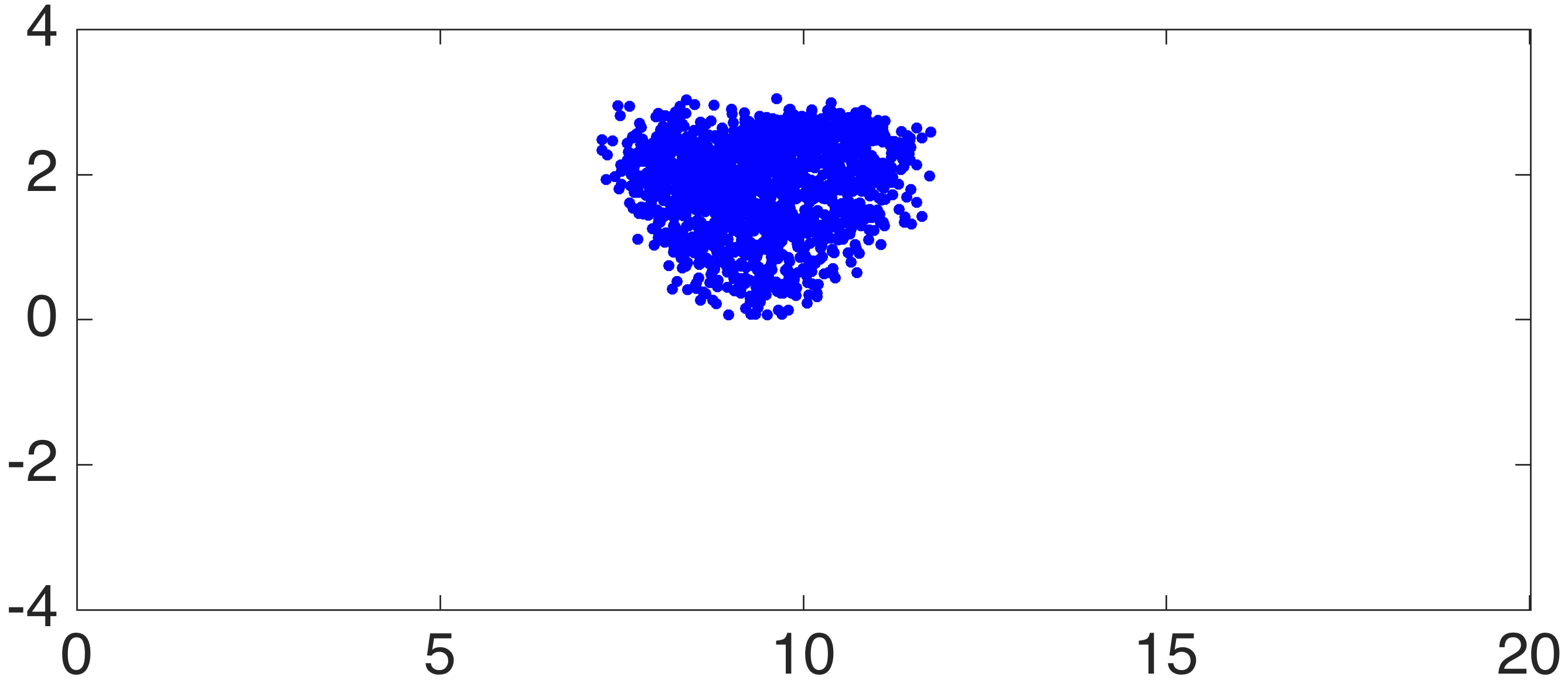} \hfill
\includegraphics[width = 0.45\textwidth]{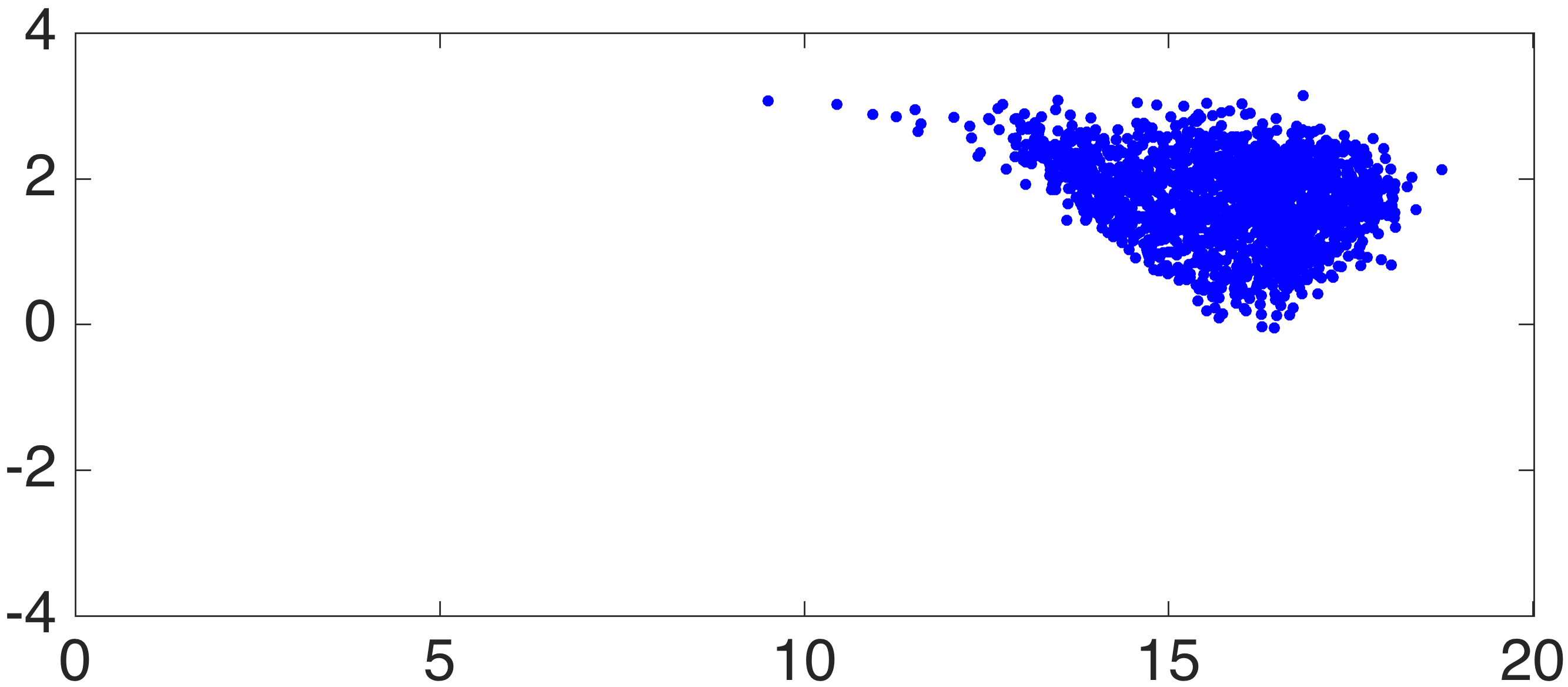} \\
\includegraphics[width = 0.45\textwidth]{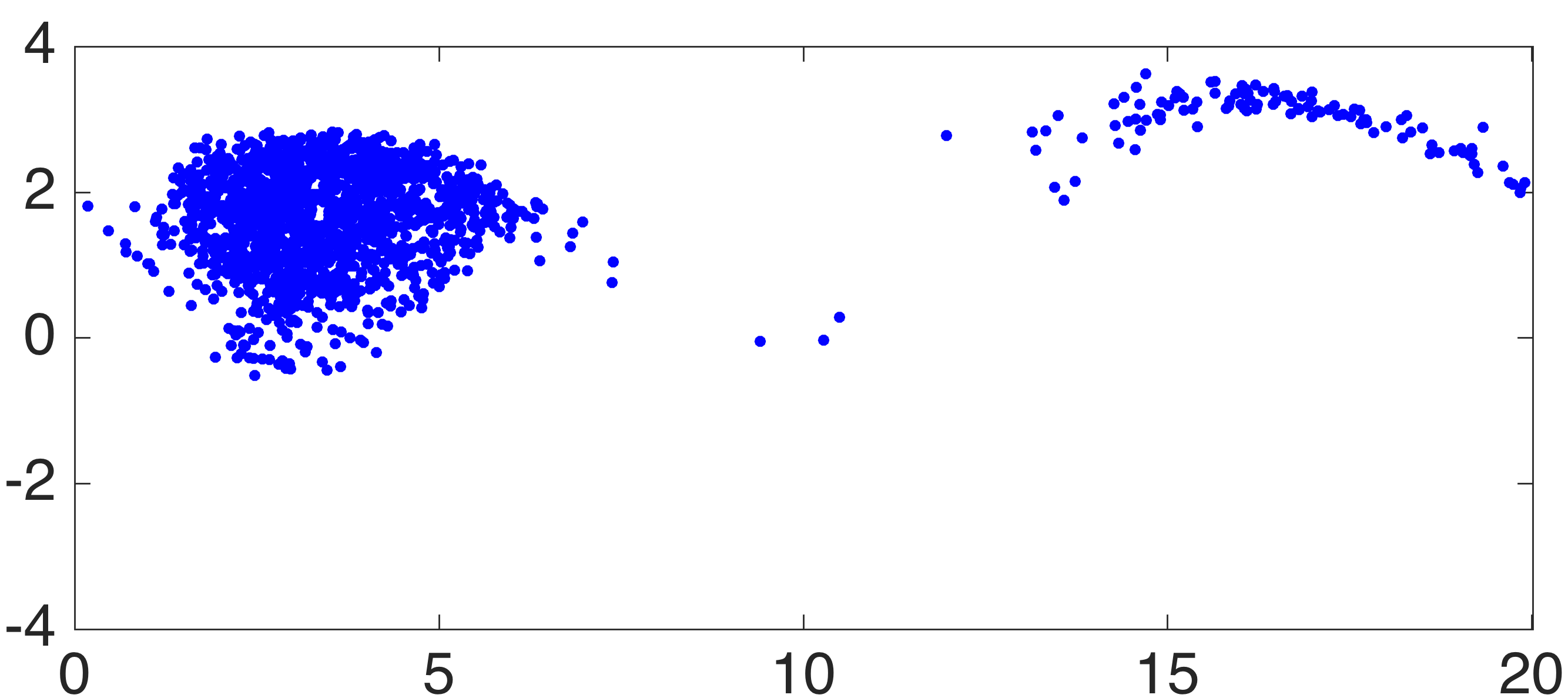} \hfill
\includegraphics[width = 0.45\textwidth]{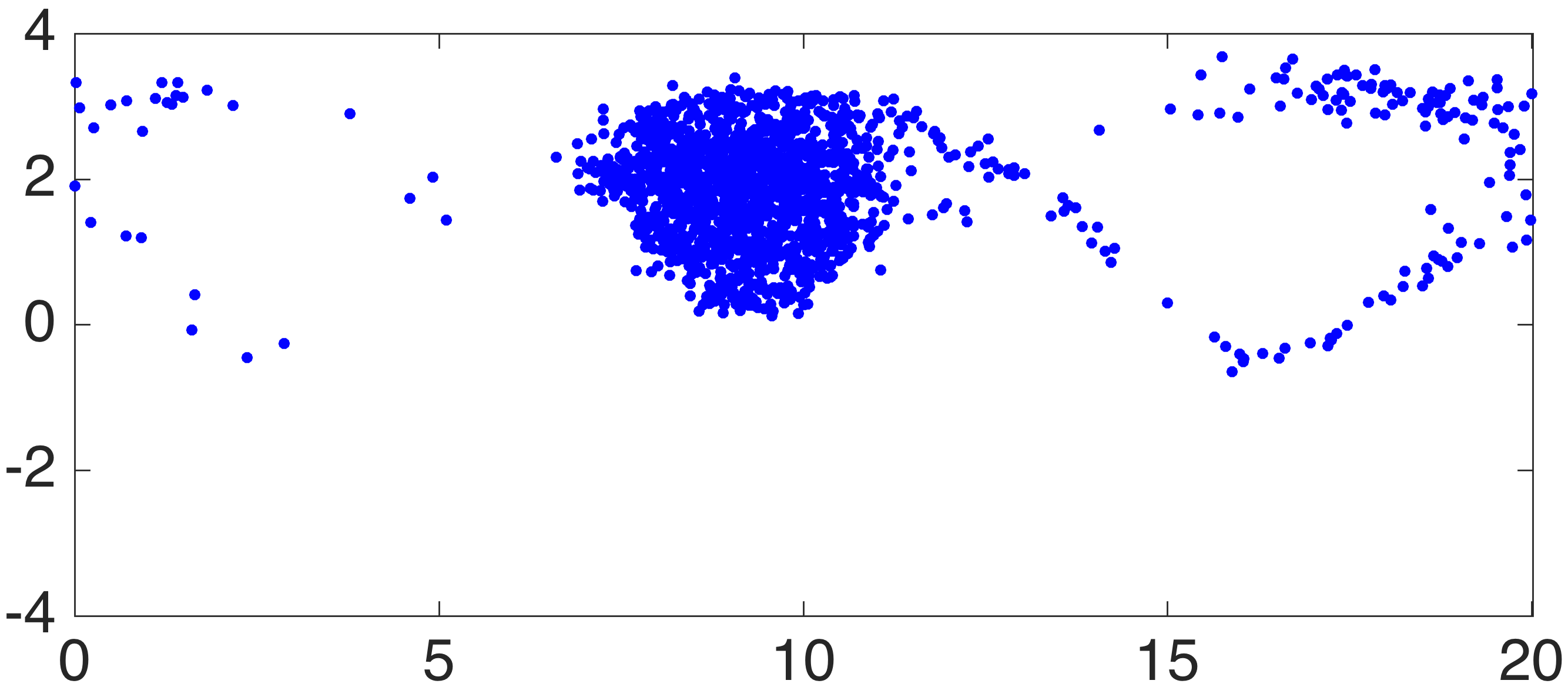}
\caption{Evolution of sample points starting from the top left vortex (cf.~the top right eigenfunction in Figure~\ref{fig:Bickley1}), shown at times~$t=3,6,9,12$ (top left, top right, bottom left, and bottom right, respectively).}
\label{fig:BickleyVortex}
\end{figure}

From the complex eigenvalues, the fourth (and fifth) and sixth (and seventh) are those which promise the most coherent sets. Since these eigenvalues are complex, we treat the corresponding eigenfunctions as described in Remark~\ref{rem:complex}. We choose the fourth eigenfunction with eigenvalue~$\mu = \alpha+\beta i = -0.0332 - 0.1131 i$; the sixth gives a similar but translated structure. We denote the corresponding eigenfunction by~$\aug{f}$, and Remark~\ref{rem:complex} tells us that
\[
A^{\pm}_t = \{\pm \P_{0,t}f_0^{\rm Re}\ge 0\} = \left\{\mathrm{Re}(e^{i\beta t}f_t) \ge 0\right\}
\]
are coherent families. The left column of Figure~\ref{fig:Bickley3} depicts the functions~$f_0^{\rm Re}$, $\P_{0,3}f_0^{\rm Re}$, and~$\P_{0,6}f_0^{\rm Re}$. This shows us that the coherent northern and southern boundaries of the state space (already shown in Figure~\ref{fig:Bickley1}, middle column) decompose into ``coherent tongues'', which move slowly from left to right.
The phase repeats every~$2\pi/0.1131\approx 55.55$ days, much longer than the driving period of~$9$ days.
Because~$55.55$ is not a multiple of~$9$, when the phase restarts, the driving will already be part way through its cycle. The spatial structures approximately return to the same~$x$-position after~$55.55$ days.

The right column of this figure shows a numerical simulation of the escape from the family~$\{A_t^+\}_{t\ge 0}$. In order not to pick points from the central horizontal region where~$\aug{f}$ is orders of magnitude smaller than its mean modulus, we initialize trajectories only where~$f_0^{\rm Re} \ge 3\cdot 10^{-4}$.

\begin{figure}[htb]
\centering
\includegraphics[width = 0.53\textwidth]{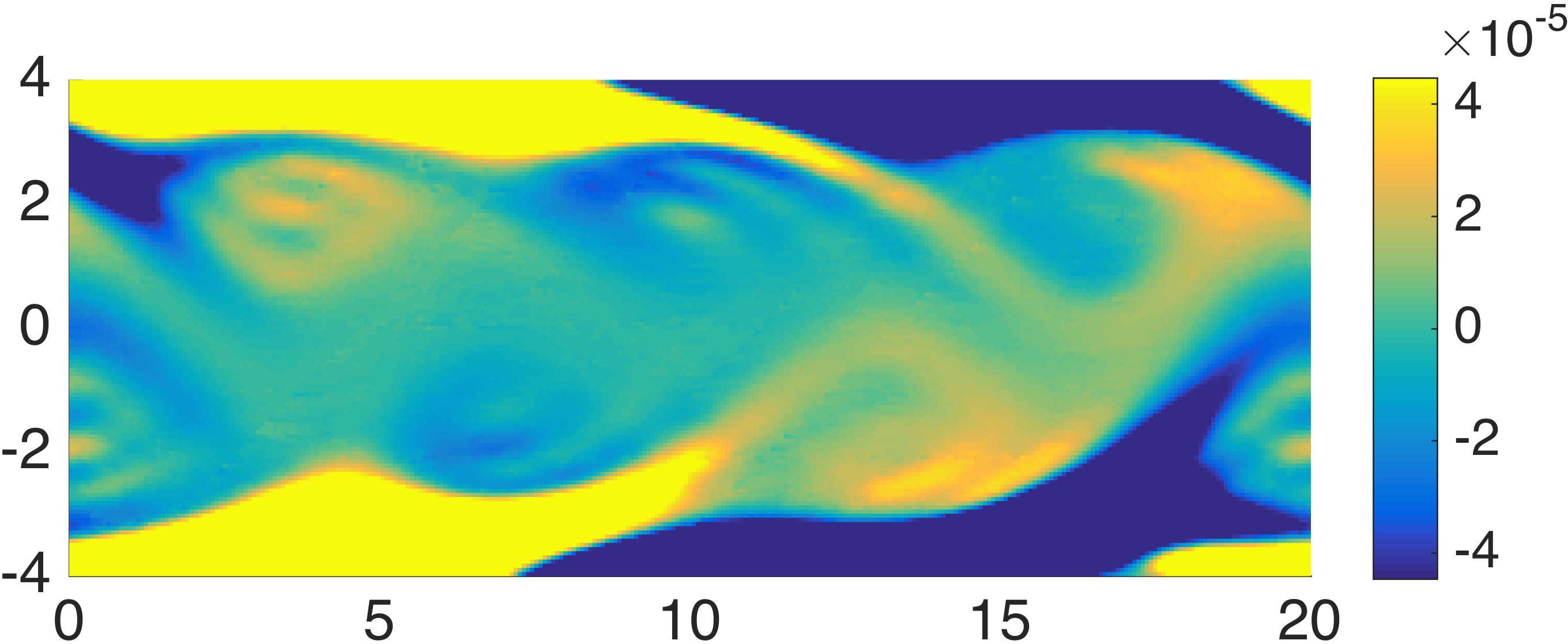} \hfill \includegraphics[width = 0.46\textwidth]{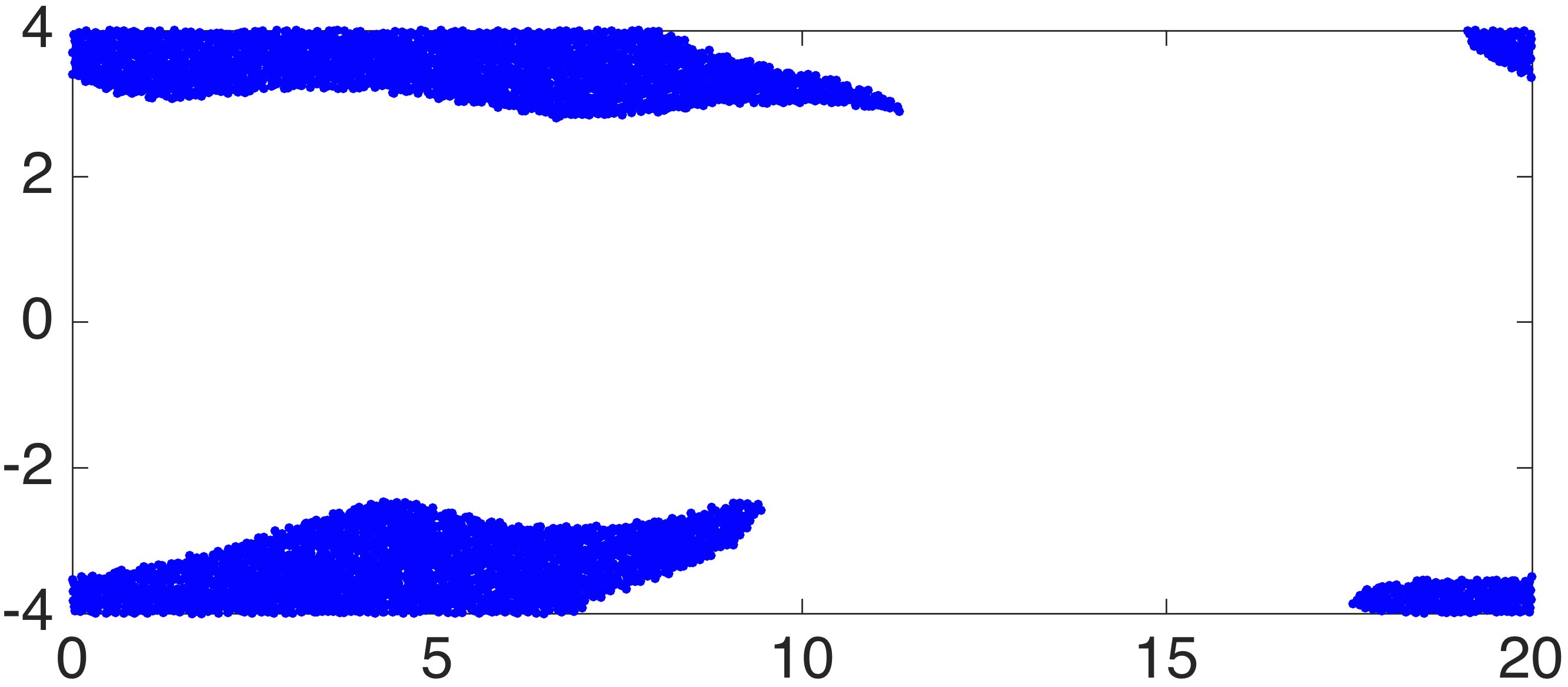} \\
\includegraphics[width = 0.53\textwidth]{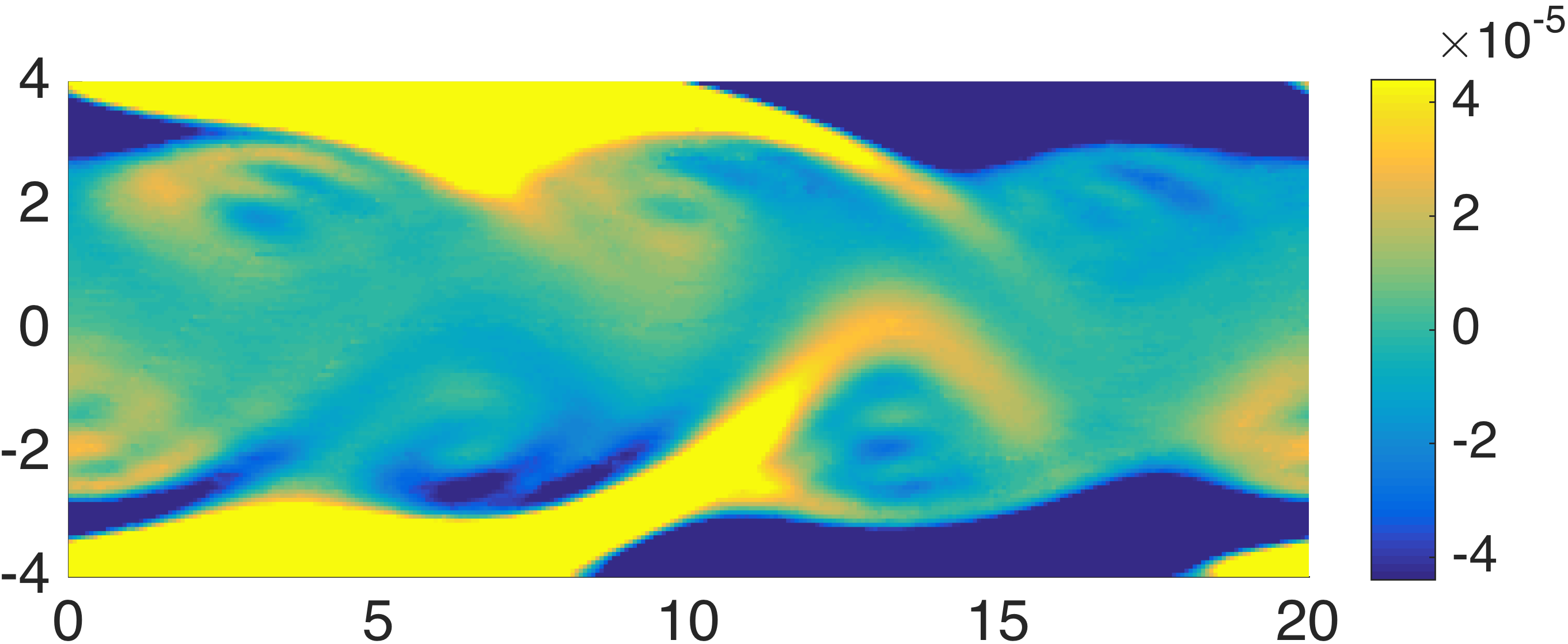} \hfill \includegraphics[width = 0.46\textwidth]{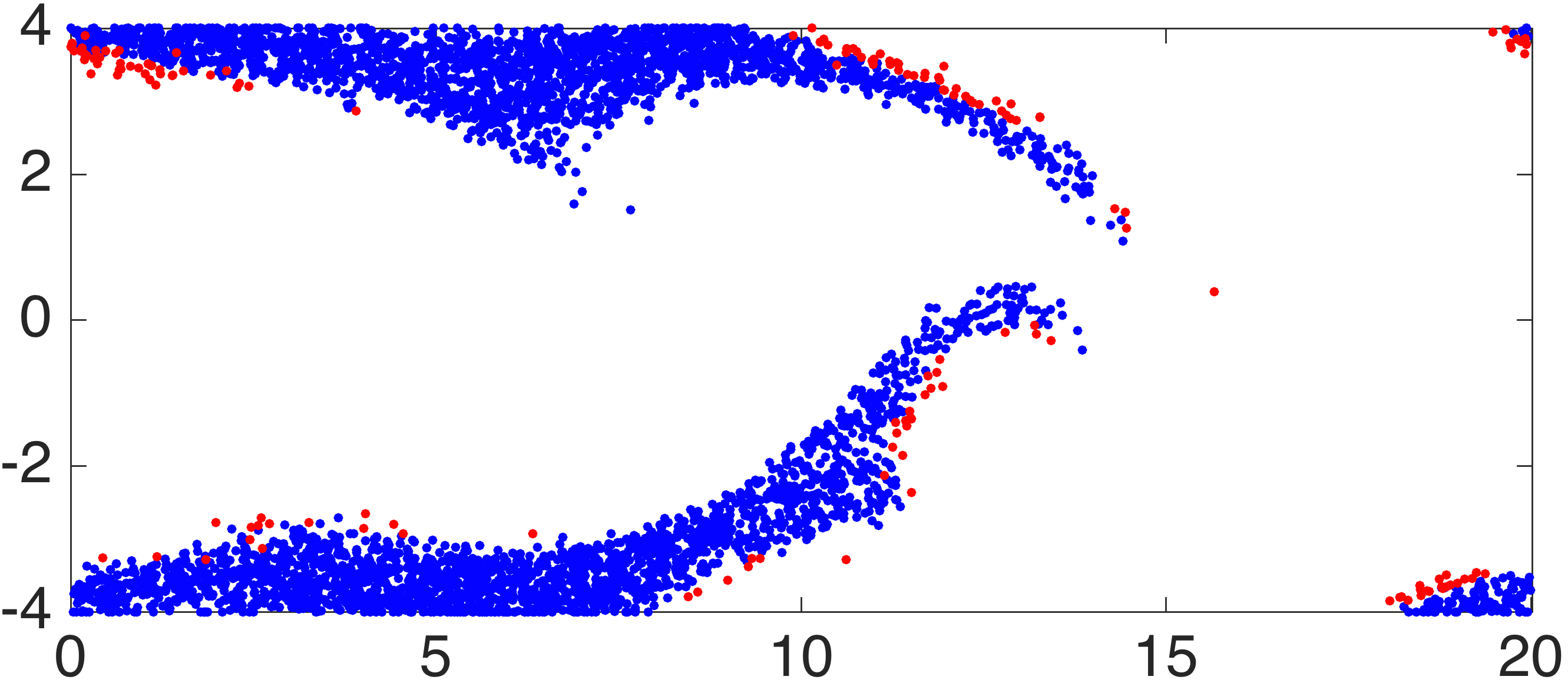} \\
\includegraphics[width = 0.53\textwidth]{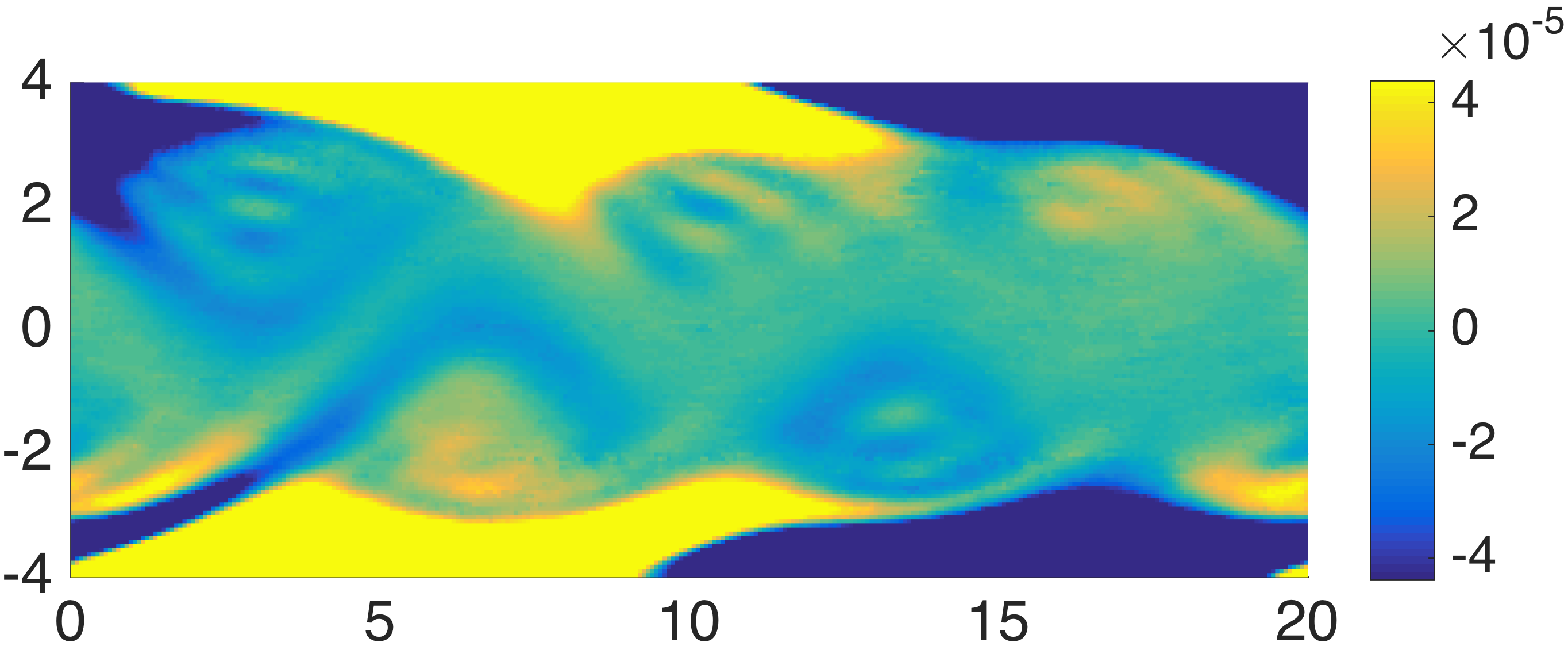} \hfill \includegraphics[width = 0.46\textwidth]{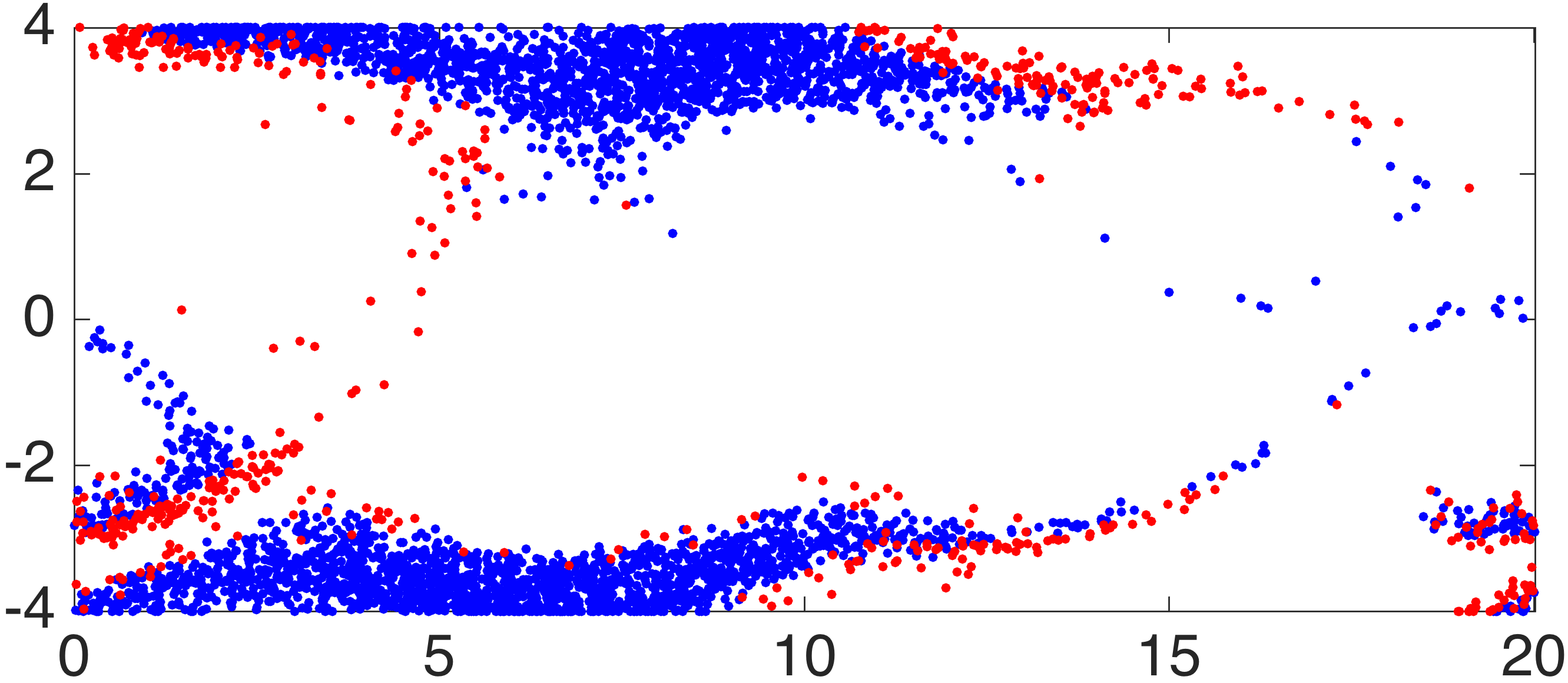}
\caption{Left: real part of the~$s=0$ time slice of the ninth eigenfunction of the hybrid discretized augmented generator evolved to times~$t=0, 3, 6$ (top to bottom). Right: escape rate simulation; trajectories shown at the corresponding times. The blue points stayed in the coherent family until the time when they are shown, the red ones have left the family prior to this time.}
\label{fig:Bickley3}
\end{figure}

\subsection{Numerical comparison with Ulam's method}
\label{ssec:Ulam4DG}

In this section we compare the results of the augmented generator method for the periodically driven double gyre, obtained in section~\ref{ssec:DG}, with the ``standard'' Ulam's method, as described in section~\ref{ssec:Galerkin}.

To this end we assemble the Ulam matrix~$P_n(s,t)$ in the usual way: via sampling-based approximation of its entries. Let us recall that Lemma~\ref{lem:spectral_con} gives us an analytic connection between the augmented generator~$\augIG$ and the one-period transfer operator~$\P_{s,s+\tau}$. Thus, we will consider the Ulam approximation of this transfer operator,~$P = P_n(0,1)$ (we omit the discretization subscript, and the times, since these will be fixed throughout this section).

In order to obtain results comparable with our previous ones, we discretize the state space~$X = [0,2]\times[0,1]$ uniformly into~$100\times 50$ boxes. The \emph{transition rates} between the boxes, i.e.\ the entries of the matrix~$P$,
\[
P_{ij} = \prob_{x_s\sim\mathrm{unif}(B_j)}(x_t\in B_i)\,,
\]
are computed via sampling. For every box~$B_j$ we choose~$N = 2500$ random test points\footnote{For comparison, in~\cite{FrPa14}, the finite-time (as opposed to infinite-time in the present study) coherent set experiments for the double gyre used a grid of $128\times 64$ boxes and~$10000$ points per box. These~$10000$ points comprised~$400$ uniformly distributed points per box, and each of these~$400$ points had an associated~$25$ points describing an~$\epsilon$-ball neighbourhood ($\epsilon=0.02$). Then, these points were integrated with a fourth-order Runge--Kutta scheme with constant stepsize~$0.01$ for a time span of length~$1$, resulting in a total number of vector field evaluation about~$8\cdot 10^9$. This procedure approximates a deterministic flow of 1 period, followed by uniform noise in an $\epsilon$-ball. In the present paper we instead simulate an SDE for 1 period.} $x_s^{(1)},x_s^{(2)}, \ldots, x_s^{(N)}$ at the initial time, and we set
\[
P_{ij} = \frac{1}{N}\sum_{k=1}^N\mathbbm{1}_{B_i}(x_t^{(k)})\,,
\]
where~$x_t^{(k)}$, $k=1,\ldots,N$, are independent realizations of the underlying stochastic process computed by the Euler--Maruyama method~\cite{KlPl92} (with reflecting boundary) for step-size~$1/30$ (the temporal direction was also discretized in~$30$ boxes in section~\ref{ssec:DG}). Note that this leads to~$375\cdot 10^6$ evaluations of the vector field, a factor~$25$ times more than in the case of the augmented generator, where we used a two-dimensional Gau{\ss} quadrature with~$4\times 4$ nodes on each of the~$6$ faces of the three-dimensional boxes. Note that since the augmented vector field is constant in the temporal direction, we could calculate the transition rates in this direction analytically, thus effectively reducing the need for quadrature to~$4$ from the~$6$ faces. Even larger savings can be obtained by applying the hybrid discretization from section~\ref{ssec:hybrid}: not only can we work with lower temporal resolutions, but also the generators~$G(t)$ in~\eqref{eq:hybrid generator} are computed on the lower-dimensional space~$X$ instead of~$\aug{X}$.

We expect the dominant eigenvalues and eigenfunctions of~$P$ to be associated with the eigenvalues and eigenfunctions of the augmented generator~$\aug{G}$, predicted by Lemma~\ref{lem:spectral_con} in the analytical case. The eigenvalues with the largest modulus are shown in Table~\ref{tab:UlamEV}, together with their logarithmic transforms, and the corresponding eigenvalues of the augmented generator~$\aug{G}$.
\begin{table}[h]
\begin{tabular}{|c|c|c|}
\hline
$\lambda_i$ & $\tau^{-1}\log(\lambda_i)$ & $\mu_i$\\ \hline
\texttt{1.0000 + 0.0000i} &  \texttt{ 0.0000 + 0.0000i} & \texttt{-0.0000 + 0.0000i} \\
\texttt{0.9150 + 0.0000i} & \texttt{-0.0888 + 0.0000i} & \texttt{-0.0832 + 0.0000i} \\
\texttt{0.3993 + 0.6678i} & \texttt{-0.2509 + 1.0319i} & \texttt{-0.3160 + 1.1437i} \\
\texttt{0.3993 - 0.6678i} & \texttt{-0.2509 - 1.0319i} & \texttt{-0.3160 - 1.1437i} \\
\texttt{0.7616 + 0.0000i} & \texttt{-0.2723 + 0.0000i} & \texttt{-0.3663 + 0.0000i} \\
\texttt{0.6929 + 0.0000i} & \texttt{-0.3669 + 0.0000i} & \texttt{-0.5169 + 0.0000i}\\ \hline
\end{tabular}
\caption{Left column: dominant eigenvalues~$\lambda_i$ of the Ulam matrix~$P$. Middle column: logarithms of the~$\lambda_i$. Right column: corresponding eigenvalues~$\mu_i$ of the Ulam discretization~$\aug{G}$ of the augmented generator.}
\label{tab:UlamEV}
\end{table}
The upwind type spatial discretization introduces so-called numerical diffusion in~$\aug{G}^{\rm drift}$ (as discussed in~\cite{FrJuKo13}), adding to the mixing in the system.\footnote{We conjecture that the second eigenvalue of~$\aug{G}$, $-0.0832$, is not smaller than the log-transform of the corresponding eigenvalue of~$P$, because (i) the zero-level curve of the corresponding eigenfunction has a significant vertical component, and (ii) the vector field is approximately vertical near the zero-level curve. This likely leads to a reduction in numerical diffusion in~$\aug{G}^{\rm drift}$ across the zero-level curve, when compared to lower eigenvalues, since numerical diffusion occurs locally, where the vector field has a significant component transversal to a box face.
}
This is the reason for the augmented generator having eigenvalues (from the third on) with smaller real parts than the corresponding log-transformed eigenvalues of the Ulam matrix.
Figure~\ref{fig:DG_Ulam1} shows the second, third, and fifth eigenfunctions of~$P$.
\begin{figure}[htb]
\centering
\includegraphics[width = 0.32\textwidth]{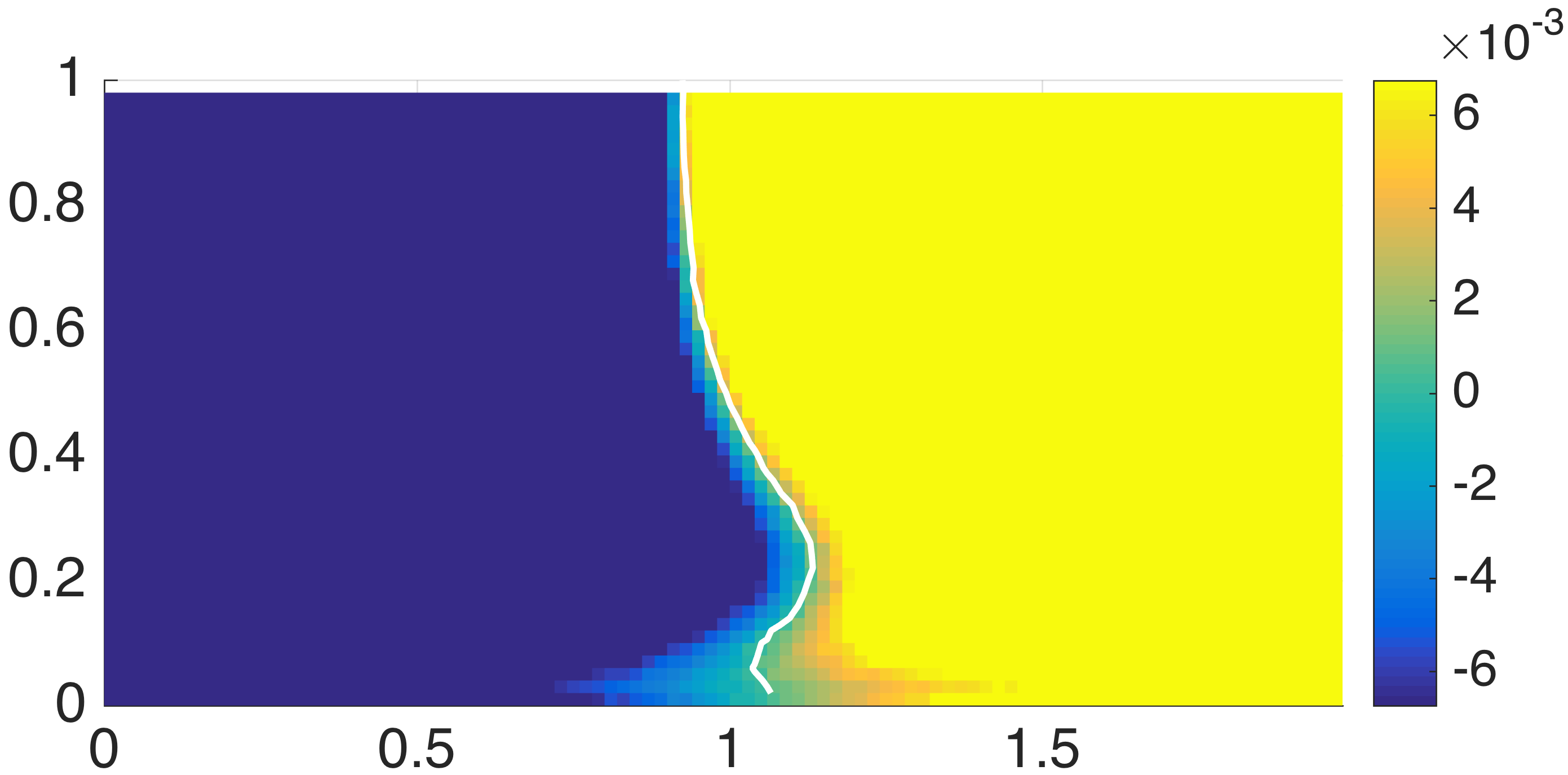}
\hfill
\includegraphics[width = 0.32\textwidth]{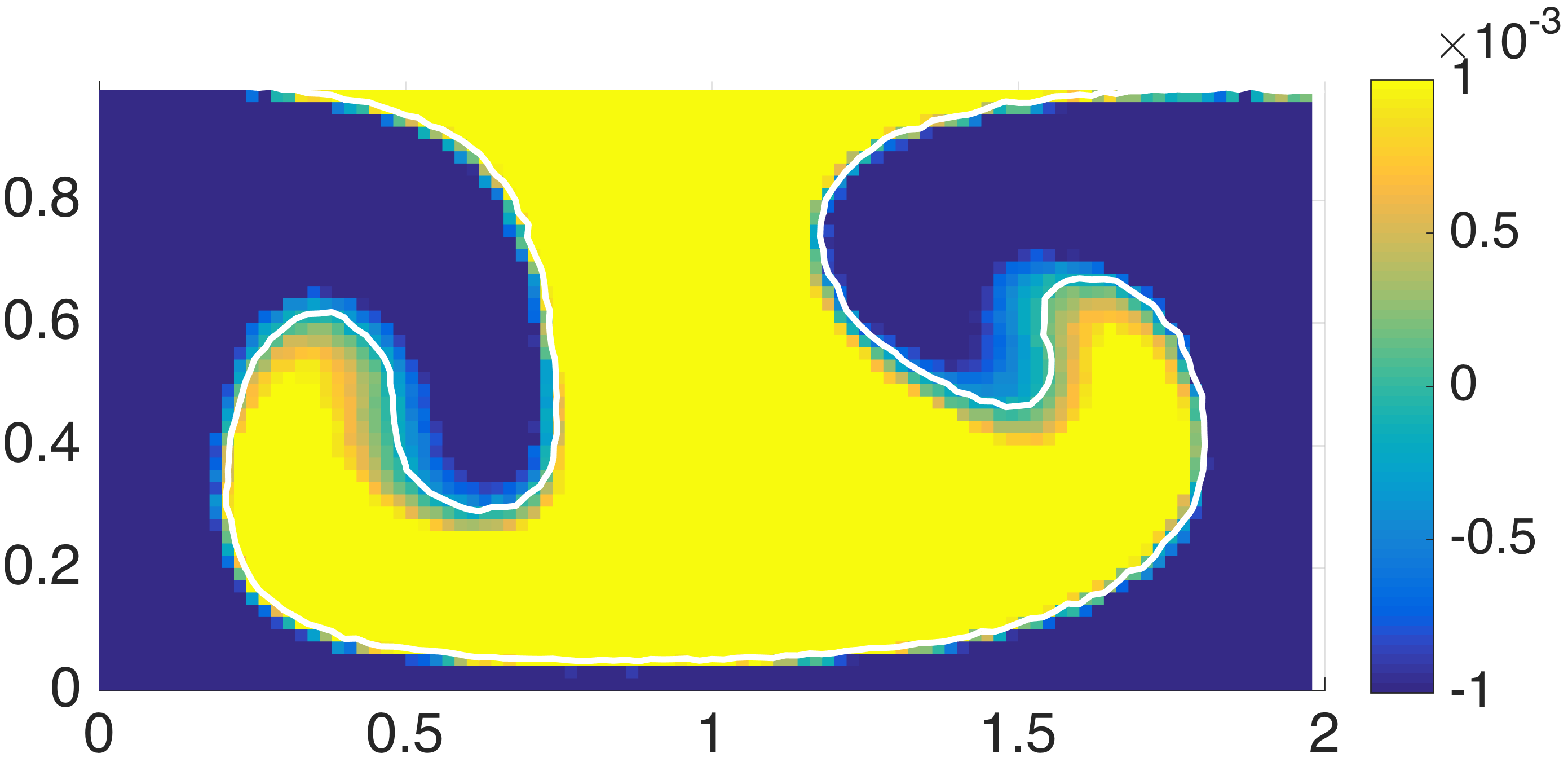}
\hfill
\includegraphics[width = 0.32\textwidth]{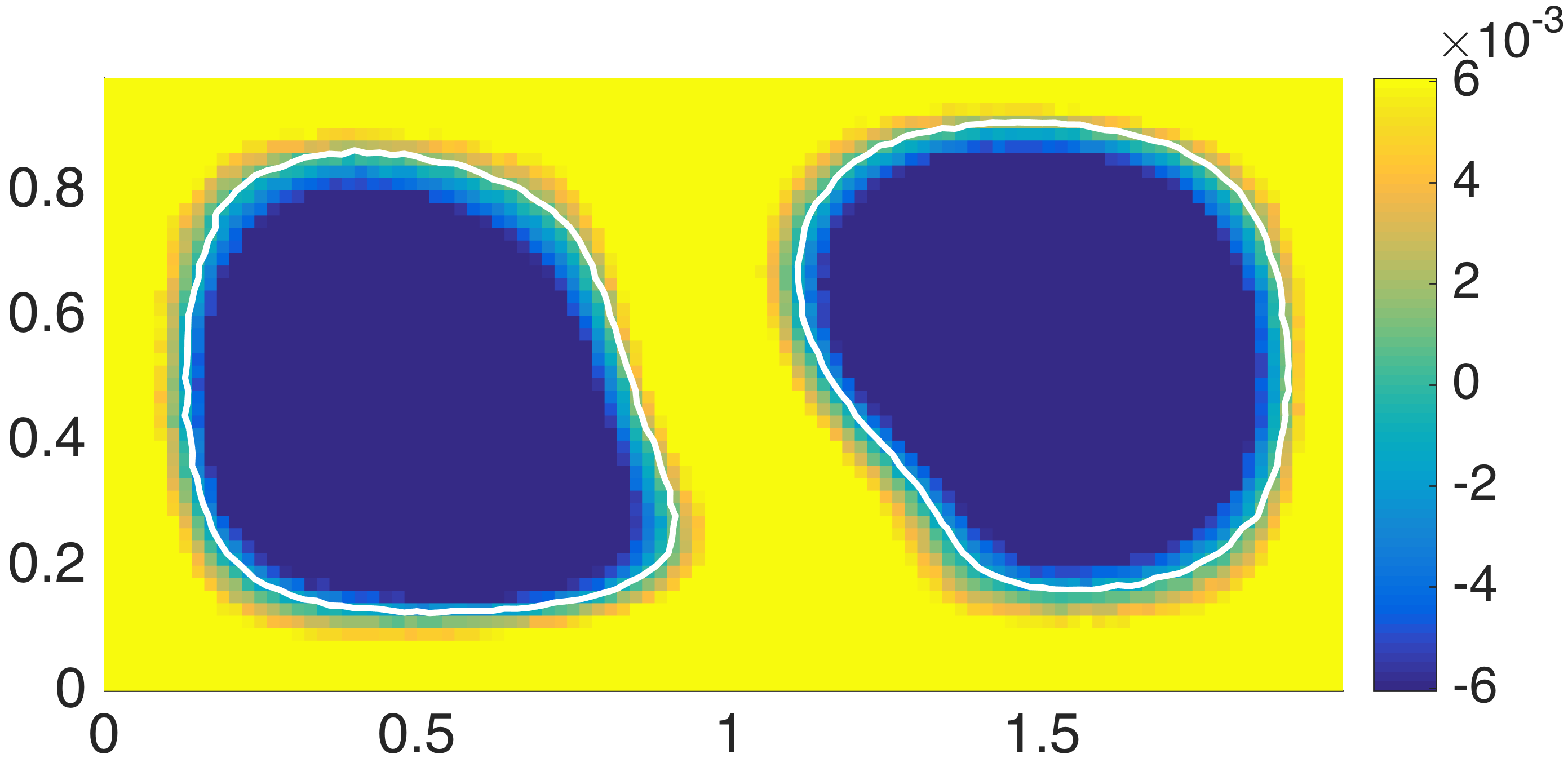}
\caption{From left to right: second, third (shown is the real part), and fifth eigenfunctions of the Ulam discretization~$P$ of the transfer operator. The white contours indicate the zero-level curves.}
\label{fig:DG_Ulam1}
\end{figure}
All the eigenfunctions are, as expected, very close to those of the augmented generator~$\aug{G}$; cf.~Figures~\ref{fig:DoubleGyre} and~\ref{fig:DoubleGyreComplex}. The third, complex eigenfunction differs only in the phase, which is arbitrary. If we repeat the computation now only with~$100$ sample points per box, such that the overall number of vector field evaluations matches that for the augmented generator, we obtain the eigenfunctions given in Figure~\ref{fig:DG_Ulam2}.
\begin{figure}[htb]
\centering
\includegraphics[width = 0.32\textwidth]{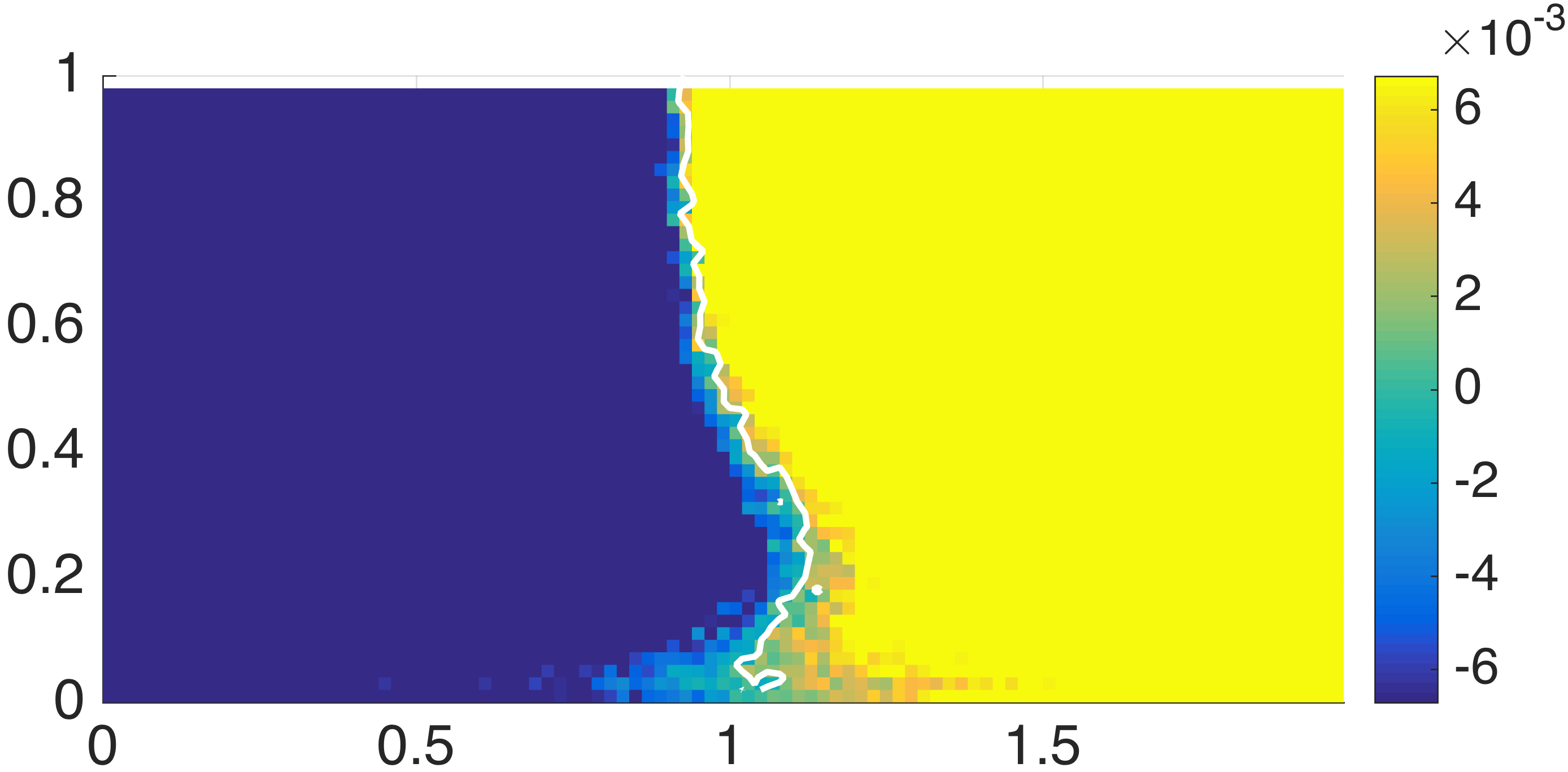}
\hfill
\includegraphics[width = 0.32\textwidth]{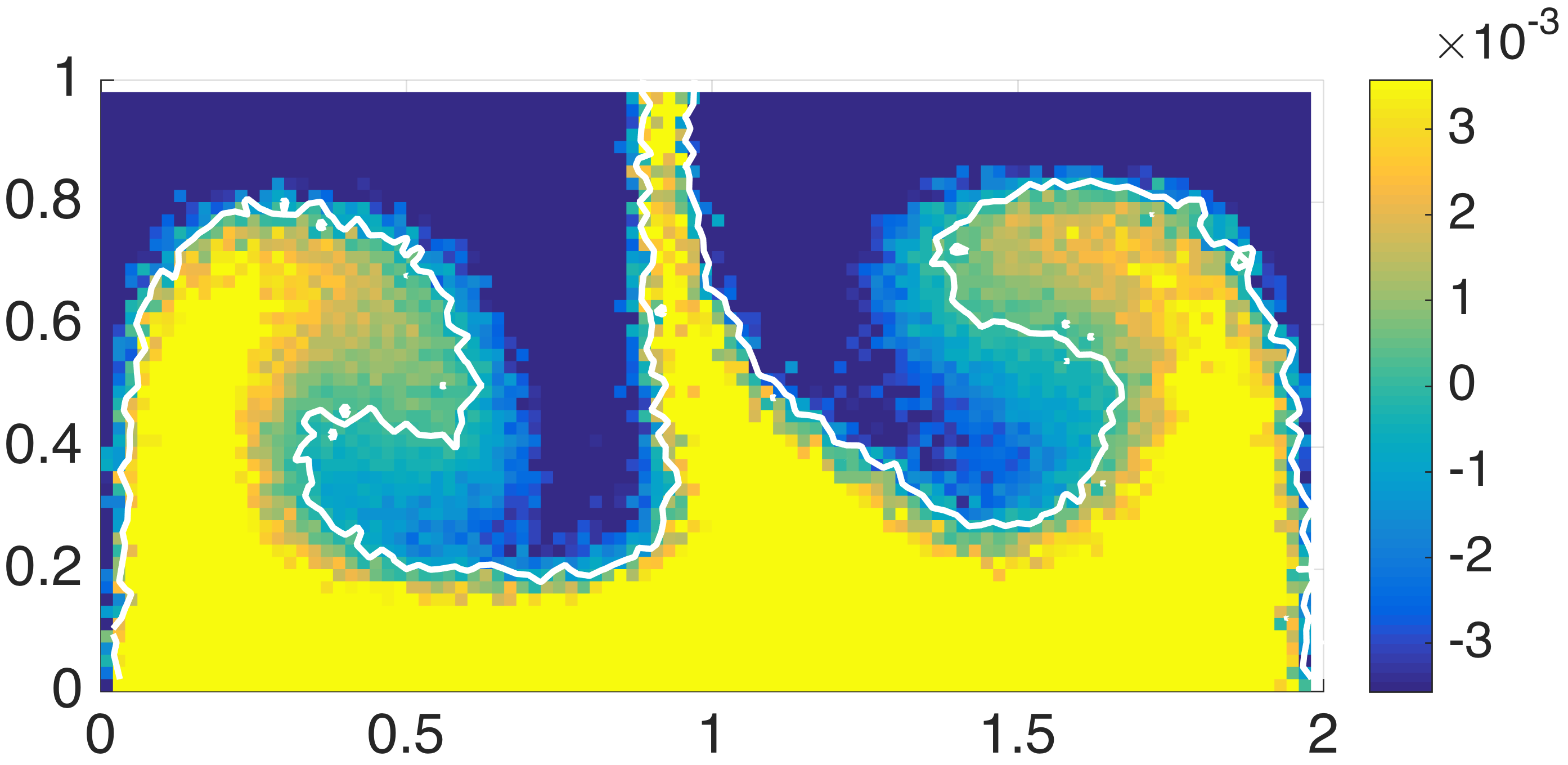}
\hfill
\includegraphics[width = 0.32\textwidth]{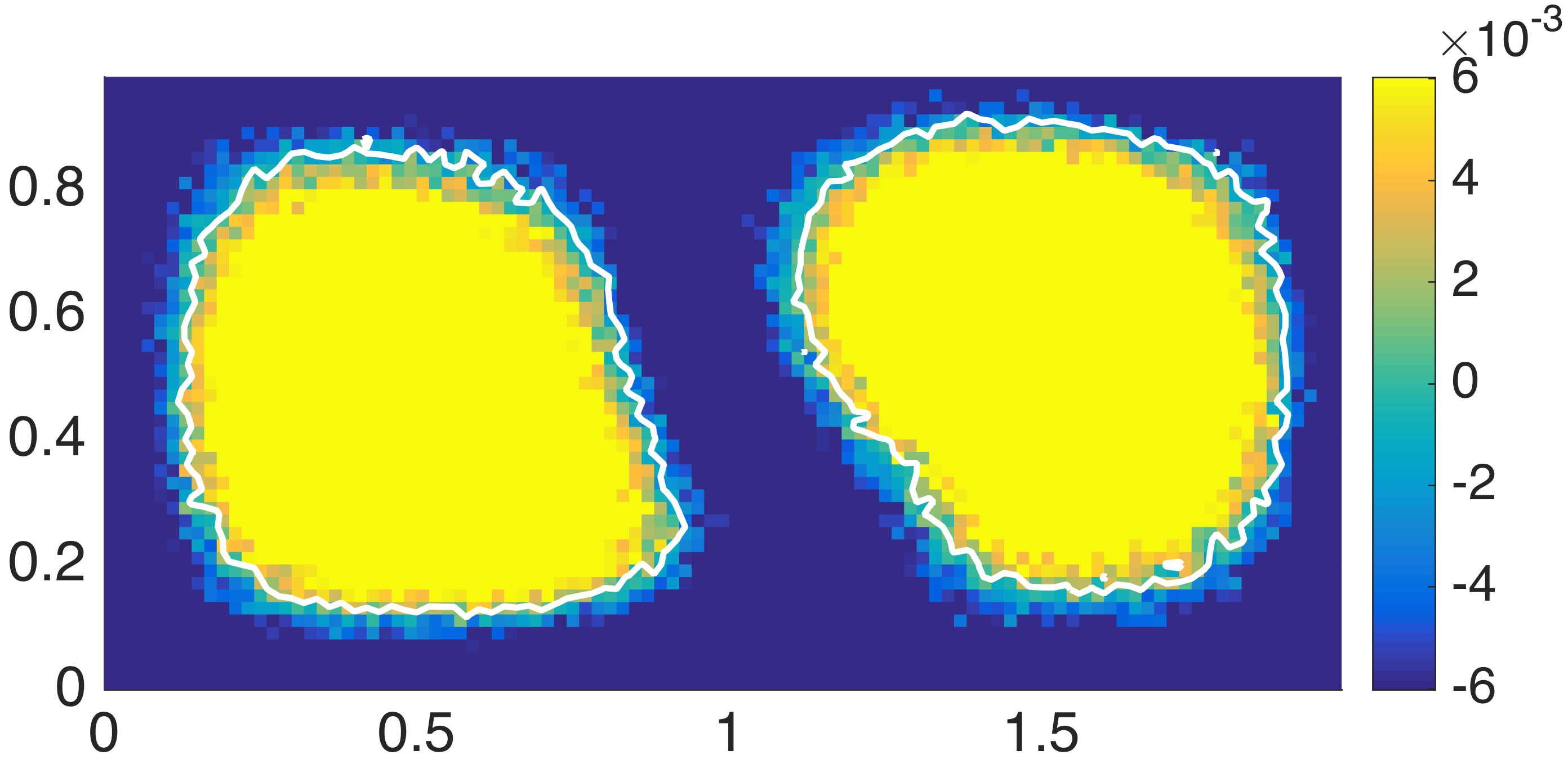}
\caption{From left to right: second, third (shown is the real part), and fifth eigenfunctions of the Ulam discretization~$P$ of the transfer operator, now computed with a 25 times smaller number of sample points as those in Figure~\ref{fig:DG_Ulam1}. The white contours indicate the zero-level curves.}
\label{fig:DG_Ulam2}
\end{figure}
Note that the smoothness of the numerical eigenfunctions has deteriorated due to an insufficient sampling; the zero-level curves have also lost smoothness.
In general, the longer the time span~$(s,t)$, the more sampling points have to be integrated in order to resolve the noise sufficiently, due to exponential expansion in the flow. Certainly, a larger time span necessitates a larger number of basis functions in the temporal direction for the augmented generator too. It is not yet clear whether this number also scales exponentially in time, or a linear growth is sufficient to maintain the same level of accuracy. The augmented generator incorporates diffusion as an additive differential operator, and guarantees for smoothness of the results.

\section*{Acknowledgments}

GF is supported by an ARC Future Fellowship. PK thanks the UNSW School of Mathematics and Statistics for hospitality at an early stage of this work, and an ARC Discovery Project for supporting this visit. 
The authors thank Michael~Kratzer for helpful comments regarding the proof of Lemma~\ref{lem:auxlemma1}, Roland Schnaubelt for pointing out useful references, and Martin Plonka for his careful proofreading.

\begin{appendix}

\section{Proofs}
\subsection{Proof of Theorem \ref{thm:1=2}} \label{app:proof:thm:1=2}
\begin{proof}[Proof of Theorem~\ref{thm:1=2}]
Let~$D_t a$ and~$D_r a$ denote the derivatives of the parametrizing function~$a$ with respect to~$t\in \tau S^1$ and~$r\in R$, respectively. Then, we have for the accumulated outflow flux:
\begin{multline} \label{eq:cumoutflow2}
\int_{\tau S^1}\int_{\partial A_t}\left\langle v(t,x)-w(t,x),n_t(x)\right\rangle^+ d\sigma(x)dt = \hspace*{15em} \\
 = \int_{\tau S^1}\int_R\langle v\left(t,a\left(t,r\right)\right)-D_t a\left(t,r\right),n_t\left(a\left(t,r\right)\right)\rangle^+ g_t(r)\, dr dt,
\end{multline}
where $g_t(r) = \det\left(D_ra(t,r)^{\rm T}D_ra(t,r)\right)^{1/2}$ is the Gram determinant.
The outer normal in the augmented space has the form
\[
\aug{n}(t,r):=\aug{n}(\aug{x}) = \begin{pmatrix}\alpha \\ \beta n_t(a(t,r))\end{pmatrix},
\]
with $\alpha,\beta\in\R$ satisfying $\alpha^2+\beta^2=1$ and
\[
\aug{n}(t,r)\perp \begin{pmatrix}1 \\ D_t a(t,r)\end{pmatrix}.
\]
Solving for $\aug{n}$ yields
\[
\aug{n}(t,r) = \left(1+\langle n_t\left(a\left( t,r\right)\right),D_t a\left(t,r\right)\rangle^2\right)^{-1/2}\begin{pmatrix}-\langle n_t(a(t,r)),D_t a(t,r)\rangle \\ n_t(a(t,r))\end{pmatrix}.
\]
The instantaneous outflow flux in augmented space, $\int_{\partial \aug{A}}\langle \aug{v}(\aug{x}),\aug{n}(\aug{x})\rangle^+ d\aug{\sigma}(\aug{x})$, now reads as
\begin{multline}	\label{eq:instoutflow2}
\int_{\tau S^1\times R} \left < \begin{pmatrix} 1\\ v(t,a(t,r)) \end{pmatrix},\begin{pmatrix} -\langle n_t(a(t,r)),D_t a(t,r)\rangle \\ n_t(a(t,r)) \end{pmatrix} \right>^+ \times\\
\times \frac{ \aug{g}(t,r)}{(1+\langle n_t(a(t,r)),D_t a(t,r)\rangle^2)^{1/2}}\, d(t,r),
\end{multline}
with~$\aug{g}(t,r) = \det\left(D_{(t,r)}a(t,r)^{\rm T}D_{(t,r)}a(t,r)\right)^{1/2}$, where~$D_{(t,r)}a$ denotes the total derivative of~$a$ with respect to the joint variable~$(t,r)$; i.e.~$D_{(t,r)}a = (D_ta,D_ra)$.
Comparing~\eqref{eq:cumoutflow2} with \eqref{eq:instoutflow2} shows that they are equal if and only if
\begin{equation}	\label{eq:outfloweq}
(1+\langle n_t(a(t,r)),D_t a(t,r)\rangle^2)^{1/2} g_t(r) = \aug{g}(t,r)\quad\text{for every }t\in S^1,\, r\in R.
\end{equation}
Using Lemma~\ref{lem:auxlemma1} below with~$M=D_ra(t,r)$, $m=D_ta(t,r)$ and~$n=n_t(a(t,r))$ proves~\eqref{eq:outfloweq} and completes the proof of~\eqref{eq:1=2}.

Now that we have shown equality in~\eqref{eq:1=2}, the independence of its left- and right-hand sides on the parametrization~$a$ is a simple corollary of the independence of surface integrals on parametrization. In particular, the right-hand side of~\eqref{eq:1=2},~$\int_{\partial \aug{A}}\langle \aug{v}(\aug{x}),\aug{n}(\aug{x})\rangle^+ d\aug{\sigma}(\aug{x})$, is a surface integral, and hence depends only on~$\aug{v}$ and~$\aug{A}$, and both these objects are parametrization-independent.
\end{proof}

\begin{lemma} \label{lem:auxlemma1}
Let~$M\in\R^{k\times k-1}$ have full rank, let~$n\in\R^k$ with~$n^{\rm T}M = 0$ and~$n^{\rm T}n=1$, and let~$m\in\R^k$. Then we have
\[
(1+(n^{\rm T}m)^2)\det(M^{\rm T}M) = \det\begin{pmatrix} 1+m^{\rm T}m & m^{\rm T}M \\ M^{\rm T}m & M^{\rm T}M \end{pmatrix}.
\]
\end{lemma}
\begin{proof}
Note that
\[
\begin{pmatrix} 1+m^{\rm T}m & m^{\rm T}M \\ M^{\rm T}m & M^{\rm T}M \end{pmatrix} = \begin{pmatrix} 1 & m^{\rm T}\\ 0 & M^{\rm T}\end{pmatrix} \begin{pmatrix} 1 & 0\\ m & M \end{pmatrix}.
\]
Our strategy will be to use a transformation matrix which has determinant one and changes the $m$ on the off-diagonal blocks in the factors to a multiple of $n$. This will impose in the product a zero off-diagonal block, since $n^{\rm T}M=0$. Hence, let $u\in\R^{k-1}$ be such that $Mu+m = (n^{\rm T}m) n$. This is possible since $M$ has full rank and by $n^{\rm T} M=0$ it holds $\text{Range}(M)^{\perp} = \R n$. Define $A:=\begin{pmatrix} 1 & 0\\ u & I_{k-1\times k-1} \end{pmatrix}$, and note $\det(A)=1$. By this,
\begin{eqnarray*}
\det\left(\begin{pmatrix} 1 & m^{\rm T} \\ 0 & M^{\rm T}&\end{pmatrix} \begin{pmatrix} 1 & 0\\ m & M \end{pmatrix}\right) & = & \det\left(A^{\rm T}\begin{pmatrix} 1 & m^{\rm T} \\ 0 & M^{\rm T}&\end{pmatrix} \begin{pmatrix} 1 & 0\\ m & M \end{pmatrix}A\right)\\
& = & \det\left(\begin{pmatrix} 1 & (n^{\rm T}m) n^{\rm T} \\ 0 & M^{\rm T}&\end{pmatrix} \begin{pmatrix} 1 & 0\\ (n^{\rm T}m) n & M \end{pmatrix}\right) \\
& = & \det \begin{pmatrix} 1 + (n^{\rm T}m)^2 & 0\\ 0 & M^{\rm T}M&\end{pmatrix}\\
& = & \left(1 + (n^{\rm T}m)^2\right)\det(M^{\rm T}M).
\end{eqnarray*}
\end{proof}

\subsection{Proof of Proposition \ref{prop:nice=measurable}} \label{app:nice=measurable}

Clearly, $A_{s,t} \subseteq A^n:=\bigcap_{i=1}^n \phi_{r_i,s}A_{r_i}$, and $A^{n+1}\subseteq A^n$. We show that\footnote{For a sequence of sets $(A^n)_{n\in\N}$ and a set $A$ we write $A^n\downarrow A$, if $A^{n+1}\subseteq A^n$ for every~$n\in\N$, and $\bigcap_{n\in\N}A^n =  A$.} $A^n\downarrow A_{s,t}$. Let us fix some~$x\in X$. Note that since the sets $A_r$ are closed, we have
\[
x\notin A_{s,t} \quad\Leftrightarrow\quad \exists\,r_*:\, x\notin \phi_{r_*,s}A_{r_*} \quad\Leftrightarrow\quad d(x,\phi_{r_*,s}A_{r_*})\ge \epsilon \text{ for some }\epsilon > 0\,.
\]
Due to property (b) of the niceness, the function $r\mapsto d(x,\phi_{r,s}A_r)$ is continuous, hence there is a $\delta>0$ such that $d(x,\phi_{r,s}A_r)>\epsilon/2$ for all $r\in(r_*-\delta,r_*+\delta)$. Then, there is an~$N\ge 1$ such that $\{r_1,\ldots,r_n\} \cap (r_*-\delta,r_*+\delta) \neq \emptyset$ for every~$n\ge N$, hence
\[
x\notin  \bigcap_{i=1}^n \phi_{r_i,s}A_{r_i} = A^n.
\]
Thus,~$A^n\downarrow A_{s,t}$ as~$n\to\infty$. As~$A_{s,t}$ is a countable intersection of measurable sets, is is measurable itself.

By the nestedness of the $A^n$, i.e.\ $A^{n+1}\subseteq A^n$ for $n\ge 1$, and the $\sigma$-additivity of~$m$ we have
\begin{eqnarray*}
m(A_{s,t}) & = & m(A^1) - \sum_{n=1}^{\infty} m(A^n\setminus A^{n+1}) \\
	 & = & \lim_{k\to\infty}
	 \left( m(A^1) - \sum_{n=1}^{k-1} m(A^n\setminus A^{n+1})\right)\\
	 & = & \lim_{k\to\infty} m(A^k)\,.
\end{eqnarray*}

\subsection{Proof of Theorem \ref{thm:escrate_indep}} \label{app:indeptime}

We begin with a lemma.
\begin{lemma}\label{lem:trafo}
Let $\varphi:X\to X$ be a diffeomorphism on a compact set~$X$. For every measurable $A\subset X$ one has
\[
|D\varphi|_{\min}\, m(A) \le m(\varphi(A)) \le |D\varphi|_{\max}\, m(A)\,,
\]
where $0<|D\varphi|_{\min} = \min_{x\in X}|\det D\varphi(x)|$, and $|D\varphi|_{\max} = {\max_{x\in X}|\det D\varphi(x)|<\infty}$.
\end{lemma}
\begin{proof}
Note that since~$\varphi$ is a diffeomorphism, $\varphi^{-1}$ is a continuous map, hence $\varphi(A)$ is a measurable set. By the integral transformation theorem
\[
m(\varphi(A)) = \int_{\varphi(A)}1\,dx = \int_A 1\,|\det D\varphi(x)|\,dx,
\]
from which the inequalities follow immediately. Since~$\varphi$ is a diffeomorphism, $D\varphi$ is continuous and everywhere nonsingular, hence by the compactness of~$X$ the continuous function $|\det(D\varphi)|$, which is bounded away from zero, attains its min, which is positive, and its max, which is finite.
\end{proof}
Next, by splitting the intersection, we obtain for $s_1\le s_2\le t$
\[
\bigcap_{r=s_1}^t\phi_{r,s_1}A_r = \left(\bigcap_{r=s_1}^{s_2}\phi_{r,s_1}A_r\right)\,\cap\,\left(\phi_{s_2,s_1}\bigcap_{r=s_2}^t\phi_{r,s_2}A_r\right)\,,
\]
which reads in our shorthand notation as
\begin{equation}
A_{s_1,t} = A_{s_1,s_2} \cap \phi_{s_2,s_1}A_{s_2,t}\,.
\label{eq:expand_remset}
\end{equation}
This implies
\begin{equation}
m(A_{s_1,t}) \le \min\left\{m(A_{s_1,s_2}), m(\phi_{s_2,s_1}A_{s_2,t})\right\} \le m(\phi_{s_2,s_1}A_{s_2,t})\,.
\label{eq:measbound}
\end{equation}
Using continuity of the vector field~$v$ over~$\tau S^1 \times X$ we have that~$\phi_{s_2,s_1}$ is a diffeomorphism for every~$s_2\ge s_1$, and $|\det D\phi_{s_2,s_1}|$ is continuous in~$s_1, s_2$. Hence there are~$J_{\min}>0$ and~$J_{\max}<\infty$ such that~$J_{\min} \le |\det D\phi_{s_2,s_1}(x)| \le J_{\max}<\infty$ for every~$s_1, s_2\in\R\slash \tau\Z$ and~$x\in X$. Lemma~\ref{lem:trafo} with~\eqref{eq:measbound} yields
\begin{corollary} \label{cor:measbound}
For $s_1\le s_2\le t$ we have
\[
m(A_{s_1,t}) \le J_{\max}\, m(A_{s_2,t})\,.
\]
\end{corollary}
Due to the monotonicity of the logarithm this corollary yields
\[
\log m(A_{s_1,t}) \le \log J_{\max} + \log m(A_{s_2,t})\,.
\]
Taking the $-\limsup_{t\to\infty}\tfrac1t$ of both sides, we obtain
\begin{equation}
E(\{A_r\}_{r\ge s_1}) \ge E(\{A_r\}_{r\ge s_2}),\qquad \text{for all }s_1\le s_2\,.
\label{eq:escrateineq1}
\end{equation}
Observe that due to periodicity~$E(\{A_r\}_{r\ge s}) = E(\{A_r\}_{r\ge s+\tau})$, since for~$t\ge s+\tau$ one has
\[
\bigcap_{r=s+\tau}^t \phi_{r,s+\tau}A_r = \bigcap_{r=s}^{t-\tau} \phi_{r+\tau,s+\tau}A_{r+\tau} = \bigcap_{r=s}^{t-\tau}\phi_{r,s}A_r\,,
\]
the first equation resulting from a change of variable $r\to r+\tau$, while the second one coming from periodicity of both the vector field and the family of sets. With this we have from~\eqref{eq:escrateineq1} that
\begin{equation}
E(\{A_r\}_{r\ge s_1}) \ge E(\{A_r\}_{r\ge s_2+\tau})=E(\{A_r\}_{r\ge s_2}),\qquad \text{for all }s_2\le s_1\le s_2+\tau\,.
\label{eq:escrateineq2}
\end{equation}
Equations~\eqref{eq:escrateineq1} and~\eqref{eq:escrateineq2} together yield Theorem~\ref{thm:escrate_indep}.

\subsection{Proof of Theorem \ref{thm:escrate_nonauto_aug}} \label{app:augrate}
Recall the identity~\eqref{eq:remainingset}. From this we have
\[
\aug{m}\left(\bigcap_{s=0}^t\aug{\phi}_{-s}\aug{A}\right) = \int_0^\tau m(A_{s,s+t})\,ds\,.
\]
Using Corollary~\ref{cor:measbound} twice gives
\[
\tau J_{\max}^{-1}\, m(A_{0,s+t}) \le \int_0^{\tau} m(A_{s,s+t})\,ds \le \tau J_{\max}\, m(A_{\tau,s+t})
\]
for sufficiently large~$t$. Taking the $-\limsup_{t\to\infty}\tfrac1t\log$ of both sides and using Theorem~\ref{thm:escrate_indep} yields the claim.

%

\subsection{Proof of Theorem \ref{thm:cont_time_escrate}} \label{app:auto_escrate}

Fix $t>0$ and let $(r_i)_{i\in\N}$ be a dense sequence in $[0,t]$ such that $r_1 = t$; this latter condition is needed such that a decomposition as in~\eqref{eq:eventsplit} is always possible. Further, define the signed measure $\nu$ via $d\nu(x) = f(x)\,dm(x)$.

We consider the events
\[
\mathcal{E}_n := \left\{\omega\,\vert\, x_{r_i}(\omega)\in A^+,\ \forall\,i=1,\ldots,n\right\}\,.
\]
As in the proof of Proposition~\ref{prop:nice=measurable}, the continuity of sample paths and the closedness of~$A^+$ implies $\mathcal{E}_n\downarrow\mathcal{E}:=\bigcap_{r\in[0,t]} \{\omega\,\vert\, x_r(\omega)\in A^+\}$, hence the measurability of $\mathcal{E}$. The very same proof yields also $\prob_{\nu}(\mathcal{E}_n) \to \prob_{\nu}(\mathcal{E})$ as $n\to\infty$; where $\prob_{\nu} = \prob_{\nu^+} - \prob_{\nu^-}$.

Note that since~$\lambda<0$, one has $\int fdm = 0$, hence we can take~$f$ such that $\int f^+dm = \int_{A^+}f dm= 1$. It follows that for every $n\in\N$ we have~\cite{FrJuKo13}
\begin{eqnarray}
e^{\lambda t} & = & e^{\lambda t}\int_{A^+} f dm\nonumber\\
& = & \int_{A^+} \P_t f \nonumber\\
& = & \prob_{\nu}(x_t\in A^+) \nonumber\\
& = & \prob_{\nu}(x_{r_i}\in A^+,\ i=1,\ldots,n) + \sum_{j=2}^n \underbrace{\prob_{\nu}(x_{r_j}\notin A^+,\ x_{r_i}\in A^+,\ \forall\ r_i>r_j)}_{=: p_j}. \label{eq:eventsplit}
\end{eqnarray}
The last equality follows from the decomposition of the event~$\{x_{r_1} \in A^+\} = \{x_t \in A^+\}$ into disjoint events $\{x_{r_i}\in A^+ \text{ for all }r_i>r_j, \text{ but }x_{r_j}\notin A^+\}$, $j=2,\ldots,n$, and $\{x_{r_i}\in A^+ \text{ for all }i=1,\ldots,n\}$. One can show~\cite{FrJuKo13} that $p_j\le 0$, because the set of initial conditions $x_{r_j}\notin A^+$ is contained in the non-positive support of~$\nu$. It follows that
\begin{equation}
e^{\lambda t}\le \prob_{\nu}(x_{r_i}\in A^+,\ i=1,\ldots,n) = \prob_{\nu}(\mathcal{E}_n).
\label{eq:eventestim}
\end{equation}
Thus,
\[
e^{\lambda t} \le \lim_{n\to\infty} \prob_{\nu}(\mathcal{E}_n) = \prob_{\nu}(\mathcal{E})\,.
\]
From now on, the proof follows the lines of~\cite{FrSt10,FrSt13}. Noting that $\prob_{\nu}(\cdot)\le \|f\|_{L^{\infty}}\prob_{m}(\cdot)$, we obtain for every $t>0$ that
\[
\frac1t \log e^{\lambda t} \le \frac1t \log \prob_{\nu}(\mathcal{E}) \le \frac1t \left(\log \prob_m(\mathcal{E}) + \log\|f\|_{L^{\infty}}\right)\,.
\]
Taking the $-\limsup_{t\to\infty}$ of both sides we conclude the proof.

\subsection{Proof of Theorem~\ref{thm:nonauto_escrate}} \label{app:nonauto_escrate}

The proof follows the lines of that of Theorem~\ref{thm:cont_time_escrate} (Appendix~\ref{app:auto_escrate}). For a fixed $t>s$ and a sequence $(r_i)_{i\in\N}$, dense in $[s,t]$ and satisfying $r_1=t$, we consider the events
\[
\mathcal{E}_n:=\big\{\omega\,\vert\, x_{r_i}(\omega)\in A_{r_i}^+,\ \forall i=1,\ldots,n \big\}\,.
\]
The sufficiently niceness of the family~$\{A_r\}$ means by Definition~\ref{def:suff_niceness}~(ii), in particular, that the~$A_r^+$ are closed and either left- of right-continuous in~$r$, thus one can show that~$\mathcal{E}_n\downarrow \mathcal{E}:= \big\{\omega\,\vert\, x_r(\omega)\in A_r^+,\ \forall r\in [s,t]\big\}$.

Since $\P_{s,t}$ is a Markov operator (positive and integral preserving), we have ${\int\P_{s,t}f = 0}$ due to $\int f=0$, and hence $\tfrac12\|\P_{s,t}f\|_1 = \int_{A_t^+}\P_{s,t}f$. Let $\nu$ be the signed measure with $d\nu = f\,dm$, and for the ease of notation denote $\prob_{x_s\sim \nu}(\cdot)$ by $\prob_{\nu}(\cdot)$. Then we obtain in an analogous manner as above, that
\begin{equation}
\frac12 \|\P_{s,t}f\|_1 = \int_{A_t^+}\P_{s,t}f = \prob_{\nu}\big(x_t\in A_t^+\big) \le \prob_{\nu}(\mathcal{E}_n) \stackrel{n\to\infty}{\longrightarrow} \prob_{\nu}(\mathcal{E})
\end{equation}
As $f\in L^{\infty}$, it follows by the monotonicity of the logarithm that
\begin{eqnarray*}
\Lambda_s(f) & = & \limsup_{t\to\infty} \frac{1}{t-s}\left(\log \|P_{s,t}f\|_1 + \log\tfrac12\right) \\
& \le & \limsup_{t\to\infty} \frac{1}{t-s}\left(\log\prob_m(\mathcal{E}) + \log\|f\|_{\infty}\right) \\
& = & -E(\{A_r\})\,,
\end{eqnarray*}
completing the proof of $E(\{A_r\}) \le -\Lambda_s(f)$.

\subsection{Proof of Proposition \ref{prop:CohPoiB}} \label{app:decaybound1}

We start with proving the independence on the starting time.
\begin{lemma}	\label{lem:cycInterchange}
If $\lambda\in\sigma(\P_{s,s+\tau})$ for one $s\in \tau S^1$, then it holds for every~$s\in \tau S^1$.
\end{lemma}
\begin{proof}
Let $s_1,s_2\in \tau S^1$. Without loss $s_2>s_1$, as elements of~$[0,\tau)$. Since $\P_{s_1,s_1+\tau} = \P_{s_2,\tau+s_1}\P_{s_1,s_2}$, we have by the cyclical interchangeability of factors in the operator product without changing the spectrum of the product, that
\begin{eqnarray*}
\sigma(\P_{s_1,s_1+\tau}) & = & \sigma(\P_{s_2,s_1+\tau} \P_{s_1,s_2}) \\
			  & = & \sigma(\P_{s_1+\tau,s_2+\tau} \P_{s_2,s_1+\tau}) \\
			  & = & \sigma(\P_{s_2,s_2+\tau})\,.
\end{eqnarray*}
\end{proof}

Now we continue proving Proposition~\ref{prop:CohPoiB}. For simplicity, let us write $\lambda_2=\lambda_2(\P_{s,s+\tau})$. We show the first equality in~\eqref{eq:CohPoiB} by showing first~``$\ge$'', then~``$\le$''. One has (because the maximum is always taken with respect to functions satisfying ${\int f=0}$)

\begin{eqnarray}
\max_{f} \limsup_{t\to\infty}\frac{1}{t-s} \log \|P_{s,t}f\|_1 & \ge & \max_{f}\limsup_{k\to\infty}\frac{1}{k\tau}\log\left\|\P_{s,s+k\tau} f\right\|_1 \nonumber\\
& = & \max_{f}\limsup_{k\to\infty}\frac{1}{k\tau}\log\left\|\left[\P_{s,s+\tau}\right]^kf\right\|_1 \nonumber\\
	& = & \frac{1}{\tau} \log|\lambda_2| \label{eq:lambda2rate}
\end{eqnarray}
where the first inequality follows from restricting the~$\limsup$ to the times $\{k\tau\}_{k\in\N}$, and the next equation comes from the $\tau$-periodicity of the vector field. The final equality can be obtained by noting that~$\P_{s,s+\tau}$ is a compact operator (its spectrum has no accumulation points apart from $0$), all eigenfunctions~$g$ at eigenvalues~$\lambda\neq1$ satisfy $g\perp\mathbbm{1}$ (i.e.\ $\int g=0$), and using the Riesz (spectral) decomposition at the eigenvalue~$\lambda_2$, which yields that the expression in the norm grows at most as $|\lambda_2|^k$.

To show the other inequality, let~$\lfloor x\rfloor$ denote the integer part of~$x\in\R$, and define~$k(t):= \lfloor \tfrac{t-s}{\tau} \rfloor$. Then, by splitting the time interval $[s,t]$ into $[s,t-k(t)\tau]$ and $[t-k(t)\tau,t]$, we have
\begin{equation}
\|\P_{s,t}f\|_1 = \|\P_{t-k(t)\tau,t}\P_{s,t-k(t)\tau}f\|_1 \le \|\P_{t-k(t)\tau,t}\|_{1,\mathbbm{1}^{\perp}}\|\P_{s,t-k(t)\tau}f\|_1\,,
\label{eq:Psplit}
\end{equation}
where $\|\cdot\|_{1,\mathbbm{1}^{\perp}}$ denotes the induced $L^1$ operator norm on the (invariant) subspace $\{f\in L^1\,\vert\, \int f =0\}$. The invariance follows from~$\int\P_{s,r}f = 0$ for every~$r\ge s$ and~$\int f=0$.
Thus, \eqref{eq:Psplit} yields
\begin{multline}
\max_{f} \limsup_{t\to\infty} \frac{1}{t-s}\log\|\P_{s,t}f\|_1 \\
\hspace*{2em} \le \max_{f} \limsup_{t\to\infty} \frac{1}{t-s}\left(\log\|\P_{t-k(t)\tau,t}\|_{1,\mathbbm{1}^{\perp}} + \log \|\P_{s,t-k(t)\tau}f\|_1\right) \hfill \\
\hspace*{2em} \le \limsup_{t\to\infty} \frac{1}{t-s}\log\|\P_{t-k(t)\tau,t}\|_{1,\mathbbm{1}^{\perp}} + \max_{f} \limsup_{t\to\infty} \frac{1}{t-s} \log \|\P_{s,t-k(t)\tau}f\|_1 \hfill
\label{eq:ratesplit}
\end{multline}
Now, due to periodicity of the forcing,
\[
\P_{t-k(t)\tau,t} = \left[\P_{t-k(t)\tau,t-(k(t)-1)\tau}\right]^{k(t)},
\]
and as in~\eqref{eq:lambda2rate}, we have that the first term in~\eqref{eq:ratesplit} is equal to~$\frac{1}{\tau} \log|\lambda_2|$. For the second term, note that $t\mapsto \|\P_{s,t-k(t)\tau}f\|_1$ is a periodic function bounded from above by $\|f\|_1$ (this latter property being a consequence of the fact that~$\P_{s,r}$ is a Markov operator). Consequently, $\limsup_{t\to\infty} \frac{1}{t-s} \log \|\P_{s,t-k(t)\tau}f\|_1=0$ for every~$f$, and so is the right term in~\eqref{eq:ratesplit}. This concludes the proof of Proposition~\ref{prop:CohPoiB}.

\subsection{Proof of Lemma~\ref{lem:spectral_con} and Theorem~\ref{thm:IGcohPeriodic}}
\label{app:auggenproofs}

We start with a technical lemma collecting some properties of the transfer operator family. Extensive work has been done here~\cite{Kato61,DaPGr79,Paz83,AcqTe85,AcqTe87,Acq88,But92, Ama95,Lun95,Pruss2001}, unfortunately results on the~$L^1$-case in particular with explicit application to partial differential equations seem to be scarce. Here, we apply the results of~\cite{Tan96} and of~\cite{Ama83} to our specific setting.

\begin{lemma} \label{lem:analytic evolution}
Let Assumption~\ref{ass:regularity} hold, then the following are true.
\begin{enumerate}[(a)]

\item Let~$t_*>s$ be arbitrary. The abstract (nonautonomous) evolution equation
\begin{equation}
\partial_t u_t = \IG_t u_t,\quad u_s = f,
\label{eq:evo_eq}
\end{equation}
\begin{enumerate}[(i)]
\item has a unique solution~$u\in C\left(\left[s,t_*\right];L^1\right)\cap C^1\left(\left(s,t_*\right];L^1\right)$, such that~$u_t\in D(\IG_t)$ for~$s<t\le t_*$. It is given by~$u_t = \P_{s,t}f$ on~$L^1$,
\end{enumerate}
where the transfer operator family~$\P_{s,t}:L^1\to L^1$ satisfies
\begin{enumerate}[(i)]
\item[(ii)] $\P_{s,r}\P_{r,t} = \P_{s,t}$ for all $s\le r\le t$;
\item[(iii)] there is a~$C>0$ such that~$\|\partial_t \P_{s,t}\|_{L^1} = \|\IG_t\P_{s,t}\|_{L^1} \le \frac{C}{t-s}$ for all~$s<t$.
\end{enumerate}
\item The operator~$\P_{s,t}:L^1\to L^1$ is compact for every~$s<t$.
\end{enumerate}
\end{lemma}

\begin{proof}
(a). The claims follow from~\cite[Section~6.13]{Tan96}, where the abstract evolution equation~\eqref{eq:evo_eq} in~$L^1$, corresponding to the partial differential equation (the Fokker--Planck equation~\eqref{eq:FPeq}), is shown to satisfy the conditions (called assumptions (P1), (P2), and (P4) therein) that allow the application of the abstract results of Acquistapace and Terreni~\cite{AcqTe86,AcqTe87}. A solution satisfying~(i) is called a \emph{classical solution} therein~\cite[Definition~6.1]{Tan96}, its uniqueness and representation by the transfer operator family (called \emph{fundamental solution} therein) are shown in~\cite[Section~6.12, Theorem~6.6]{Tan96}. Property~(ii) is a fundamental property of solutions (following from uniqueness), and~(iii) is given in~\cite[Section~6.10]{Tan96}.

Next we have to show that under Assumption~\ref{ass:regularity} the conditions of~\cite{Tan96} are satisfied for our Fokker--Planck equation~\eqref{eq:FPeq}. To this end, we need to define the boundary operator~$B(t,x) \equiv B(x): \xi \mapsto \langle n(x),\xi\rangle$ for~$\xi\in\R^d$, where~$n(x)$ is the unit outward normal at~$x\in\partial X$. We define the Fokker--Planck differential operator~$\mathcal{L}(t,x,\partial):= \frac{\ep^2}{2}\sum_{i=1}^d \partial_i^2 - \sum_{i=1}^d (\partial_i v_i(t,x)) - \sum_{i=1}^d v_i(t,x)\partial_i$, as the right-hand side of the Fokker--Planck equation~\eqref{eq:FPeq}.
Tanabe~\cite[Section~6.13]{Tan96} requires that
\begin{itemize}
\item $X$ is a bounded open domain of class~$C^4$ (the order~$m$ of our differential operator~$\mathcal{L}$ is~$2$, and the domain has to be of class~$C^{2m}$).
\item The coefficients of the differential operator~$\mathcal{L}$ and boundary operator~$B$ should be H\"older continuous (of some positive order) in~$t$. The same should hold for the coefficients of the adjoint boundary value problem~$(\mathcal{L}',B')$.
\end{itemize}
It is immediate that Assumption~\ref{ass:regularity} fulfills these conditions (note that the adjoint boundary conditions to the Neumann boundary conditions considered here,~$B(\partial)u = 0$, are the very same conditions, i.e.~$B' = B$). Further, it is required that the hypotheses of~\cite[Section~5.3]{Tan96} are satisfied by the operators~$B,\mathcal{L}$ for every~$t\in[s,t_*]$. These are the conditions of~\cite[Theorems~5.5 and~5.6]{Tan96}, i.e.
\begin{itemize}
\item local regularity of class~$C^{2m}$ and uniform regularity of class~$C^m$ of the domain~$X$;
\item conditions (iii) and (iv) in \cite[Section~5.2]{Tan96};
\end{itemize}
and the conditions of~\cite[Section~4.2]{Tan96}:
\begin{itemize}
\item smoothness conditions on the coefficients of the operator~$\mathcal{L}$ (boundedness and uniform continuity in~$x$ on~$\overline{X}$ is sufficient);
\item uniform ellipticity of~$\mathcal{L}$ in~$X$;
\item the so-called root condition for~$\mathcal{L}$;
\item smoothness of the coefficients of the boundary operator (for us, it means differentiability with bounded and uniformly continuous derivatives, which is given, since the regularity class of the domain implies the smoothness of the outward normal~$n(x)$ in~$x$); and
\item the so-called complementing condition for~$\mathcal{L}$ and~$B$.
\end{itemize}
All of these are straightforward to check. The root condition is satisfied by strongly elliptic operators~$\mathcal{L}$, cf.~\cite[Theorem~5.4]{Tan96}.
\\
(b). By~(iii) of part~(a), we have that
\begin{equation}
\|\IG_t\P_{s,t}\|_{L^1} \le \frac{C}{t-s}
\label{eq:G-bound}
\end{equation}
for~$t>s$. Since~$\IG_t$ is the~$L^1$-realization of the differential operator~$\mathcal{L}(t,\cdot,\cdot)$, the inequality~\eqref{eq:G-bound} imposes a certain smoothening property on~$\P_{s,t}$. We will exploit this smoothness to show compactness of~$\P_{s,t}$.

First note, that the regularity assumptions by Amann on the problem (the coefficients of the differential operator and the boundary parametrization), given in~\cite[pp.~225 and pp.~235]{Ama83}, are implied by our Assumption~\ref{ass:regularity}. It is shown in~\cite[Proposition~9.2]{Ama83}, that~$D(\IG_t)$ (imposed with the graph norm) is continuously embedded in the Sobolev space~$W_1^1$ (weakly differentiable~$L^1$ functions), i.e.\ there is a constant~$\tilde{C}>0$ such that for every~$u\in D(\IG_t)$ one has
\begin{equation}
\|u\|_{W_1^1} \le \tilde{C}\left(\|\IG_t u\|_{L^1} + \|u\|_{L^1}\right)\,.
\label{eq:cnt_embed}
\end{equation}
Substituting~\eqref{eq:G-bound} into~\eqref{eq:cnt_embed}, we have for~$u = \P_{s,t}u_s$ that
\[
\|\P_{s,t}u_s\|_{W_1^1} \le \tilde{C}\left(\frac{C}{t-s} + 1\right)\|u_s\|_{L^1}\,.
\]
Thus, by the Rellich--Kondrachov theorem (which implies compact embedding of~$W_1^1$ into~$L^1$) we have that~$\P_{s,t}:L^1\to L^1$ is a compact operator for~$s<t$.

We remark that the compactness of~$\P_{s,t}$ on~$L^r$, $1<r<\infty$, is somewhat easier to obtain, since a bound as in~\eqref{eq:cnt_embed} readily follows from the \emph{Agmon--Douglis--Nirenberg estimate}~\cite[(4.25) on pp.~131]{Tan96}.
\end{proof}

\begin{proof}[Proof of Lemma \ref{lem:spectral_con}]
By the definition of~$\augIG$ we have
\begin{equation}
\partial_r \aug{f}(r,\cdot) = (\IG_r-\mu)f_r\,, \qquad r\in\tau S^1\,.
\label{eq:auggenEV}
\end{equation}
Since~$\P_{s,t}$ is the evolution operator of the parabolic evolution equation~$\partial_r u(r) = \IG_r u(r)$, we see that $e^{-\mu t}\P_{s,s+t} f_s$ solves~\eqref{eq:auggenEV}. 
Uniqueness of the solutions to~\eqref{eq:auggenEV} in~$L^1$ follows from Lemma~\ref{lem:analytic evolution}, hence~$f_{s+t\mod \tau} = e^{-\mu t}\P_{s,s+t}f_s$.
\end{proof}

\begin{proof}[Proof of Theorem~\ref{thm:IGcohPeriodic}]

By Proposition~\ref{prop:CohPoiB} the left hand side of equation~\eqref{eq:lyap auggen} is equal to $\tfrac{1}{\tau} \log |\lambda_2(\P_{s,s+\tau})|$. If~$\augIG$ has an eigenvalue~$\mu$ then~$\P_{s,s+\tau}$ has an eigenvalue~$e^{\mu \tau}$, as shown in Lemma~\ref{lem:spectral_con}. It remains to show, that~$\P_{s,s+\tau}$ having an eigenvalue~$\lambda$ implies that~$\augIG$ has an eigenvalue~$\mu$ such that~$e^{\mu \tau}=\lambda$, with the corresponding eigenfunctions being equal too.

First note, that~$\lambda\in\sigma(\P_{s,s+\tau})$ implies~$\lambda\in\sigma(\P_{s+t,s+t+\tau})$ for every~$t\ge 0$ by Lemma~\ref{lem:cycInterchange}.
Fix~$s\in\tau S^1$, and let~$f_s$ be an eigenfunction of~$\P_{s,s+\tau}$ with eigenvalue~$\lambda\neq 0$. Further, let~$\mu\in\C$ be any complex number such that~$e^{\mu \tau} = \lambda$.

Define~$f_{s+t} = e^{-\mu t}\P_{s,s+t} f_s$ for~$t\ge 0$. It is easy to see that the consistency properties hold, i.e.\ $f_{s+\tau} = f_s$, and~$f_{s+t}$ is an eigenfunction of~$\P_{s+t,s+t+\tau}$ at eigenvalue~$\lambda$ for any~$t> 0$.

Lemma~\ref{lem:analytic evolution} (b) states that~$f_{s+t}\in D(\IG_{s+t})$ and differentiable with respect to~$t$ for every~$t>0$. Now, differentiate~$f_{s+t}$ with respect to~$t$ to obtain
\begin{eqnarray*}
\partial_t f_{s+t} & = & -\mu e^{-\mu t}\P_{s,s+t}f_s + e^{-\mu t} \IG_{s+t}\P_{s,s+t} f_s \\
& = & -\mu f_{s+t} + \IG_{s+t} f_{s+t}
\end{eqnarray*}
Rearranging terms yields that~$\aug{f}$, defined by $\aug{f}(r,\cdot) = f_r$, is an eigenfunction of~$\augIG$ at the eigenvalue~$\mu$. The construction shows that~$\mathrm{Re}(\mu) = \tfrac{1}{\tau}\log|\lambda|$.

\end{proof}

\end{appendix}

\bibliographystyle{alpha}
\bibliography{References}

\end{document}